\documentclass[english,11pt, a4paper]{article}

\title{Fractional Besov-Sobolev Spaces on Chord-arc Curves}

\usepackage[T1]{fontenc}

\usepackage[hypertexnames=false]{hyperref}
\usepackage[utf8]{inputenc}

\usepackage{mathtools}
\usepackage{lmodern}
\usepackage{amsthm,amsmath,amsfonts,amssymb}
\usepackage{graphicx}
\usepackage{float}
\usepackage{xcolor}
\DeclareRobustCommand{\revision}[1]{#1}

\usepackage{tikz}
\usetikzlibrary{arrows.meta,calc}
\usepackage{enumerate,enumitem}
\usepackage{authblk}
\usepackage{cite}
\usepackage[title,titletoc]{appendix}

\let\OLDthebibliography\thebibliography
\renewcommand\thebibliography[1]{
  \OLDthebibliography{#1}
  \setlength{\parskip}{0pt}
  \setlength{\itemsep}{0pt plus 0.3ex}
}

\usepackage[activate={true,nocompatibility},final,tracking=true,kerning=true,spacing=false,factor=1100,stretch=15,shrink=15]{microtype}
\ifXeTeX
  \microtypesetup{expansion=false,tracking=false,kerning=false}
\fi

\allowdisplaybreaks

\makeatletter
\numberwithin{equation}{section}
\theoremstyle{plain}

\newtheorem{theo}{\protect\theoremname}
  \theoremstyle{plain}
  \newtheorem{lemma}[theo]{\protect\lemmaname}
    \newtheorem{prop}[theo]{\protect\propname}

\numberwithin{theo}{section}

\newtheorem{cor}[theo]{Corollary}

\newtheorem{rem}[theo]{Remark}

\makeatother

\usepackage{babel}
\providecommand{\propname}{Proposition}

\providecommand{\lemmaname}{Lemma}
\providecommand{\theoremname}{Theorem}

\usepackage{hyperref}
\hypersetup{
    colorlinks=true,
    allcolors=black, 
    linktoc=all 
}

 \author{
 Huaying Wei\thanks{Center for applied Mathematics, Tianjin University, Weijin Road 92, Tianjin, 300072, P.R. China, email: \url{hywei@tju.edu.cn} } \, and Michel Zinsmeister\thanks{Institut Denis Poisson,  Orl\'eans, 45067,  France, email: \url{zins@univ-orleans.fr}}}
 \date{}

\begin{document}

\maketitle

\begin{abstract}
Let $\Gamma$ be a bounded rectifiable Jordan curve, with
complementary domains $\Omega_i$ and $\Omega_e$. For
$1<p<\infty$ and $0<s<1$, we study the spaces
$\mathcal B_{p,p}^s(\Omega)$ of harmonic functions satisfying
\[
  \iint_\Omega |\nabla u(z)|^p
  \operatorname{dist}(z,\Gamma)^{(1-s)p-1}\,dA(z)<\infty,
  \qquad \Omega\in\{\Omega_i,\Omega_e\}.
\]
When $\Gamma$ is the unit circle, the traces of these two
spaces coincide with the fractional Besov-Sobolev space
$B^s_{p,p}(\Gamma)$. 
 The equivalence of boundary seminorms and harmonic energies on the circle raises a geometric question: which rough curves preserve these estimates? 

For every radial-Lipschitz curve $\Gamma$, we prove
\[
  \operatorname{Tr}_\Gamma\mathcal B_{p,p}^s(\Omega_i)
  =B_{p,p}^s(\Gamma)
  =\operatorname{Tr}_\Gamma\mathcal B_{p,p}^s(\Omega_e)
\]
whenever
\[
  \begin{cases}
    1/p-1/2\le s<1, & 1<p<2,\\
    0<s<1, & p=2,\\
    0<s\le 1/p+1/2, & 2<p<\infty.
  \end{cases}
\]
  For arbitrary
chord-arc curves, the same trace identities hold when
$1/p\le s<1/2+1/(2p)$. Taking $s=1-1/p$ in the radial-Lipschitz range yields an almost-Dirichlet principle for the
unweighted $p$-energy on radial-Lipschitz domains when
$4/3\le p\le4$, and this interval cannot be enlarged for the class of radial-Lipschitz domains. In particular, every Riemann map
$g:\Omega_i\to\mathbb D_i$ from a bounded radial-Lipschitz
domain satisfies $g'\in L^4(\Omega_i)$, and the exponent $4$
is optimal uniformly over this class. This gives a sharp  endpoint counterpart of Brennan’s integrability conclusion for radial-Lipschitz domains.

   

\end{abstract}

\newpage

\tableofcontents

\section*{AI disclosure}
We conjectured Theorem~\ref{thm:reverseinclusion}\textup{(c)}
because we expected the region of validity in the $(1/p,s)$-plane
to be centrally symmetric about $(1/2,1/2)$. GPT's contribution
to the proof was limited to the subrange $1/p-1/2\le s<2/p-1$,
for which it provided the main idea.  We also used GPT to draw the figures, locate
references, and review the full manuscript.

\section{Introduction and statement of the results}
\subsection{Introduction}
Throughout this paper, $\Gamma$ denotes a bounded Jordan curve, and
$\Omega_i$ and $\Omega_e$ denote the complementary interior and exterior 
components of $\Gamma$ in the Riemann sphere, respectively. For any one side $\Omega \in \{\Omega_i, \Omega_e\}$, if $f$ is
continuous on $\Gamma$ and $F$ is a continuous extension of $f$ to
$\overline{\Omega}$ with $F\in C^1(\Omega)$, we define its
$2$-energy by
\[
  \iint_{\Omega}|\nabla F(z)|^2\,dA(z)\in[0,+\infty].
\]
where $dA$ denotes the Euclidean area measure in $\mathbb C$. 
The Dirichlet principle states that the harmonic extension minimizes
this energy among all such extensions of $f$. The energy of the
harmonic extension of a continuous function may be infinite; in that
case, every such extension has infinite energy. When $\Gamma=\mathbb T$
is the unit circle, the Dirichlet space $\mathcal D(\mathbb D_i)$ consists
of the harmonic functions in the unit disk $\mathbb D_i$ with finite
$2$-energy. Equivalently, it consists of the Poisson integrals of
functions $f\in L^2(\mathbb T)$ satisfying
\[
  \sum_{n\in\mathbb Z}|n|\,|\hat f(n)|^2<\infty.
\]
As Douglas observed in \cite{Dou}, this condition is equivalent to
\[
  \iint_{\mathbb T\times\mathbb T}
  \frac{|f(z)-f(\zeta)|^2}{|z-\zeta|^2}\,|dz|\,|d\zeta|<\infty.
\]
The map $\iota(z)=1/\bar z$ is reflection across the unit circle and is
bi-Lipschitz on the annulus $\{1/2<|z|<2\}$. If
$u_i\in\mathcal D(\mathbb D_i)$ is the Poisson integral of
$f\in L^2(\mathbb T)$, then its reflected function
$u_e:=u_i\circ\iota$ is the Poisson extension of $f$ to
$\mathbb D_e=\widehat{\mathbb C}\setminus\overline{\mathbb D_i}$.
Conformal and anticonformal invariance of the $2$-energy show that
$u_e$ also has finite energy when $u_i$ has.

After identification with its boundary trace space, the Dirichlet space
is a special case of the homogeneous fractional Sobolev space. More
precisely, for $0<s<1$, let $H^s(\mathbb T)$ be the subspace of
$L^2(\mathbb T)$ consisting of functions $f$ such that the square of the following seminorm is finite: 
\[
  [f]_{H^s}^2:=\sum_{n\in\mathbb Z}|n|^{2s}|\hat f(n)|^2<\infty.
\]
 It is equivalent
to each of the following quantities:
\begin{itemize}
\item the square of the Douglas seminorm,
\[
  \iint_{\mathbb T\times\mathbb T}
  \frac{|f(z)-f(\zeta)|^2}{|z-\zeta|^{1+2s}}\,|dz|\,|d\zeta|;
\]
\item the square of the Littlewood--Paley seminorm,
\[
  \iint_{\mathbb D_i}|\nabla u_i(z)|^2(1-|z|)^{1-2s}\,dA(z).
\]
\end{itemize}
For $s=1/2$, these quantities agree up to fixed multiplicative constants.
The terminology reflects the connection with Littlewood--Paley
estimates for harmonic extensions; see \cite[Chapter~IV]{Stein1970}.

In Besov notation,
\[
  H^s(\mathbb T)=B_{2,2}^s(\mathbb T).
\]
As a space of functions, it also coincides with
$(I-\Delta)^{-s/2}(L^2(\mathbb T))$, the space of Bessel potentials of
order $s$. The corresponding inhomogeneous norm is obtained by adding
the $L^2$-norm to the homogeneous seminorm.

Replacing $2$ by $p\in(1,\infty)$, we define $B_{p,p}^s(\mathbb T)$
as the space of functions $f\in L^p(\mathbb T)$ satisfying either of
the following equivalent conditions:
\begin{equation*}
    \iint_{\mathbb T\times \mathbb T}\frac{|f(z)-f(\zeta)|^p}{|z-\zeta|^{1+ps}}|dz||d\zeta|<\infty, \leqno (\rm i)
\end{equation*}
or
\begin{equation*}
    \iint_{\mathbb D_i}|\nabla u_i(z)|^p(1-|z|)^{(1-s)p-1} dA(z)<\infty. \leqno (\rm ii)
\end{equation*}
A change of variables under inversion, together with harmonicity
near the origin, shows that the latter integral is comparable to its
exterior counterpart:
\begin{equation*}
    \iint_{\mathbb D_e}|\nabla u_e(z)|^p(|z|-1)^{(1-s)p-1} dA(z) < \infty. \leqno (\rm iii)
\end{equation*}
For $p\ne2$, however,
\[
  B_{p,p}^s(\mathbb T)\ne(I-\Delta)^{-s/2}(L^p(\mathbb T)).
\]
To avoid ambiguity, we therefore use Besov terminology. The spaces
$B_{p,p}^s(\mathbb T)$, with $0<s<1$ and $1<p<\infty$, are intermediate
spaces between Sobolev spaces of orders zero and one, and are also
called fractional Sobolev spaces in the literature. Throughout this
paper, we call them fractional Besov--Sobolev spaces.

The equivalence of boundary seminorms and harmonic energies on the circle raises a geometric question: which rough curves preserve these estimates? We work with quasicircles, the quasiconformal images of circles. These curves share the reflection property of circles; that is bi-Lipschitz in a neighborhood of the boundary \cite{Ahl}. Geometrically, a Jordan
curve $\Gamma$ is a quasicircle if and only if there exists $C>1$ such
that
\begin{equation}\label{def:quasi}
    \min\{\operatorname{diam}(\gamma_1),
         \operatorname{diam}(\gamma_2)\}
  \le C|z_1-z_2|,
  \qquad z_1,z_2\in\Gamma,
\end{equation}
where $\gamma_1$ and $\gamma_2$ are the two subarcs of $\Gamma$ with
endpoints $z_1$ and $z_2$.

Quasicircles need not be smooth, and their Hausdorff dimensions can
take any value in $[1,2)$. For example, the von Koch snowflake is a
quasicircle with no tangent at any point, and its Hausdorff and
Minkowski dimensions are both $\log4/\log3$. On the other hand, the
arclength-based analogue of the Douglas seminorm considered here
requires rectifiability. Unless otherwise stated, we therefore
restrict attention to bounded rectifiable quasicircles.

\subsection{Statement of main theorems}
Throughout the paper, $C$ and $c$ denote positive constants that may
change from one occurrence to the next. For nonnegative quantities
$A$ and $B$, we write $A\lesssim B$ if $A\le CB$, and $A\gtrsim B$ if $B\lesssim A$. We write $A\asymp B$
if both $A\lesssim B$ and $B\lesssim A$ hold. In all such comparisons,
the implicit constants are independent of the functions under
consideration and of the variables in the estimate. 

Let $\Gamma$ be a bounded Jordan curve,  and
write $\delta(z)=\operatorname{dist}(z,\Gamma)$. For a harmonic function
$u$ on one side $\Omega\in\{\Omega_i,\Omega_e\}$, define the weighted
energy, for $1<p<\infty$ and $0<s<1$, by
\[
  \mathcal E_{s,p;\Omega}(u)^p
  :=\iint_\Omega|\nabla u(z)|^p
     \delta(z)^{(1-s)p-1}\,dA(z).
\]
The harmonic weighted energy space is
\[
  \mathcal B_{p,p}^s(\Omega)
  :=\{u:\ u\text{ is harmonic in }\Omega,\quad
          \mathcal E_{s,p;\Omega}(u)<\infty\},
\]
understood modulo constants. Harmonicity in $\Omega_e$ includes
harmonicity at infinity, as explained in Section~\ref{subsec: exte}. Clearly, this space is included in the homogeneous weighted Sobolev space $\dot W^{1,p}(\Omega, \omega)$ where $\omega(z)=\delta(z)^{(1-s)p-1}$. 

A classical theorem of Jones \cite{Jones} says that if $\Omega$ is a  quasidisk; that is,   an $(\varepsilon,\infty)$-domain in the terminology
of the weighted extension theorem, 
then every $u \in \dot W^{1,p}(\Omega)$, the homogeneous unweighted Sobolev space, can be extended to a function in $\dot W^{1,p}(\mathbb C)$. More precisely, there is a continuous linear operator $\mathcal E: \dot W^{1,p}(\Omega)\to\dot W^{1,p}(\mathbb C)$ such that $\mathcal E (u)|_\Omega = u$. This theorem was generalized by Chua \cite{Chua, Chua-2} to the weighted case with weight $\omega$ in Muckenhoupt class $A_p$ on $\mathbb C$ when $\Omega$ is a  quasidisk, which, combined with the assertion that for any compact set $K\subset\mathbb C$ the function $u\in L^p(K,\omega)$ if $u\in\dot W^{1,p}(\mathbb C,\omega)$, implies that the homogeneous and inhomogeneous cases coincide: $\dot W^{1,p}(\Omega,\omega) = W^{1,p}(\Omega,\omega)$ as sets when $\Omega$ is a bounded quasidisk; see \cite{Chua}. 

Let $\omega \in A_p$. By Muckenhoupt's result \cite{Muc72}, the Hardy-Littlewood maximal function is bounded on the weighted $L^p$-space $L^p(\mathbb C, \omega)$. It follows that the classical proof that $C_c^\infty(\mathbb C)$ is dense in $ W^{1,p}(\mathbb C)$ goes through in the weighted setting $ W^{1,p}(\mathbb C, \omega)$. Since $C_c^\infty(\mathbb C)\subset C_c^{\text{Lip}}(\mathbb C)$, the space of Lipschitz functions with compact support, this latter space is dense in $ W^{1,p}(\mathbb C, \omega)$ as well. Since $C_c^{\text{Lip}}(\mathbb C)$ is bi-Lipschitz invariant, we prefer to use functions in it as test functions. 

In this paper, we will be mainly concerned with the case $\omega(z)=\delta(z)^{(1-s)p-1}$ where $\Gamma$ is a bounded chord-arc curve, that is, a quasicircle with Ahlfors $1$-regularity. 
A curve $\Gamma$ is called chord-arc if it is the bi-Lipschitz image
of a circle. Equivalently, it is rectifiable and satisfies
\begin{equation}\label{chordarcK}
  \min\{\operatorname{length}(\gamma_1),
         \operatorname{length}(\gamma_2)\}
  \le K|z_1-z_2|,
  \qquad z_1,z_2\in\Gamma,
\end{equation}
for some $K>1$, where $\gamma_1$ and $\gamma_2$ are the two subarcs
with endpoints $z_1$ and $z_2$. In this case, we also call $\Gamma$
a $K$-chord-arc curve. It can be seen that the chord-arc property implies $\omega \in A_p$ on $\mathbb C$.


If $\Gamma$ is rectifiable, we define, following Douglas \cite{Dou}, the Besov space $B^s_{p,p}(\Gamma)$ as the set of functions 
 $f\in L^p(\Gamma)$ for which the seminorm 
\[
  [f]_{\dot B^s_{p,p}(\Gamma)}
  :=\left(
    \int_{\Gamma}\int_{\Gamma}
    \frac{|f(z_1)-f(z_2)|^p}{|z_1-z_2|^{1+sp}}
    |dz_1|\,|dz_2|
  \right)^{1/p}<\infty.
\]

Let $\Gamma$ be a bounded chord-arc curve.  The space $B_{p,p}^s(\Gamma)$ shares with $C_c^{\text{Lip}}(\mathbb C)$ the property of being bi-Lipschitz invariant. Let $u\in  W^{1,p}(\mathbb C, \omega)$. We have seen that there exists a sequence $(\varphi_j)\subset  C_c^{\text{Lip}}(\mathbb C)$ such that $\varphi_j\to u$ in $ W^{1,p}(\mathbb C, \omega)$. Clearly, the trace $\varphi|_\Gamma$ makes sense for any $\varphi\in C_c^{\text{Lip}}(\mathbb C)$, and moreover, 
$$
\Vert\varphi|_\Gamma\Vert_{ B^s_{p,p}(\Gamma)}\lesssim\Vert\varphi\Vert_{ W^{1,p}(\mathbb C, \omega)}.
$$
Then the sequence $\varphi_j|_\Gamma$ converges to some $f \in B^s_{p,p}(\Gamma)$ and the limit is independent of the approximation. This function is called the trace of $u$ on $\Gamma$.  
Conversely, every $f\in B^s_{p,p}(\Gamma)$ is the trace of some $u\in  W^{1,p}(\mathbb C, \omega)$ in this sense satisfying
\[
  \Vert u\Vert_{ W^{1,p}(\mathbb C,\omega)} \lesssim \Vert f\Vert_{ B^s_{p,p}(\Gamma)}.
\]
Consequently,  
\begin{equation}\label{eq:weighted-trace}  
\operatorname{Tr}_\Gamma  W^{1,p}(\mathbb C,\omega)   
=B^s_{p,p}(\Gamma)
\end{equation}
as sets.
Indeed, since two spaces $C_c^{\text{Lip}}(\mathbb C)$ and $B^s_{p,p}(\Gamma)$ have the bi-Lipschitz invariance, and the chord-arc curve $\Gamma$ is the image of $\mathbb T$ under a bi-Lipschitz homeomorphism of the plane,  it  suffices to establish these assertions for $\Gamma=\mathbb T$. This circle case follows from the classical weighted trace and extension theorem; see \cite[Theorem~2.9.2/1]{Triebel78}. Finally, localization after subtracting a local average, together with the weighted Poincar\'e inequality, yields the homogeneous counterpart of \eqref{eq:weighted-trace}.

Let us turn to spaces $\mathcal B^s_{p,p}(\Omega)$ with $\Gamma:=\partial \Omega$ being chord-arc. As was already mentioned, $\mathcal B^s_{p,p}(\Omega)\subset \dot W^{1,p}(\Omega,\omega)$,
and every function in the latter space admits an extension
to $\dot W^{1,p}(\mathbb C,\omega)$.
Such an extension has a well-defined trace on $\Gamma$,
defined by completion.
To define the trace of $u\in\mathcal B^s_{p,p}(\Omega)$, 
we need to prove that the trace does not depend on the extension.  
The proof  will be given  in the appendix. Consequently, the trace operator is well-defined on
$\mathcal B^s_{p,p}(\Omega)$, and
\[
\operatorname{Tr}_\Gamma\mathcal B^s_{p,p}(\Omega)
\subset B^s_{p,p}(\Gamma).
\]

When $\Gamma$ is the unit circle, the three spaces
$\operatorname{Tr}_\Gamma\mathcal B_{p,p}^s(\Omega_i)$,
$\operatorname{Tr}_\Gamma\mathcal B_{p,p}^s(\Omega_e)$, and
$B_{p,p}^s(\Gamma)$ coincide (see p.~151 in \cite{Stein1970}). In the critical case $sp=1$, our earlier work \cite{WZ1,WZ2} characterizes chord-arc curves through equivalence of the interior, exterior and arclength boundary seminorms. We address the full fractional scale $0<s<1$, $1<p<\infty$. For each fixed pair $(s,p)$, the trace estimate still characterizes the chord-arc condition, while a single curve that is not chord-arc satisfies the reverse estimate for every pair. Away from the critical line, chord-arc geometry alone need not give two-sided estimates throughout this scale. We identify parameter ranges and conditions on conformal derivatives under which the interior, exterior and boundary spaces coincide, and construct examples showing the limitations of these conclusions.

We can now state the main results. Let $\varphi_i$ and $\varphi_e$
be conformal maps from $\mathbb D_i$ and $\mathbb D_e$ onto
$\Omega_i$ and $\Omega_e$, respectively, with
$\varphi_e(\infty)=\infty$. When treating both cases simultaneously,
we write $(\Omega,\mathbb D,\varphi)$ for either
$(\Omega_i,\mathbb D_i,\varphi_i)$ or
$(\Omega_e,\mathbb D_e,\varphi_e)$.

\begin{theo}\label{thm:inclusion}
Let $\Gamma$ be a rectifiable quasicircle in $\mathbb C$.
For $0<s<1$ and $1<p<\infty$, the following statements hold.
\begin{itemize}
\item[\rm (1)] If $\Gamma$ is chord-arc, then
$[f]_{\dot B_{p,p}^s(\Gamma)}\lesssim\mathcal E_{s,p;\Omega}(u)$
for every $u\in\mathcal B_{p,p}^s(\Omega)$ with
$f:=\operatorname{Tr}_\Gamma u$. Consequently,
$\operatorname{Tr}_\Gamma\mathcal B_{p,p}^s(\Omega)
\subset B_{p,p}^s(\Gamma)$.
\item[\rm (2)] Conversely, if
$[f]_{\dot B_{p,p}^s(\Gamma)}
\lesssim\mathcal E_{s,p;\Omega}(\mathcal P_\Omega f)$
for every $f\in C^1(\mathbb C)|_\Gamma$, where
$\mathcal P_\Omega f:= \mathcal P_\mathbb D(f\circ\varphi)\circ\varphi^{-1}$ is its Poisson extension, then
$\Gamma$ is chord-arc.
\end{itemize}
\end{theo}

However, reversing the
inequality in part~\textup{(2)} does not imply the chord-arc
condition. More precisely, there exists a  non-chord-arc rectifiable quasicircle for which
the reverse inequality holds on the interior domain for every
$1<p<\infty$ and $0<s<1$.

\begin{theo}\label{thm:reverse-not-chordarc}
There exists a bounded rectifiable quasicircle $\Gamma$ that is not chord-arc such that, for every $1<p<\infty$ and $0<s<1$,
\begin{equation}\label{eq:nonchordarc-reverse-intro}
  \mathcal E_{s,p;\Omega_i}(\mathcal P_{\Omega_i}f)
  \lesssim [f]_{\dot B^s_{p,p}(\Gamma)}                
\end{equation}
holds for any $ f\in B^s_{p,p}(\Gamma)$.      
\end{theo}

Let $\Gamma$ be a rectifiable Jordan curve and  $z_0\in\Omega$.
Suppose that $f\in L^p(\Gamma)$, $1<p<\infty$, is real-valued
and that its Poisson extension $u=\mathcal P_\Omega f$ is
well-defined. There is a unique harmonic conjugation $\tilde u$ of
$u$ in $\Omega$ normalized by $\tilde u(z_0)=0$; thus
$u+i\tilde u$ is holomorphic in $\Omega$. Harmonic conjugation is
extended to complex-valued functions by complex linearity.
For real-valued $u$, the Cauchy--Riemann
equations give $|\nabla\tilde u|=|\nabla u|$, so the harmonic
energy space $\mathcal B_{p,p}^s(\Omega)$ is invariant under
harmonic conjugation. 

Denote the boundary function of $\tilde u$ by $\tilde f$, and write
$K_\Gamma f:=\tilde f$ for the corresponding conjugation operator.
It is known (see \cite{HMW}, p.~247 in \cite{Gar}) that $K_\Gamma$
is bounded on $L^p(\Gamma)$, $1<p<\infty$, if and only if
$|\varphi'|$ belongs to  the Muckenhoupt class
$A_p$, $1<p<\infty$; see Section \ref{sec:Sobolev} for its definition.

Motivated by this $L^p$ criterion, we establish the following
sufficient condition for invariance under conjugation.

\begin{prop}\label{conj}
Let $1<p < \infty$,  $0<s<1$ and $p'=p/(p-1)$. Suppose that $\Gamma$ is a chord-arc curve and $|\varphi'|\in A_p\cap A_{p'}$ on $\mathbb T$. Then $B_{p,p}^s(\Gamma)$
is invariant under harmonic conjugation.
\end{prop}

This result leads to the reverse inclusion: 
\begin{theo}\label{thm:reverseinclusion}
Let $\Gamma$ be a chord-arc curve. Suppose that one of the following three conditions holds: 
\begin{itemize}
\item[\rm (a)] $0<s<1$, $1<p < \infty$, and $|\varphi'|\in A_p\cap A_{p'}$ on $\mathbb T$; 
\item[\rm (b)] $0<s\le\frac1p + \frac12$, $p>2$, and
               $|\varphi'|\in A_2$ on $\mathbb T$;
\item[\rm (c)] $\frac1p-\frac12 \le s <1$, $1<p < 2$, and $|\varphi'|\in A_2$ on $\mathbb T$.
\end{itemize}
Then  every $f\in B_{p,p}^s(\Gamma)$ has the Poisson extension satisfying 
\[
  \mathcal E_{s,p;\Omega}(\mathcal P_\Omega f)
  \lesssim[f]_{\dot B_{p,p}^s(\Gamma)},\qquad \operatorname{Tr}_\Gamma\mathcal P_\Omega f=f.
\]
 Consequently,
\[
  B_{p,p}^s(\Gamma)\subset
  \operatorname{Tr}_\Gamma\mathcal B_{p,p}^s(\Omega).
\]
\end{theo} 
For $p\neq 2$, the assumption on $|\varphi'|$ in (a) is stronger than the $A_2$ condition
in parts~\textup{(b), (c)} since $A_{p_1}\subset A_{p_2}$ if $1<p_1<p_2<\infty$, but allows the full range $0<s<1$.

For each fixed $1<p<\infty$, there exists a chord-arc curve
$\Gamma$ for which $|\varphi'|\notin A_p$; see \cite{JoZi}; whereas every radial-Lipschitz
curve satisfies $|\varphi_i'|,|\varphi_e'|\in A_2$ on $\mathbb T$,
as shown in Section~\ref{subsec:radial-A2}, and the quadrilaterals in
Proposition~\ref{prop:radial-A2-sharp} show that no fixed
$p_0\in(1,2)$ works for all radial-Lipschitz curves. 
We say a curve is a radial-Lipschitz curve if it has
a parametrization
\[
  z(t)=z_0+\rho(t)e^{it},
  \qquad t\in\mathbb R,
\]
for some $z_0\in\mathbb C$, where
$\rho:\mathbb R\to(0,\infty)$ is $2\pi$-periodic and Lipschitz
continuous. Every radial-Lipschitz curve is chord-arc. 

Combining the two-sided $A_2$ condition with
Theorems~\ref{thm:inclusion} and~\ref{thm:reverseinclusion}
yields the following.

\begin{cor}\label{thm:radial}
Let $\Gamma$ be a radial-Lipschitz curve. If one of the following three conditions holds: 
\begin{itemize} 
\item[\rm (i)] $0<s\le\frac1p+\frac12$, $p>2$;
\item[\rm (ii)] $0<s < 1$, $p=2$;
\item[\rm (iii)] $\frac1p-\frac12\le s < 1$, $1<p < 2$,
\end{itemize}
then every $f\in B_{p,p}^s(\Gamma)$ has the Poisson extension satisfying 
\begin{equation}\label{three-1}
  \mathcal E_{s,p;\Omega_i}(\mathcal P_{\Omega_i}f)
  \asymp[f]_{\dot B_{p,p}^s(\Gamma)}
  \asymp\mathcal E_{s,p;\Omega_e}(\mathcal P_{\Omega_e}f), \qquad \operatorname{Tr}_\Gamma\mathcal P_{\Omega_{i,e}} f=f.
\end{equation}
Consequently, 
\begin{equation}\label{three-2}
  \operatorname{Tr}_\Gamma\mathcal B_{p,p}^s(\Omega_i)
  =B_{p,p}^s(\Gamma)
  =\operatorname{Tr}_\Gamma\mathcal B_{p,p}^s(\Omega_e).
\end{equation}
\end{cor}

\begingroup
\setlength{\intextsep}{6pt}
\begin{figure}[H]
\centering
\begin{tikzpicture}[x=3.15cm,y=3.15cm,font=\footnotesize,
  axis/.style={-{Stealth[length=1.5mm]},thin},
  included/.style={line width=.85pt},
  endpoint/.style={circle,draw=black,fill=white,inner sep=0pt,
    minimum size=3.5pt,line width=.5pt}]
        \fill[black!20] (0,0)--(.5,0)--(1,.5)--(1,1)--(.5,1)--(0,.5)--cycle;
        \draw[included] (0,.5)--(.5,1) (.5,0)--(1,.5);
        \node at (.62,1.12) {Radial-Lipschitz curves};
      \draw[black!65,densely dotted] (0,1)--(1,1)--(1,0);
      \draw[axis] (0,0)--(1.14,0) node[right] {$1/p$};
      \draw[axis] (0,0)--(0,1.13) node[left] {$s$};
      \node[below left] at (0,0) {$0$};
      \foreach \tick/\value in {.5/{\tfrac12},1/1}{
        \draw (\tick,0)--++(0,-.025) node[below] {$\value$};
        \draw (0,\tick)--++(-.025,0) node[left] {$\value$};
      }
      \foreach \point in {(0,0),(0,.5),(1,1)}
        \node[endpoint] at \point {};
        \foreach \point in {(.5,0),(1,.5),(.5,1)}
          \node[endpoint] at \point {};
\end{tikzpicture}
\caption{Range for \eqref{three-1} on radial-Lipschitz curves: \(0<x:=1/p<1\), \(0<s<1\), and
\(\lvert s-x\rvert\le\tfrac12\). Both sloping edges are included.}
\label{fig:trace-parameter-ranges}
\end{figure}
\endgroup

The bounds on \(s\) in Theorem~\ref{thm:reverseinclusion}\textup{(b),(c)}
are sharp, even for radial-Lipschitz curves. 
A reentrant corner gives both examples.
\begin{theo}\label{thm:counter}
\begin{enumerate}[label=\textup{(\roman*)}]
\item (Upper bound.) For every $p>2$ and $1/p+1/2<s<1$, there exists a
radial-Lipschitz domain $\Omega_i$ and a function
$f\in B_{p,p}^s(\Gamma)$ such that
\[
  \mathcal E_{s,p;\Omega_i}(\mathcal P_{\Omega_i}f)=\infty.
\]
\item \begingroup
(Lower bound.) For every \(1<p<2\) and \(0<s<1/p-1/2\), there exists a radial-Lipschitz domain \(\Omega_i\) and a 
function \(f\in B^s_{p,p}(\Gamma)\) such that
\[
  \mathcal E_{s,p;\Omega_i}(\mathcal P_{\Omega_i}f)=\infty.
\]
\endgroup
\end{enumerate}
\end{theo}

For the examples constructed in both parts of
Theorem~\ref{thm:counter}, the exterior Poisson extensions have finite
weighted energies:  \(\mathcal E_{s,p;\Omega_e}(\mathcal P_{\Omega_e}f)<\infty\).
However, we do not need the exterior estimates, so we omit the proof.

For general chord-arc curves, the same equivalences hold in the following parameter range, close to the critical case $s=1/p$, without any additional Muckenhoupt condition on the conformal derivatives $|\varphi'|$.
\begin{theo}\label{thm:chordarc}
Let $\Gamma$ be a chord-arc curve. If
$1/p\le s<1/2+1/(2p)$, then \eqref{three-1} and \eqref{three-2} hold.
\end{theo}



One consequence of this study is an almost-Dirichlet principle. 
Let $\Omega$ be a Jordan domain with $\partial\Omega = \Gamma$. Suppose that $f\in C(\Gamma)$, a continuous function on  $\Gamma$,  has a continuous extension $F: \overline \Omega \to \mathbb C$ that is in $C^1(\Omega)$. Recall that the classical Dirichlet principle states that 
$$\iint_{\Omega}|\nabla \mathcal P_\Omega f(z)|^2\,dA(z)\le \iint_{\Omega}|\nabla F(z)|^2\,dA(z);$$
that is, among all the extensions of $f$ the one with the lowest $2$-energy is the harmonic one. 
More generally, we say that $\Omega$ satisfies the almost-Dirichlet
principle for $(p,s)$ if there exists a constant $C$ such that
$$\iint_{\Omega}|\nabla \mathcal P_\Omega f(z)|^p  
\delta(z)^{(1-s)p-1}\,dA(z)\le C\iint_{\Omega}|\nabla F(z)|^p  
\delta(z)^{(1-s)p-1}\,dA(z).$$
The almost-Dirichlet estimate on the upper half-plane $\mathbb H$
for $1<p<\infty$ and $0<s<1$ goes back to Uspenskii
\cite{Usp1961}, as noted by Mironescu and Russ \cite[Theorem~1.11]{MR}; this theorem gives a modern
formulation for smooth data. The following proposition shows that
the reverse trace estimate yields this principle on chord-arc domains.

\begin{prop}\label{prop:almost}
Let $\Gamma$ be a bounded chord-arc curve, and fix
$0<s<1$ and $1<p<\infty$.
Suppose that
\[
  \mathcal E_{s,p;\Omega}(\mathcal P_\Omega f)
  \lesssim [f]_{\dot B^s_{p,p}(\Gamma)}
\]
for every $f\in C(\Gamma)\cap B^s_{p,p}(\Gamma)$.
Then $\Omega$ satisfies the almost-Dirichlet principle
for $(p,s)$. 
\end{prop}

 Consider a special case: $s=1-1/p$, $1<p<\infty$, which in particular implies $(1-s)p-1=0$. Then 
for a harmonic function
$u$ on one side $\Omega\in\{\Omega_i,\Omega_e\}$, the distance term $ \delta(z)^{(1-s)p-1}$ disappears in the definition of harmonic  weighted energy: 
\[
  \mathcal E_{1-1/p, p;\Omega}(u)^p
  =\iint_\Omega|\nabla u(z)|^p\,dA(z).
\]
The harmonic $p$-energy space is
\[
  \mathcal B_{p,p}^{1-1/p}(\Omega)
   =\{u:\ u\text{ is harmonic in }\Omega,\quad
          \mathcal E_{1-1/p,p;\Omega}(u)<\infty\}. 
\]
Then by taking $s=1-1/p$ in Corollary \ref{thm:radial}  we get the following corollary of Proposition \ref{prop:almost}: 
\begin{cor}\label{cor:almostB}
Suppose 
  $\Gamma$ is a radial-Lipschitz curve. 
  Then $\Omega $ satisfies the almost-Dirichlet principle for any  $4/3\le p\le4$; that is,
$$\iint_{\Omega}|\nabla \mathcal P_\Omega f(z)|^p  
\,dA(z)\lesssim \iint_{\Omega}|\nabla F(z)|^p\,dA(z).$$
Here,  $F: \overline \Omega \to \mathbb C$ is any continuous extension of $f\in C(\Gamma)$ which is in $C^1(\Omega)$. Moreover, the range $[4/3,\, 4]$ is sharp over the class of radial-Lipschitz domains.
\end{cor}

\revision{The interior estimate in Corollary~\ref{cor:almostB} also follows from
Jerison and Kenig \cite[Theorem~0.5(a)]{JK95} for every bounded planar
Lipschitz domain, by applying their estimate to
$v=\mathcal P_{\Omega_i}f-F$. For each fixed domain, their range of
exponents contains $[4/3,4]$ in its interior.}

As an application of the almost-Dirichlet
principle in Corollary~\ref{cor:almostB}, we obtain
the following sharp endpoint integrability result related to
Brennan's conjecture. This connection  was brought to our attention by Svitlana Mayboroda (private communication).

\begin{theo}\label{thm:Bren}
Let $\Gamma$ be a radial-Lipschitz curve, and let $\Omega :=\Omega_i$.  Then the Riemann map
$g:=\varphi^{-1}:\Omega\to\mathbb D_i$ satisfies
$g'\in L^4(\Omega)$.

The exponent $4$ is optimal uniformly over the class of radial-Lipschitz curves:
for every $p>4$, there exists a bounded radial-Lipschitz
domain $\Omega_p$ whose Riemann map
$g:\Omega_p\to\mathbb D_i$ satisfies
$g'\notin L^p(\Omega_p)$.
\end{theo}

\noindent
The geometric assumption is essential: on general
simply connected domains, $L^4$-integrability may fail, as
illustrated by the classical slit-tip example recalled in 
\cite[Section~1.1]{BIPS}.

The paper is organized as follows. Section~2 collects the
preliminary results. Section~3 develops the weighted trace
and extension framework and proves
Theorem~\ref{thm:inclusion}\textup{(1)} and
Proposition~\ref{prop:almost}. Section~4 examines the geometric
content of the trace estimates, proving the chord-arc
characterization in Theorem~\ref{thm:inclusion}\textup{(2)}
and constructing the counterexample in
Theorem~\ref{thm:reverse-not-chordarc}. Section~5 uses harmonic
conjugation and Cauchy integrals to establish
Proposition~\ref{conj} and
Theorem~\ref{thm:reverseinclusion}\textup{(a)}.
Section~6 treats the radial-Lipschitz curves and the sharpness
of the universal parameter bounds, proving
Corollary~\ref{thm:radial} and Theorem~\ref{thm:counter}, and proving the sharpness of $p$-range in Corollary \ref{cor:almostB}. 
Section~7 completes the reverse trace estimates by
complex interpolation for Theorem~\ref{thm:chordarc} and
Theorem~\ref{thm:reverseinclusion}\textup{(b)}, and by a direct
Poisson estimate for Theorem~\ref{thm:reverseinclusion}\textup{(c)}, and
concludes with the sharp endpoint integrability result
in Theorem~\ref{thm:Bren}.  Finally, the appendix provides the
technical proofs identifying Sobolev traces with averaged
and non-tangential boundary values.

\section{Notation and preliminary results}
\subsection{Weighted Sobolev spaces}\label{sec:Sobolev}

A locally integrable function $\omega$, positive almost everywhere on
$\mathbb C$, belongs to the Muckenhoupt class $A_p$, written
$\omega\in A_p$, if, for $1<p<\infty$,
\begin{equation}\label{AP}
    [\omega]_{A_p}
  :=\sup_{B\subset\mathbb C}
  \left(\frac{1}{|B|}\int_B\omega(z)\,dA(z)\right)
  \left(\frac{1}{|B|}\int_B\omega(z)^{-1/(p-1)}\,dA(z)\right)^{p-1}
  <\infty,
\end{equation}
and, for $p=1$,
\begin{equation}
     [\omega]_{A_1}
  :=\sup_{B\subset\mathbb C}
  \left(\frac{1}{|B|}\int_B\omega(z)\,dA(z)\right)
  \operatorname*{ess\,sup}_{z\in B}\omega(z)^{-1}
  <\infty.
\end{equation}
Here, the suprema are taken over all balls $B\subset\mathbb C$, and
$|B|$ denotes their area. We also set
\[
  A_\infty:=\bigcup_{p\in[1,\infty)}A_p.
\]

For future use, let us note that the $A_p$-weights on $\mathbb T$ are defined similarly.  
It is known that for all $p>1$, $A_1\subset A_p \subset A_{\infty}$,  $\omega\in A_p$ is equivalent to $\omega^{-1/(p-1)} \in A_{p'}$ where $1/p+1/p'=1$, and   $A_{p_1}\subset A_{p_2}$ for $1 < p_1 < p_2$. By  H\"older inequality, the left part of the inequality (\ref{AP}) is always not less than $1$, and thus one can interpret this class as verifying some kind of reverse H\"older inequality.

Let $\Omega\subset\mathbb C$ be open, let $1\le p<\infty$, and let
$\omega\in A_p$. The weighted Lebesgue space $L^p(\Omega,\omega)$
consists of measurable functions $u$ on $\Omega$ for which
\begin{equation}
  \|u\|_{L^p(\Omega,\omega)}
  :=\left(\int_\Omega |u(z)|^p\omega(z)\,dA(z)\right)^{1/p}
  <\infty.
\end{equation}
The weighted Sobolev space $W^{1,p}(\Omega,\omega)$ consists of
functions $u\in L^p(\Omega,\omega)$ whose distributional first
derivatives also belong to $L^p(\Omega,\omega)$, and is equipped with
the norm
\begin{equation}
  \|u\|_{W^{1,p}(\Omega,\omega)}
  :=\left(
    \int_\Omega\bigl(|u(z)|^p+|\nabla u(z)|^p\bigr)
    \omega(z)\,dA(z)
  \right)^{1/p}.
\end{equation}
Here, $\nabla u:=\left(\partial u/\partial x,                    
\partial u/\partial y\right)$ 
denotes the distributional gradient.  When $\omega\equiv1$, we write
$L^p(\Omega)$ and $W^{1,p}(\Omega)$, respectively.

The homogeneous weighted Sobolev space is defined by
\[
\dot W^{1,p}(\Omega,\omega)
:=
\left\{
u\in L_{\mathrm{loc}}^1(\Omega):
\nabla u\in L^p(\Omega,\omega)
\right\}.
\]
It is equipped with the seminorm
\[
[u]_{\dot W^{1,p}(\Omega,\omega)}
:=
\left(
\int_\Omega |\nabla u(z)|^p \omega(z)\,dA(z)
\right)^{1/p}.
\]
If $\Omega$ is connected, the kernel of this seminorm consists precisely
of constant functions. Thus, the seminorm becomes a norm after
passing to the quotient by $\mathbb C$. 

The following global weighted Poincar\'e inequality refers to \cite[Theorem~1.2]{Chua1993}.
\begin{lemma}\label{le:poincare}
Let $1<p<\infty$ and $\omega\in A_p$. Assume that $\Omega$ is a bounded John domain.  Then there exists a constant
$C>0$ such that
\[
  \|u-u_{\Omega,\omega}\|_{L^p(\Omega,\omega)}
  \le C\,\operatorname{diam}(\Omega)
       \|\nabla u\|_{L^p(\Omega,\omega)}
\]
for every $u\in W^{1,p}(\Omega,\omega)$, where
\[
  u_{\Omega,\omega}
  :=\frac{1}{\omega(\Omega)}
    \int_\Omega u(x)\omega(x)\,dx,
  \qquad
  \omega(\Omega):=\int_\Omega\omega(x)\,dx.
\]
\end{lemma}
Recall that $W^{1,p}(\Omega, \omega) = \dot W^{1,p}(\Omega, \omega)$ as sets when $\Omega$ is a bounded quasidisk, which is in particular a John domain.  Lemma \ref{le:poincare} implies that the seminorm $[\cdot]_{\dot W^{1,p}(\Omega,\omega)}$ is equivalent  to the full Sobolev norm $\Vert\cdot\Vert_{W^{1,p}(\Omega,\omega)}$ on the mean-zero subspace: $\{u\in W^{1,p}(\Omega, \omega):\;  u_{\Omega,\omega}=0\}$. However, the seminorm $[\cdot]_{\dot W^{1,p}(\Omega,\omega)}$ and the norm $\Vert\cdot\Vert_{W^{1,p}(\Omega,\omega)}$ are not equivalent on the whole space since for the constant $u\equiv 1$, 
$$
[u]_{\dot W^{1,p}(\Omega,\omega)}=0, \quad \text{while}\quad \Vert u\Vert_{W^{1,p}(\Omega,\omega)}
=\omega(\Omega)^{1/p}>0.
$$

\subsection{Periodic Besov spaces}
The identification
\[
  \mathbb T=\{e^{it}:0\le t<2\pi\}
  \cong\mathbb R/(2\pi\mathbb Z)
\]
allows us to regard functions on $\mathbb T$ as $2\pi$-periodic
functions on $\mathbb R$. We use the periodic Besov
spaces. Let $0<s<1$ and $1<p<\infty$. For a measurable function
$f:\mathbb T\to\mathbb C$, finite almost everywhere, define the
Gagliardo seminorm
\[
  [f]_{\dot B^s_{p,p}(\mathbb T)}
  :=\left(
    \int_{\mathbb T}\int_{\mathbb T}
    \frac{|f(z_1)-f(z_2)|^p}{|z_1-z_2|^{1+sp}}
    |dz_1|\,|dz_2|
  \right)^{1/p}.
\]
The inhomogeneous Besov space $B^s_{p,p}(\mathbb T)$ consists of
functions $f\in L^p(\mathbb T)$ for which
\[
  \|f\|_{B^s_{p,p}(\mathbb T)}
  :=\left(
    \|f\|_{L^p(\mathbb T)}^p+[f]_{\dot B^s_{p,p}(\mathbb T)}^p
  \right)^{1/p}<\infty.
\]
In this range of parameters, these function-space definitions agree
with the usual definitions in terms of periodic distributions. Since
constants have zero Gagliardo seminorm, the homogeneous space is
defined modulo constants:
\[
  \dot B^s_{p,p}(\mathbb T)
  :=\left\{
    f:\mathbb T\to\mathbb C\text{ measurable}:
    [f]_{\dot B^s_{p,p}(\mathbb T)}<\infty
  \right\}/\mathbb C.
\]
Functions are also identified when they agree almost everywhere.
The Gagliardo seminorm induces a norm on this quotient.

Finiteness of the Gagliardo seminorm implies that each representative
$f$ belongs to $L^p(\mathbb T)$.  Nevertheless, the $L^p$-norm of $f$ cannot
be controlled by its seminorm alone: for $f\equiv C$,
\[
  [f]_{\dot B^s_{p,p}(\mathbb T)}=0,
  \qquad
  \|f\|_{L^p(\mathbb T)}=(2\pi)^{1/p}|C|.
\]
What the seminorm controls is $f$ modulo constants. Writing
\[
  \widehat f(0):=\frac{1}{2\pi}\int_{\mathbb T}f(z)\,|dz|,
\]
Jensen's inequality gives
\[
  \|f-\widehat f(0)\|_{L^p(\mathbb T)}^p
  \le\frac{1}{2\pi}
    \int_{\mathbb T}\int_{\mathbb T}
    |f(z_1)-f(z_2)|^p\,|dz_1|\,|dz_2|
  \lesssim [f]_{\dot B^s_{p,p}(\mathbb T)}^p.
\]
Thus, on the zero-mean subspace, the Gagliardo seminorm is equivalent
to the full inhomogeneous norm. Each homogeneous class has a unique
zero-mean representative, so the homogeneous space can equivalently
be realized as
\[
  \left\{f\in B^s_{p,p}(\mathbb T):\widehat f(0)=0\right\}.
\]
At the endpoint $p=\infty$, with $0<s<1$, the Besov space has the
H\"older description
\[
  \|f\|_{B^s_{\infty,\infty}(\mathbb T)}
  :=\|f\|_{L^\infty(\mathbb T)}
       +[f]_{\dot B^s_{\infty,\infty}(\mathbb T)},
  \qquad
  [f]_{\dot B^s_{\infty,\infty}(\mathbb T)}
  :=\sup_{\substack{z_1,z_2\in\mathbb T\\z_1\ne z_2}}
    \frac{|f(z_1)-f(z_2)|}{|z_1-z_2|^s},
\]
where the continuous representative is used. Again, the homogeneous
space is taken as modulo constants, and its zero-mean representative
satisfies
\[
  \|f-\widehat f(0)\|_{L^\infty(\mathbb T)}
  \le\sup_{z_1,z_2\in\mathbb T}|f(z_1)-f(z_2)|
  \lesssim[f]_{\dot B^s_{\infty,\infty}(\mathbb T)}.
\]
Consequently, for $1<p\le\infty$, there are natural isomorphisms
\[
  \dot B^s_{p,p}(\mathbb T)
  \cong\{f\in B^s_{p,p}(\mathbb T):\widehat f(0)=0\},
  \qquad
  B^s_{p,p}(\mathbb T)
  \cong\mathbb C\oplus\dot B^s_{p,p}(\mathbb T),
\]
given by the decomposition
$f=\widehat f(0)+(f-\widehat f(0))$.
The inhomogeneous space thus additionally records the constant
Fourier mode.

We next state the periodic versions of the complex interpolation
formulas which will be used later; see \cite[Section~3.6]{ST87}.

\begin{lemma}[Complex interpolation for $B^s_{p,p}(\mathbb T)$]
Let $0<s_0,s_1<1$, $1<p_0,p_1\le\infty$, $\min\{p_0,p_1\}<\infty$, and $0<\theta<1$.
Define
\begin{equation}\label{sp}
  s=(1-\theta)s_0+\theta s_1,
  \qquad
  \frac1p=\frac{1-\theta}{p_0}+\frac{\theta}{p_1}.
\end{equation}
Then
\begin{equation}\label{inhomosp}
  [B^{s_0}_{p_0,p_0}(\mathbb T),B^{s_1}_{p_1,p_1}(\mathbb T)]_\theta
  =B^s_{p,p}(\mathbb T)
\end{equation}
where the equality is understood with equivalent norms.
\end{lemma}

\begin{lemma}[Complex interpolation for $\dot B^s_{p,p}(\mathbb T)$]
Under the same hypotheses, with $s,p$ defined by \eqref{sp},
\begin{equation}\label{homosp}
  [\dot B^{s_0}_{p_0,p_0}(\mathbb T),
   \dot B^{s_1}_{p_1,p_1}(\mathbb T)]_\theta
  =\dot B^s_{p,p}(\mathbb T)
\end{equation}
where the equality is understood with equivalent seminorms. 
\end{lemma}
For completeness, we deduce the homogeneous formula \eqref{homosp}
from \eqref{inhomosp}. The operator
\[
  Pf:=f-\widehat f(0)
\]
is a bounded linear projection from $B^s_{p,p}(\mathbb T)$ onto its
zero-mean subspace. Complex interpolation commutes with the ranges
of a common bounded projection. Identifying homogeneous classes
with their zero-mean representatives, we therefore obtain
\[
\begin{aligned}
  [\dot B^{s_0}_{p_0,p_0}(\mathbb T),
   \dot B^{s_1}_{p_1,p_1}(\mathbb T)]_\theta
  &\cong
  [PB^{s_0}_{p_0,p_0}(\mathbb T),
   PB^{s_1}_{p_1,p_1}(\mathbb T)]_\theta\\
  &\cong
  P[B^{s_0}_{p_0,p_0}(\mathbb T),
    B^{s_1}_{p_1,p_1}(\mathbb T)]_\theta\\
  &=PB^s_{p,p}(\mathbb T)
  \cong\dot B^s_{p,p}(\mathbb T).
\end{aligned}
\]

Let $\Gamma$ be a bounded chord-arc curve.
For $0<s<1$ and $1<p<\infty$, the homogeneous fractional
Besov--Sobolev space $\dot B^s_{p,p}(\Gamma)$ consists of functions
$f\in L^p(\Gamma)$, modulo constants such that
\[
  [f]_{\dot B^s_{p,p}(\Gamma)}
  :=\left(
    \iint_{\Gamma\times\Gamma}
    \frac{|f(z)-f(\zeta)|^p}{|z-\zeta|^{1+sp}}
    |dz|\,|d\zeta|
  \right)^{1/p}<\infty.
\]
As on $\mathbb T$, this quotient can be identified with its
zero-mean representatives. The inhomogeneous space
$B^s_{p,p}(\Gamma)$ consists of functions $f\in L^p(\Gamma)$ for
which
\[
  \|f\|_{B^s_{p,p}(\Gamma)}
  :=\left(
    \|f\|_{L^p(\Gamma)}^p+[f]_{\dot B^s_{p,p}(\Gamma)}^p
  \right)^{1/p}<\infty.
\]
The space $C_c^\infty(\mathbb C)|_\Gamma$ is dense in
$B^s_{p,p}(\Gamma)$. Indeed, this can be seen from the Besov
trace and extension theorems; see 
\cite[Chapter~VI]{JW}, together with the density of
$C_c^\infty(\mathbb C)$ in
$B^{s+1/p}_{p,p}(\mathbb C)$.

Let
$\ell=\operatorname{length}(\Gamma)$, and
let $\gamma:\mathbb R/\ell\mathbb Z\to\Gamma$ be an arc-length parametrization of $\Gamma$. Define
\begin{equation*}
    \lambda(e^{it})=\gamma(t\ell/(2\pi)),\qquad t\in\mathbb R.
\end{equation*}
The chord-arc condition makes $\lambda$ bi-Lipschitz:  for any $t_1, t_2 \in \mathbb R$, 
\begin{equation}\label{t1t2}
    \ell/(2\pi K)|e^{it_1}-e^{it_2}|\leq |\lambda(e^{it_1})-\lambda(e^{it_2})|\leq \ell/4|e^{it_1}-e^{it_2}|
\end{equation}
where $K$ is the chord-arc constant occurring in \eqref{chordarcK}. The pullback
$f\mapsto f\circ\lambda$ is an isomorphism between
$\dot B^s_{p,p}(\Gamma)$ and $\dot B^s_{p,p}(\mathbb T)$. This reduces
the corresponding interpolation identities on $\Gamma$ to those
on $\mathbb T$.

\subsection{Harmonic weighted energy spaces on  $\Omega_e$}\label{subsec: exte}

Let $\Gamma$ be a bounded Jordan curve, and let $\Omega_i$
and $\Omega_e$ be the complementary components of $\Gamma$ in
$\widehat{\mathbb C}$, as above. The purpose of this subsection is to explain that  all assertions in this paper involving the unbounded domain $\Omega_e$ can be transferred to the bounded cases. 

For $1<p<\infty$ and $0<s<1$,
recall that the harmonic weighted energy spaces are
\[
\mathcal B_{p,p}^s(\Omega)
:=\{u:\ u\text{ is harmonic in }\Omega,\ 
\mathcal E_{s,p;\Omega}(u)<\infty\},
\qquad \Omega\in\{\Omega_i,\Omega_e\}.
\]
Since $\infty$ is an interior point of $\Omega_e$ in
$\widehat{\mathbb C}$, we must specify what is meant by
harmonicity at infinity.

After a translation and dilation, we may assume that
$\Gamma\subset\mathbb D_i$ and $0\in\Omega_i$.
If $\Gamma$ is radial-Lipschitz, we choose the translation
so that $0$ is the radial center.
The reflection $\iota:z\mapsto1/\bar z$ maps $\Omega_e$ onto
a bounded Jordan domain $\Omega_i'$ containing $0$, with boundary
$\Gamma'=\iota(\Gamma)=1/\overline{\Gamma}$. By definition, $\Gamma$ is chord-arc
if and only if $\Gamma'$ is chord-arc. Moreover, $\Gamma$
is radial-Lipschitz with respect to $0$ if and only if
$\Gamma'$ is radial-Lipschitz with respect to $0$. The point $0$ corresponds to $\infty$ under this reflection.
A function $u$, harmonic in $\Omega_e\setminus\{\infty\}$,
is said to be harmonic at infinity if
\[
v(z):=u\!\left(1/\bar z\right),
\qquad z\in\Omega_i'\setminus\{0\},
\]
extends harmonically across $0$.
We then set $u(\infty):=v(0)$.
Throughout the paper, harmonicity in $\Omega_e$ is understood
in this sense.

Suppose $\varphi_e$ is a conformal map from $\mathbb D_e$ onto $\Omega_e$ with $\varphi_e(\infty)=\infty$. Define $\psi (z) = \iota\circ\varphi_e\circ\iota(z)$, $0<|z|<1$. One can see easily $\psi$ extends holomorphically across $0$ with $\psi(0)=0$, $\psi'(0) \neq 0$ so that $\psi: \mathbb D_i\to\Omega_i'$ is a conformal map. A direct computation shows 
$$
|\psi'|=|\varphi_e'|/|\varphi_e|^2
$$
almost everywhere on $\mathbb T$. Since $\Gamma$ is compact and does not contain $0$, we have
$$
M^{-2}|\varphi_e'(e^{it})|\leq |\psi'(e^{it})|\leq m^{-2}|\varphi_e'(e^{it})|, \quad m := \min_{\zeta\in\Gamma}|\zeta|>0,\; M := \max_{\zeta\in\Gamma}|\zeta|<\infty
$$
almost everywhere. 
Therefore, $|\varphi_e'|\in A_r$ on $\mathbb T$ if and only if $|\psi'|\in A_r$ on $\mathbb T$ for any $1\leq r \leq \infty$.

Write
\[
\delta(w):=\operatorname{dist}(w,\Gamma),
\qquad
\delta'(z):=\operatorname{dist}(z,\Gamma').
\]
For the corresponding harmonic functions $u$ and $v=u\circ\iota$,
a change of variables and a comparison of the boundary distances give the following.
\begin{equation}\label{uvuv}
\iint_{\Omega_e}|\nabla u(w)|^p\delta(w)^{(1-s)p-1}\,dA(w)
\asymp
\iint_{\Omega_i'}|\nabla v(z)|^p
\delta'(z)^{(1-s)p-1}|z|^{p(1+s)-3}\,dA(z).
\end{equation}
We establish this comparison and the following assertions.

\begin{prop}\label{ie}
Let $1<p<\infty$ and $0<s<1$.
The following statements hold.
\begin{enumerate}
\item[\rm (1)]
If $u$ is harmonic in $\Omega_e$
(including at infinity) and $v=u\circ\iota$, then
\[
\mathcal E_{s,p;\Omega_e}(u)
\asymp
\mathcal E_{s,p;\Omega_i'}(v).
\]
\item[\rm (2)]
If $\Gamma$ is rectifiable and
$f:\Gamma\to\mathbb C$ is measurable, then
\[
[f\circ\iota]_{\dot B_{p,p}^s(\Gamma')}
\asymp
[f]_{\dot B_{p,p}^s(\Gamma)}.
\]
In particular, this applies to corresponding boundary traces
of $u$ and $v$, whenever those traces are defined.
\end{enumerate}
\end{prop}

\begin{proof}
\textbf{Step 1. Derivation of \eqref{uvuv}.}
Recall that $z=1/\bar w$ maps $\Omega_e$ onto $\Omega_i'$.
The chain rule and the Jacobian formula give
\begin{equation}\label{wz24}
|\nabla u(w)|=|\nabla v(z)|\,|z|^2,
\qquad dA(w)=|z|^{-4}\,dA(z).
\end{equation}
Set
\[
m:=\min_{\zeta\in\Gamma}|\zeta|>0,\qquad
M:=\max_{\zeta\in\Gamma}|\zeta|<1.
\]
Choose $w_0\in\Gamma$ so that $\iota(w_0)$ is a nearest
boundary point to $z$. Then
\[
\delta'(z)
=\frac{|w-w_0|}{|w|\,|w_0|}
\geq\frac{\delta(w)}{M|w|}.
\]
Conversely, if $w_1\in\Gamma$ is a nearest boundary point to $w$,
then
\[
\delta'(z)
\leq\frac{|w-w_1|}{|w|\,|w_1|}
=\frac{\delta(w)}{|w|\,|w_1|}
\leq\frac{\delta(w)}{m|w|}.
\]
Combining these inequalities yields
\begin{equation}\label{distwz}
\delta(w)\asymp\delta'(z)|z|^{-1}.
\end{equation}
A change of variables, combined with \eqref{wz24} and
\eqref{distwz}, yields
\begin{equation}\label{uvuvuv}
\int_{\Omega_e}|\nabla u(w)|^p\delta(w)^{(1-s)p-1}\,dA(w)
\asymp
\int_{\Omega_i'}|\nabla v(z)|^p
\delta'(z)^{(1-s)p-1}|z|^\gamma\,dA(z),
\end{equation}
where $\gamma:=p(1+s)-3$.

\textbf{Step 2. Proof of assertion (1).}
It remains to compare the right-hand side of \eqref{uvuvuv}
with $\mathcal E_{s,p;\Omega_i'}(v)^p$; that is, to prove
\[
\int_{\Omega_i'}|\nabla v(z)|^p\delta'(z)^{(1-s)p-1}\,dA(z)
\asymp
\int_{\Omega_i'}|\nabla v(z)|^p
\delta'(z)^{(1-s)p-1}|z|^\gamma\,dA(z).
\]
Set $r=1/4$. Our normalization ensures that
$\overline{B(0,2r)}\subset\Omega_i'$. Let
$A=\{z:1/4<|z|<1/2\}$, and split the integrals into
their contributions from $B(0,r)$ and its complement in $\Omega_i'$.
On $\Omega_i'\setminus B(0,r)$, the factor $|z|^\gamma$
is bounded above and below by positive constants:
since $\Omega_i'$ is bounded, there exists $R<\infty$ such that
$1/4\leq|z|\leq R$ on this set. Therefore,
\[
\begin{aligned}
&\int_{\Omega_i'\setminus B(0,r)}
|\nabla v(z)|^p\delta'(z)^{(1-s)p-1}\,dA(z)\\
&\qquad\asymp
\int_{\Omega_i'\setminus B(0,r)}
|\nabla v(z)|^p\delta'(z)^{(1-s)p-1}|z|^\gamma\,dA(z).
\end{aligned}
\]

We next control the integrals over the disk centered at $0$.
Since $v$ is harmonic on $B(0,2r)$, the function $|\nabla v|^p$
is subharmonic. Its circular mean
\[
h(\rho):=\frac{1}{2\pi}\int_0^{2\pi}
|\nabla v(\rho e^{it})|^p\,dt
\]
is nondecreasing in $\rho$ (see \cite[Chapter~I]{Gar}).
Consequently,
\begin{equation*}
\begin{split}
\int_{B(0,r)}|\nabla v(z)|^p|z|^\gamma\,dA(z)
&=\int_0^{1/4}\int_0^{2\pi}
|\nabla v(\rho e^{it})|^p\rho^{\gamma+1}\,dt\,d\rho\\
&=2\pi\int_0^{1/4}h(\rho)\rho^{\gamma+1}\,d\rho\\
&\leq\frac{2\pi h(1/4)}{(\gamma+2)4^{\gamma+2}},
\end{split}
\end{equation*}
and
\begin{equation*}
\begin{split}
\int_A|\nabla v(z)|^p|z|^\gamma\,dA(z)
&=\int_{1/4}^{1/2}\int_0^{2\pi}
|\nabla v(\rho e^{it})|^p\rho^{\gamma+1}\,dt\,d\rho\\
&=2\pi\int_{1/4}^{1/2}h(\rho)\rho^{\gamma+1}\,d\rho\\
&\geq
\frac{2\pi h(1/4)(2^{\gamma+2}-1)}
     {(\gamma+2)4^{\gamma+2}}.
\end{split}
\end{equation*}
Since $\gamma+2>0$, combining these inequalities yields
\begin{equation}\label{gammaplus}
\int_{B(0,r)}|\nabla v(z)|^p|z|^\gamma\,dA(z)
\leq\frac{1}{2^{\gamma+2}-1}
\int_A|\nabla v(z)|^p|z|^\gamma\,dA(z).
\end{equation}
The same argument with exponent $0$ gives
\begin{equation}\label{0plus}
\int_{B(0,r)}|\nabla v(z)|^p\,dA(z)
\leq\frac13\int_A|\nabla v(z)|^p\,dA(z).
\end{equation}
Let $d_0=\operatorname{dist}(0,\Gamma')=1/M>1$.
For $z\in B(0,2r)$,
\[
d_0-\frac12\leq\delta'(z)\leq d_0+\frac12.
\]
Thus the boundary-distance weight is bounded above and below
by positive constants on this disk. Also, $1/4<|z|<1/2$ on $A$,
so
\begin{equation}\label{BA}
\delta'(z)^{(1-s)p-1}\asymp1
\quad\text{on }B(0,2r),
\qquad
|z|^\gamma\asymp(1/4)^\gamma
\quad\text{on }A.
\end{equation}
Combining \eqref{gammaplus}, \eqref{0plus}, and \eqref{BA} yields
\begin{equation}
\int_{B(0,r)}|\nabla v(z)|^p\delta'(z)^{(1-s)p-1}\,dA(z)
\lesssim
\int_A|\nabla v(z)|^p
\delta'(z)^{(1-s)p-1}|z|^\gamma\,dA(z),
\end{equation}
and
\begin{equation}
\int_{B(0,r)}|\nabla v(z)|^p
\delta'(z)^{(1-s)p-1}|z|^\gamma\,dA(z)
\lesssim
\int_A|\nabla v(z)|^p\delta'(z)^{(1-s)p-1}\,dA(z).
\end{equation}
The annular terms are contained in the corresponding
outer-region integrals. Combining these estimates with the
comparison on $\Omega_i'\setminus B(0,r)$ gives
\begin{equation*}
\begin{split}
&\quad\int_{\Omega_i'}|\nabla v(z)|^p
\delta'(z)^{(1-s)p-1}\,dA(z)\\
&\lesssim
\int_A|\nabla v(z)|^p
\delta'(z)^{(1-s)p-1}|z|^\gamma\,dA(z) +
\int_{\Omega_i'\setminus B(0,r)}|\nabla v(z)|^p
\delta'(z)^{(1-s)p-1}|z|^\gamma\,dA(z)\\
&\lesssim
\int_{\Omega_i'}|\nabla v(z)|^p
\delta'(z)^{(1-s)p-1}|z|^\gamma\,dA(z).
\end{split}
\end{equation*}
The reverse inequality follows in the same way.
Taking $p$th roots proves assertion (1).

\textbf{Step 3. Proof of assertion (2).}
Set $f^\iota=f\circ\iota$
on $\Gamma'$. A change of variables gives
\[
\begin{aligned}
[f^\iota]_{\dot B_{p,p}^s(\Gamma')}^p
={}&\iint_{\Gamma\times\Gamma}
\frac{|f(\zeta)-f(\xi)|^p}{|\zeta-\xi|^{1+sp}} \cdot|\zeta|^{sp-1}|\xi|^{sp-1}
\,|d\zeta|\,|d\xi|.
\end{aligned}
\]
Since $m\leq|\zeta|,|\xi|\leq M$, the additional factor
is bounded above and below by positive constants
depending only on $m,M,p,s$. Hence
\[
[f^\iota]_{\dot B_{p,p}^s(\Gamma')}
\asymp
[f]_{\dot B_{p,p}^s(\Gamma)}.
\]
\end{proof}

\section{Trace and extension: Theorem \ref{thm:inclusion} (1) and Proposition \ref{prop:almost}} \label{subsec:traceoperators}

In this section, we record the trace and extension
results, together with their boundary-value
interpretations. 
Throughout, let $\Gamma$ be a bounded chord-arc
curve, $1<p<\infty$, and $0<s<1$. Set
\[
  \alpha=(1-s)p-1,\qquad
  \delta(z)=\operatorname{dist}(z,\Gamma),\qquad
  \omega(z)=\delta(z)^\alpha,\qquad d\mu=\omega\,dA.
\]
Then $-1<\alpha<p-1$. A chord-arc curve is Ahlfors $1$-regular, so the distance weight
$\omega$ belongs to $A_p$; see \cite[Proposition~2.8]{HaP}.  Recall that \eqref{eq:weighted-trace} states 
\begin{equation*}
    \operatorname{Tr}_\Gamma W^{1,p}(\mathbb C,\omega)
  = B^s_{p,p}(\Gamma)
\end{equation*}
with equivalent norms.

Recall that  the Sobolev trace $\operatorname{Tr}_\Gamma u$  of a function $u\in  W^{1,p}(\mathbb C,\omega)$ on $\Gamma$ is defined by completion. We point out that it also admits a point-wise description. Precisely, it agrees with the weighted full-ball averages: 
\begin{equation}\label{eq:fullballave}  
\lim_{r\downarrow0}  
\frac{1}{\mu(B(x,r))}  
\int_{B(x,r)}  
\bigl|u(z)-(\operatorname{Tr}_\Gamma u)(x)\bigr|  
\,d\mu(z)=0.
\end{equation}
for
$\mathcal H^1$-almost every $x\in\Gamma$. At each such point $x$, the weighted full-ball and one-sided averages have the same limit:  for $\Omega\in\{\Omega_i,\Omega_e\}$,
\begin{equation}\label{eq:strong-one-sided}  
\lim_{r\downarrow0}  
\frac{1}{\mu(B(x,r)\cap\Omega)}  
\int_{B(x,r)\cap\Omega}|u(z)-(\operatorname{Tr}_\Gamma u)(x)|\,d\mu(z)=0.
\end{equation} 
In particular, \begin{equation}\label{eq:equal-average-limits}
\begin{aligned}  
\lim_{r\downarrow0}  
\frac{1}{\mu(B(x,r))}  
\int_{B(x,r)}u(z)\,d\mu(z)  
&=  
\lim_{r\downarrow 0}  
\frac{1}{\mu(B(x,r)\cap\Omega)}  
\int_{B(x,r)\cap\Omega}u(z)\,d\mu(z)\\  
&=(\operatorname{Tr}_\Gamma u)(x).
\end{aligned}
\end{equation}
For harmonic functions,
the mean-value property further yields non-tangential convergence: $u(z)\to (\operatorname{Tr}_\Gamma u)(x)$ as $z\in\Omega$ tends to $x$ non-tangentially when the full-ball convergence \eqref{eq:fullballave}  holds at $x$. We remark here that for the classical unweighted case, these point-wise descriptions of a trace on $\Gamma$ can refer to \cite{JW}, while for the weighted case, we have not been able to locate  a proof in the literature. In the appendix, we provide proofs of these facts for completeness. 

These trace and boundary-value identifications in the above paragraph also extend to the homogeneous setting by standard localization argument. In particular, \begin{equation}\label{prop:homogeneous-trace}
\operatorname{Tr}_\Gamma  \dot W^{1,p}(\mathbb C,\omega)   =\dot B^s_{p,p}(\Gamma)
\end{equation}
modulo constants. More precisely, 
\begin{equation*}
    [f]_{\dot B^s_{p,p}(\Gamma)}\asymp \inf_{[\operatorname{Tr}_\Gamma u]=[f]}\Vert\nabla u\Vert_{L^p(\mathbb C, \omega)}
\end{equation*}
where the brackets in the constraint denote equivalence classes modulo constants. 

An immediate corollary of the weighted trace and extension equality \eqref{prop:homogeneous-trace} is Proposition \ref{prop:almost}. Let's restate it as 
\begin{cor}\label{cor:almost}
Let $\Gamma$ be a bounded chord-arc curve, and fix
$0<s<1$ and $1<p<\infty$.
Suppose that
\begin{equation}\label{eq:ppp}
    \mathcal E_{s,p;\Omega}(\mathcal P_\Omega f)  
\lesssim [f]_{\dot B^s_{p,p}(\Gamma)}
\end{equation}
for every $f\in C(\Gamma)\cap B^s_{p,p}(\Gamma)$.
Then, for any $F\in C(\overline{\Omega})
\cap C^1(\Omega)$ with $F|_\Gamma = f$ it holds that  
\begin{equation}\label{eq:qqq}
    \mathcal E_{s,p;\Omega}(\mathcal P_\Omega f)  
\lesssim [F]_{\dot W^{1,p}(\Omega, \omega)}.
\end{equation}
\end{cor}
\begin{proof}
Let
$F\in C(\overline{\Omega})
\cap C^1(\Omega)$
have boundary values $f$. If $[F]_{\dot W^{1,p}(\Omega, \omega)}=\infty$, the desired inequality is immediate. We therefore assume that this energy is finite. Since  $\omega(z)=\delta(z)^{(1-s)p-1}\in A_p$ and $\Omega$ is chord-arc, in particular, a quasidisk, $F\in \dot W^{1,p}(\Omega, \omega)$  can be extended to $\mathcal E F \in \dot W^{1,p}(\mathbb C, \omega)$ such that $\mathcal E F=F$ almost everywhere on $\Omega$, with seminorm bounded by that of $F$. Combined with the weighted trace and extension equality \eqref{prop:homogeneous-trace} it gives
\[  
[f]_{\dot B^s_{p,p}(\Gamma)}  
\lesssim  
[F]_{\dot W^{1,p}(\Omega, \omega)}.
\] 
Combined with \eqref{eq:ppp}, this gives \eqref{eq:qqq}.
\end{proof}
Another corollary is part~(1) of Theorem \ref{thm:inclusion}; that is, 
\begin{cor}
Let $1<p<\infty$, $0<s<1$. 
    If $\Gamma$ is chord-arc, then
    \begin{equation*}
        [f]_{\dot B_{p,p}^s(\Gamma)}\lesssim\mathcal E_{s,p;\Omega}(u)
    \end{equation*}
for every $u\in\mathcal B_{p,p}^s(\Omega)$ with
$f:=\operatorname{Tr}_\Gamma u$. Consequently,
$\operatorname{Tr}_\Gamma\mathcal B_{p,p}^s(\Omega)
\subset B_{p,p}^s(\Gamma)$.
\end{cor}
\begin{proof}
    If $u\in\mathcal B_{p,p}^s(\Omega)$, then $u$ belongs to $\dot W^{1,p}(\Omega, \omega)$. Choose an extension $\Lambda u$ of $u$ so that $\Lambda u = u$ almost everywhere on $\Omega$ and $\Lambda u\in \dot W^{1,p}(\mathbb C, \omega)$ with seminorm bounded by that of $u$. Set $f = \text{Tr}_\Gamma\Lambda u = \text{Tr}_\Gamma u$. Then, we have
    $$
    [f]_{\dot B_{p,p}^s(\Gamma)} \lesssim \Vert \nabla \Lambda u\Vert_{L^p(\mathbb C, \omega)}           \lesssim\mathcal E_{s,p;\Omega}(u).
    $$
  This completes the proof of corollary.  
\end{proof}

\section{The chord-arc condition: Theorems \ref{thm:inclusion} (2) and \ref{thm:reverse-not-chordarc}}
\label{sec:proof-inclusion}

Throughout this section, we assume $\Omega := \Omega_i$. The purpose of this section is to prove 
Theorems \ref{thm:inclusion} (2) and \ref{thm:reverse-not-chordarc}.  Both proofs are treated for the bounded  case. Notice that $\Gamma'=1/\bar\Gamma$ and $\Gamma$ are simultaneously chord-arc curves or not. Using the assertions in Section \ref{subsec: exte} we can see that the unbounded case follows from the bounded one.

\subsection{Trace inclusion characterizes the chord-arc condition}
\label{sec3:converse}
In this subsection we prove part~(2) of Theorem \ref{thm:inclusion}: 
 if
there is a constant $C$ such that
\begin{equation}\label{sec3:assumed-trace}
  [f]_{\dot B^s_{p,p}(\Gamma)}
  \le C\,\mathcal E_{s,p;\Omega}(\mathcal P_\Omega f),
  \qquad f\in C^1(\mathbb C)|_\Gamma
\end{equation}
then $\Gamma$ is chord-arc. 

Throughout this proof, write
\[
  \ell:=\text{length}(\Gamma),\qquad D:=\operatorname{diam}\Gamma,
  \qquad L_x(r):=\text{length}(\Gamma\cap B(x,r)),
\]
and set
\begin{equation}\label{sec3:exponents}
  \alpha:=(1-s)p-1.
\end{equation}
In particular, $-1<\alpha<p-1$.

We use two geometric properties of a quasicircle. First, a quasicircle has
a bounded-turning property; see \eqref{def:quasi}. Second, it has a two-sided corkscrew property:  there exists a constant $c_0\in(0,1)$ such that, for every $x\in\Gamma$ and $0<r<D$,
there is a point $w\in\Omega_{i,e}$ for which $B(w, c_0r)\subset\Omega_{i,e}\cap B(x,r)$. Then, 
\begin{equation}\label{sec3:corkscrew}
  |w-x|<r,\qquad
  c_0r\le d:=\operatorname{dist}(w,\Gamma)\le r.
\end{equation}

By the two-sided corkscrew property, we choose \(w\) satisfying
\eqref{sec3:corkscrew} on the side of \(\Gamma\) opposite to \(\Omega\):
take \(w\in\Omega_e\) if \(\Omega=\Omega_i\), and \(w\in\Omega_i\)
if \(\Omega=\Omega_e\). For this choice of \(w\), define
\begin{equation}\label{sec3:test-function}
  F_w(z):=\frac1{z-w},\qquad z\in\Omega,\qquad
  f_w:=F_w|_\Gamma.
\end{equation}
The function $F_w$ is holomorphic in a neighborhood of
$\overline{\Omega}$. Moreover, $f_w$ belongs to the test class in
\eqref{sec3:assumed-trace}. Indeed, choose
$\chi\in C_c^\infty(\mathbb C)$ which equals $1$ near $\Gamma$
and vanishes near $w$. The function $\chi(z)/(z-w)$, defined to
be zero near $w$, is smooth on $\mathbb C$ and restricts to $f_w$
on $\Gamma$. The uniqueness for the continuous Dirichlet problem gives
\begin{equation}\label{sec3:poisson-test}
  \mathcal P_\Omega f_w=F_w.
\end{equation}
We now distinguish the two cases to prove the assertion: $\alpha\geq 0$ and $\alpha<0$.

\textbf{The case $\alpha\geq0$}.

Assume first that \(\alpha\ge0\), or equivalently
\(0<s\le 1-1/p\). We use the test function \(F_w(z)=(z-w)^{-1}\)
chosen above and recall that \(c_0r\le d\le r\), where
\(d=\operatorname{dist}(w,\Gamma)\).

For every $z\in\Omega$, the segment joining $z$ to $w$ meets
$\Gamma$. Consequently,
\[
  \operatorname{dist}(z,\Gamma)\leq|z-w|,
  \qquad |z-w|\geq d.
\]
Since $|\nabla F_w(z)|\asymp|z-w|^{-2}$ and $\alpha\geq0$,
polar integration gives
\begin{align}
  \mathcal E_{s,p;\Omega}(F_w)^p
  &\lesssim
  \int_{\{|z-w|\geq d\}}|z-w|^{\alpha-2p}\,dA(z)
  \notag\\
  &\lesssim
  \int_d^\infty t^{\alpha-2p+1}\,dt
  =\frac1\beta d^{-\beta}.
  \label{sec3:positive-energy}
\end{align}
Here, convergence follows from $\beta := 2p-\alpha-2>0$.

We next obtain a lower bound for the boundary seminorm of $f_w$. 
For $z\in \Gamma\cap B(x,r)$, the corkscrew estimates imply
\begin{equation}\label{sec3:positive-E-location}
  |z-w|\leq|z-x|+|x-w|<2r\leq C_0d, \qquad C_0 := 2/c_0.
\end{equation}
Restrict $r$ to a fixed sufficiently small range, independent of $x$,
so that the corkscrew construction applies and
\begin{equation}\label{sec3:positive-small-radius}
  \left(C_0+3\right)r<\frac D2.
\end{equation}
There exists $y\in\Gamma$ such that $|y-w|\geq D/2$.
Indeed, if $a,b\in\Gamma$ satisfy $|a-b|=D$, then at least one of
$|a-w|$ and $|b-w|$ is at least $D/2$.
In view of $d\leq r$ and \eqref{sec3:positive-small-radius},
\[
  |y-w|>\left(C_0+3\right)d.
\]

Choose either subarc of $\Gamma$ joining $x$ to $y$, and let
$\gamma:[0,1]\to\Gamma$ be a continuous injective parametrization
of that subarc, with $\gamma(0)=x$ and $\gamma(1)=y$.
The continuous function $h(t):=|\gamma(t)-w|$ starts below $C_0d$
and ends above $(C_0+3)d$. Let $t_2$ be its first crossing of the
level $(C_0+2)d$, and let $t_1<t_2$ be its last crossing of the
level $(C_0+1)d$ before $t_2$. Then the subarc
$S:=\gamma([t_1,t_2])$ satisfies
\begin{equation}\label{sec3:positive-S}
  (C_0+1)d\leq|\eta-w|\leq(C_0+2)d
  \quad(\eta\in S),
  \qquad \text{length}(S)\geq d.
\end{equation}
The last inequality follows because the radial distances of the
endpoints of $S$ differ by $d$, so their Euclidean distance, and
hence the length of $S$, is at least $d$.

For $z\in \Gamma\cap B(x,r)$ and $\eta\in S$, equations
\eqref{sec3:positive-E-location} and \eqref{sec3:positive-S} give
\[
  |z-\eta|\geq d,
  \qquad |z-w|\leq C_0d,
  \qquad |\eta-w|\leq(C_0+2)d.
\]
Using the identity
\[
  |f_w(z)-f_w(\eta)|
  =\frac{|z-\eta|}{|z-w|\,|\eta-w|}
\]
and $p-1-sp=\alpha\geq0$, we obtain
\begin{align}
  [f_w]_{\dot B^s_{p,p}(\Gamma)}^p
  &=\iint_{\Gamma\times\Gamma}
    \frac{|z-\eta|^\alpha}
         {|z-w|^p|\eta-w|^p}\,|dz|\,|d\eta|
    \notag\\
  &\geq\int_{\Gamma\cap B(x,r)}\int_S
    \frac{|z-\eta|^\alpha}
         {|z-w|^p|\eta-w|^p}\,|d\eta|\,|dz|
    \notag\\
  &\gtrsim d^{\alpha-2p}\text{length}(\Gamma\cap B(x,r))\;\text{length}(S)
   \gtrsim d^{-\beta-1}L_x(r).
   \label{sec3:positive-boundary}
\end{align}
All implicit constants here are independent of $x$ and $r$.

Since $f_w$ is admissible in the assumed trace inequality and
$\mathcal P_{\Omega}f_w=F_w$, combining
\eqref{sec3:positive-energy} and \eqref{sec3:positive-boundary}
yields
\[
  d^{-\beta-1}L_x(r)
  \lesssim[f_w]_{\dot B^s_{p,p}(\Gamma)}^p
  \lesssim\mathcal E_{s,p;\Omega}(F_w)^p
  \lesssim d^{-\beta}.
\]
Therefore
\begin{equation}\label{sec3:positive-local-regularity}
  L_x(r)\lesssim d\leq r.
\end{equation}
The allowable upper bound on $r$ in
\eqref{sec3:positive-small-radius} is denoted by $r_*$.   Thus,  there exist constants $r_*>0$ and $C_*>0$
such that
\[
  L_x(r)\leq C_*r
  \qquad(x\in\Gamma,\ 0<r<r_*).
\]
For $r\ge r_*$, rectifiability gives
\[
  L_x(r)\le\ell\le\frac{\ell}{r_*}\,r.
\]
Thus there is a constant $C_L$ such that
\begin{equation}\label{sec3:global-length}
  L_x(r)\le C_Lr,\qquad x\in\Gamma,\quad r>0.
\end{equation}
This proves that $\Gamma$ is Ahlfors $1$-regular, and thus chord-arc since we have assumed it is a quasicircle. 


\textbf{The case $-1<\alpha<0$}.

In this range, the distance weight is singular near $\Gamma$.
We first establish the geometric estimate needed to control it.
The following argument uses rectifiability and bounded turning; its explicit proof can be obtained by changing $p-2$ in \cite[Lemma 6]{WZ2} to $\alpha$.

\begin{lemma}\label{sec3:tubular-weight}
Let $\Gamma$ be a rectifiable $K$-quasicircle. Then, 
for every $x\in\Gamma$ and $r>0$,
\begin{equation}\label{sec3:local-weight}
  \int_{\Omega\cap B(x,2r)}\delta(z)^\alpha\,dA(z)
  \le C_\alpha
  \left(r^{\alpha+1}L_x(8Kr)+r^{\alpha+2}\right),
  \qquad -1<\alpha<0.
\end{equation} 
\end{lemma}

Return to the test function \eqref{sec3:test-function}, and put
\[
  A_x(t):=\frac{L_x(t)}{t},\qquad t>0.
\]
For  distinct points $z,\zeta\in \Gamma\cap B(x,r)$, we have
\[
  |z-w|<2r,\qquad|\zeta-w|<2r,\qquad|z-\zeta|<2r.
\]
Because $\alpha<0$, the last inequality implies
$|z-\zeta|^\alpha > (2r)^\alpha$. Therefore,
\begin{align}
  [f_w]_{\dot B^s_{p,p}(\Gamma)}^p
  &=\iint_{\Gamma\times\Gamma}
    \frac{|z-\zeta|^\alpha}{|z-w|^p|\zeta-w|^p}
    \,|dz|\,|d\zeta|\notag\\
  &\ge\int_{\Gamma\cap B(x,r)}\int_{\Gamma\cap B(x,r)}
    \frac{|z-\zeta|^\alpha}{|z-w|^p|\zeta-w|^p}
    \,|dz|\,|d\zeta|\notag\\
  &\gtrsim r^{\alpha-2p}L_x(r)^2
   =r^{-\beta}A_x(r)^2.
   \label{sec3:negative-boundary}
\end{align}
Next, we estimate the energy. For $z\in\Omega$, we have $|z-w|\ge d$.
Set
\[
  R_j:=2^jd,\qquad
  Q_j:=\{z\in\Omega:R_j\le|z-w|<2R_j\},
  \qquad j=0,1,2,\ldots.
\]
Since $|\nabla F_w(z)|\asymp|z-w|^{-2}$, these annuli give
\begin{equation}\label{sec3:annular-energy}
  \mathcal E_{s,p;\Omega}(F_w)^p
  \lesssim\sum_{j=0}^\infty R_j^{-2p}
       \int_{\Omega\cap B(w,2R_j)}\delta(z)^\alpha\,dA(z).
\end{equation}
By \eqref{sec3:corkscrew}, there is $C_1\ge1$, depending only on
$c_0$, such that
\[
  B(w,2R_j)\subset B(x,2C_1R_j).
\]
Applying Lemma~\ref{sec3:tubular-weight} with radius $C_1R_j$
and writing $C_2:=8KC_1\ge1$, we obtain
\[
  \int_{\Omega\cap B(w,2R_j)}\delta(z)^\alpha\,dA(z)
  \lesssim R_j^{\alpha+1}L_x(C_2R_j)+R_j^{\alpha+2}.
\]
It follows that
\begin{align}
  \mathcal E_{s,p;\Omega}(F_w)^p
  &\lesssim\sum_{j=0}^\infty
       R_j^{-\beta}\left(1+A_x(C_2R_j)\right)\notag\\
  &\lesssim r^{-\beta}\sum_{j=0}^\infty
       2^{-j\beta}\left(1+A_x(C_2\,2^jr)\right).
       \label{sec3:negative-energy}
\end{align}
To justify the last line, use $c_0r\le d\le r$ and the
monotonicity of $L_x$:
\[
  A_x(C_2\,2^jd)
  =\frac{L_x(C_2\,2^jd)}{C_2\,2^jd}
  \le\frac rd\,A_x(C_2\,2^jr)
  \le c_0^{-1}A_x(C_2\,2^jr).
\]
 The series in
\eqref{sec3:negative-energy} is finite for each fixed $r>0$,
since $L_x(t)\le\ell$ and $\beta>0$.

Now apply \eqref{sec3:assumed-trace} to $f_w$. Equations
\eqref{sec3:poisson-test}, \eqref{sec3:negative-boundary}, and
\eqref{sec3:negative-energy} imply
\begin{equation}\label{sec3:recurrence}
  A_x(r)^2
  \le C\sum_{j=0}^\infty
       2^{-j\beta}\left(1+A_x(C_2\,2^jr)\right),
  \qquad x\in\Gamma,\quad 0<r<D.
\end{equation}
For $0<\varepsilon<D$, define
\[
  M_\varepsilon
  :=\sup\{A_x(t):x\in\Gamma,\ \varepsilon\le t<D\}.
\]
This number is finite, since $A_x(t)\le\ell/\varepsilon$ in
the indicated range. If $\varepsilon\le r<D$ and
$C_2\,2^jr<D$, then $A_x(C_2\,2^jr)\le M_\varepsilon$,
because $C_2\ge1$. If $C_2\,2^jr\ge D$, then instead
\[
  A_x(C_2\,2^jr)\le\frac{\ell}{D}.
\]
Since $\sum_{j=0}^\infty2^{-j\beta}<\infty$, taking the
supremum in \eqref{sec3:recurrence} gives
\begin{equation}\label{sec3:quadratic}
  M_\varepsilon^2
  \le C\left(1+M_\varepsilon+\frac{\ell}{D}\right).
\end{equation}
Solving this quadratic inequality bounds $M_\varepsilon$
independently of $\varepsilon$. Letting $\varepsilon\downarrow0$,
we conclude that
\begin{equation}\label{sec3:negative-regularity}
  L_x(r)\le C r,\qquad x\in\Gamma,\quad 0<r<D.
\end{equation}
This completes the proof of part~(2) of Theorem \ref{thm:inclusion}.

\subsection{Reverse trace inclusion does not characterize the chord-arc condition}
\label{subsec:reverse-not-chordarc}


In this subsection, we show Theorem~\ref{thm:reverse-not-chordarc}: the reverse trace estimate
\begin{equation}\label{eq:nonchordarc-reverse}
\mathcal E_{s,p;\Omega_i}(\mathcal P_{\Omega_i}f)
\lesssim [f]_{\dot B^s_{p,p}(\Gamma)},
\qquad f\in B^s_{p,p}(\Gamma)
\end{equation}
does not imply that $\Gamma$ is chord-arc. 
The example in our recent paper \cite{WZ2} establishes this conclusion in the
critical case $s=1/p$, $1<p<\infty$.
Using the same type of singular-inner construction, we obtain
a single curve for which the reverse trace estimate holds
throughout the full range $0<s<1$, $1<p<\infty$.
We restate Theorem~\ref{thm:reverse-not-chordarc} in the following form.

\begin{prop}\label{prop:reverse-not-chordarc}
There exists a bounded rectifiable quasicircle $\Gamma$ that
is not chord-arc such that, for every $1<p<\infty$ and $0<s<1$,
\eqref{eq:nonchordarc-reverse} holds.
\end{prop}

The proof proceeds in three steps. We first construct a rectifiable quasicircle that is not chord-arc by specifying the corresponding conformal map. We then treat the cases \(sp<1\) and \(sp>1\) separately. The first two steps follow arguments similar to those in \cite{WZ2}. The last step requires more work: in the range $sp>1$, the conformal pullback introduces the negative power $|\varphi'|^{1-sp}$, and the little Bloch property of $\log\varphi'$ gives a radial comparison that controls this factor in terms of boundary separation; this is the additional estimate needed beyond the critical case.

\textbf{Step 1: Construction of the conformal map.} 

Choose a nonzero finite positive singular measure $\sigma$ on
$\mathbb T$ whose Herglotz transform belongs to the little Bloch
space. Such measures are supplied by the classical small-Zygmund
constructions in \cite{dss,kah}. Thus, the function
\[
b(z)=\int_{\mathbb T}\frac{\zeta+z}{\zeta-z}\,d\sigma(\zeta)
\]
satisfies 
\begin{equation}\label{eq:littleBloch}
    \sup_{z\in\mathbb D_i}(1-|z|^2)|b'(z)|=: M<\infty, \qquad \lim_{|z|\to 1}(1-|z|^2)|b'(z)|=0. 
\end{equation}  
Choose $c>0$ sufficiently small so that $cM<1$ 
and define
\[
\varphi(z)=\int_0^z e^{-cb(\zeta)}\,d\zeta.
\]
Notice that $\varphi'=e^{-cb}$ never vanishes and
$\varphi''/\varphi'=-cb'$. The conformal map $\varphi$ satisfies 
$$
\sup_{z\in\mathbb D_i}(1-|z|^2)\left|\frac{\varphi''(z)}{\varphi'(z)}\right| = cM < 1. 
$$
By Becker's criterion (\cite{Becker1972}),  $\varphi$ admits a
quasiconformal extension  to the plane.
Thus, $\Omega_i=\varphi(\mathbb D_i)$ is a quasidisk.

The boundary-limit theorem for Poisson integrals of finite measures says that their non-tangential limits equal to the density of the absolutely continuous part almost everywhere. Since $\sigma$ is singular, that density is zero. Hence, $\mathcal P_{\mathbb D_i}\sigma(z)\to 0$ non-tangentially at almost every $\zeta\in\mathbb T$, then at such a boundary point $|\varphi'(z)|=e^{-c\mathcal P_{\mathbb D_i}\sigma(z)}\to 1$. Thus, 
$\varphi'$ is a non-constant singular inner function.
In particular,
\begin{equation}\label{ex}
|\varphi'(e^{it})|=1
\quad\text{for almost every }t,
\qquad
0<|\varphi'(w)|<1
\quad(w\in\mathbb D_i).
\end{equation}
 That, in particular, implies the boundary
map $t\mapsto\varphi(e^{it})$ is an arclength parametrization and
$\operatorname{length}(\Gamma)=2\pi$.  Moreover, the shorter parameter-arc
satisfies 
\begin{equation}\label{upper-param}
|\varphi(e^{it})-\varphi(e^{is})|
\le \pi\left|\sin\frac{t-s}{2}\right|
=\frac{\pi}{2}|e^{it}-e^{is}|.
\end{equation}

A rectifiable Jordan domain $D$ is a Smirnov domain if
$\varphi'$ is outer in $H^1(\mathbb D_i)$; equivalently,
\begin{equation}\label{outer}
    \log|\varphi'(z)|
    =\frac{1}{2\pi}\int_{\mathbb S}
    \frac{1-|z|^2}{|\zeta-z|^2}\log|\varphi'(\zeta)|\,|d\zeta|,
    \qquad z\in\mathbb D_i.
\end{equation}
Lavrentiev proved that chord-arc domains are Smirnov domains
\cite{lav}; more generally, a Jordan domain with Ahlfors $1$-regular boundary is
Smirnov \cite{zinRegular}.  We see from \eqref{ex} that the  singular inner function $\varphi'$  does not satisfy \eqref{outer} so that $\Gamma$ cannot be chord-arc.

\textbf{Step 2: The range $sp\leq 1$.}

Fix $f\in B^s_{p,p}(\Gamma)$ and put $g=f\circ\varphi$ on
$\mathbb T$.
The  change of variables, together with \eqref{ex}, 
\eqref{upper-param}, yields
\begin{align}
  [g]_{\dot B^s_{p,p}(\mathbb T)}^p
  &=\int_0^{2\pi}\int_0^{2\pi}
     \frac{|g(e^{it})-g(e^{i\theta})|^p}
          {|e^{it}-e^{i\theta}|^{1+sp}}\,dt\,d\theta\notag\\*
  &\lesssim \int_0^{2\pi}\int_0^{2\pi}
     \frac{|f(\varphi(e^{it}))-f(\varphi(e^{i\theta}))|^p}
          {|\varphi(e^{it})-\varphi(e^{i\theta})|^{1+sp}}
       \,dt\,d\theta=[f]_{\dot B^s_{p,p}(\Gamma)}^p,
  \label{eq:singular-inner-besov-pullback}
\end{align}
and also $\Vert g \Vert_{L^p(\mathbb T)} = \Vert f \Vert_{L^p(\Gamma)}$. 
In particular, $g\in B^s_{p,p}(\mathbb T)$.

Since $g\in L^p(\mathbb T)\subset L^1(\mathbb T)$, its
Poisson extension is well-defined. Set
$V=\mathcal P_{\mathbb D_i}g$ and
$U=V\circ\varphi^{-1}=\mathcal P_{\Omega_i}f$. 
The Koebe distance estimate and the chain rule respectively give  
\[
  \delta(\varphi(w))\asymp
  (1-|w|)|\varphi'(w)|\quad (w\in\mathbb D_i),\qquad |\nabla(U\circ\varphi)| = |\varphi'||\nabla U\circ\varphi|.
\]
  Combining these estimates, the change of variables yields 
 \begin{equation}\label{eq:singu}
     \mathcal E_{s,p;\Omega_i}(U)^p \asymp\iint_{\mathbb D_i}|\nabla V(w)|^p
       (1-|w|)^{(1-s)p-1}|\varphi'(w)|^{1-sp}\,dA(w).
 \end{equation}
If $sp\leq 1$, then $|\varphi'|\le1$ shows that the last integral
is at most
\[
  \iint_{\mathbb D_i}|\nabla V(w)|^p
        (1-|w|)^{(1-s)p-1}\,dA(w).
\]
These estimates, combined with 
the unit-circle Poisson characterization of $B^s_{p,p}(\mathbb T)$; i.e., the equivalence of (i) and (ii) in Section 1,  give
\[
  \mathcal E_{s,p;\Omega_i}(\mathcal P_{\Omega_i}f)
  \lesssim \mathcal E_{s,p;\mathbb D_i}(\mathcal P_{\mathbb D_i}g) \asymp   [g]_{\dot B^s_{p,p}(\mathbb T)}
  \lesssim [f]_{\dot B^s_{p,p}(\Gamma)}.
\]
This proves the proposition when $sp \leq 1$. 

\medskip
\textbf{Step 3: The range $sp>1$.}

Starting with \eqref{eq:singu}, we write $w=(1-t)e^{i\theta}$ so that $dA(w)=(1-t)dtd\theta$. Thus, 
\begin{align}
     \mathcal E_{s,p;\Omega_i}(U)^p &\asymp\iint_{\mathbb D_i}|\nabla V(w)|^p
       (1-|w|)^{(1-s)p-1}|\varphi'(w)|^{1-sp}\,dA(w)\notag\\*
       &= \int_0^{2\pi}\int_0^1
  |\nabla V((1-t)e^{i\theta})|^p
  t^{(1-s)p-1} |\varphi'((1-t)e^{i\theta})|^{1-sp}
  (1-t)\,dt\,d\theta\notag\\*
  &\lesssim\int_0^{2\pi}\int_0^1
  |\nabla V((1-t)e^{i\theta})|^p
  t^{p-2} |t\varphi'((1-t)e^{i\theta})|^{1-sp}\,dt\,d\theta.\label{eq:left}
\end{align}
We estimate the two factors \(|\nabla V((1-t)e^{i\theta})|^p\) and \(|t\varphi'((1-t)e^{i\theta})|^{1-sp}\) separately, then combine the resulting upper bounds to obtain the desired inequality \eqref{eq:nonchordarc-reverse}.

\noindent\textbf{Step 3.1: Estimating $|\nabla V((1-t)e^{i\theta})|^p$.}\\   
\noindent For any $w\in\mathbb D$, $\zeta\in\mathbb T$, the Poisson kernel satisfies 
\begin{equation}\label{eq:Poisson1}
    P(w, \zeta) = \text{Re}\frac{\zeta +w}{\zeta - w} = \frac{1-|w|^2}{|\zeta -w|^2}, \qquad |\nabla_wP(w, \zeta)| = \left|\frac{d}{d w}\frac{\zeta +w}{\zeta - w}\right|=\frac{2}{|\zeta - w|^2}.
\end{equation} 
Every compact subset of $\mathbb D_i$ is contained in
$\{w:|w|\le r\}$ for some $0<r<1$. For $|w|\le r$, we have
\begin{equation}\label{eq:Poisson2}
|\nabla_w P(w,e^{iu})|
\le \frac{2}{(1-r)^2},
\qquad
|\nabla_w P(w,e^{iu})|\,|g(e^{iu})|
\le \frac{2|g(e^{iu})|}{(1-r)^2}.
\end{equation}
Since $g\in L^1(\mathbb T)$, both bounds are integrable in $u$
over $(-\pi,\pi]$.
They therefore justify differentiation under the integral.
Then, for any $w\in\mathbb D_i$, we have
$$
\int_{-\pi}^{\pi} P(w, e^{iu})du = 2\pi\quad\Rightarrow \quad\int_{-\pi}^{\pi} \nabla_wP(w, e^{iu})du = 0, 
$$
and 
$$
\nabla V(w) = \frac{1}{2\pi}\int_{-\pi}^\pi\nabla_w P(w, e^{iu}) g(e^{iu})du.
$$
These imply 
\begin{align}
    |\nabla V(w)| &= \left|\frac{1}{2\pi}\int_{-\pi}^\pi\nabla_w P(w, e^{iu}) g(e^{iu})du - \frac{1}{2\pi}\int_{-\pi}^\pi\nabla_w P(w, e^{iu}) g(e^{i\theta})du\right|\notag\\
    &\leq \frac{1}{\pi}\int_{-\pi}^\pi\frac{|g(e^{iu}) - g(e^{i\theta})|}{|e^{iu}-w|^2} du
\end{align}
where $w=(1-t)e^{i\theta}$ and we used \eqref{eq:Poisson1}  on the last step. Changing variables $u=\theta+h$, together with the elementary estimate $|e^{ih}-(1-t)| \geq (1+\pi)^{-1}(t+|h|)$,  we obtain  
$$
|\nabla V((1-t)e^{i\theta})| \leq \frac{1}{\pi}\int_{-\pi}^\pi\frac{|g(e^{i(\theta + h)}) - g(e^{i\theta})|}{|e^{ih}-(1-t)|^2} dh\lesssim \int_{-\pi}^\pi\frac{|g(e^{i(\theta + h)}) - g(e^{i\theta})|}{(t+|h|)^2} dh. 
$$
For any $0<\epsilon < p(1-s)/2$, H\"older's inequality yields 
\begin{equation}
    |\nabla V((1-t)e^{i\theta})|^p \lesssim t^{p(s-1)+2\epsilon} \int_{-\pi}^\pi\frac{|g(e^{i(\theta + h)}) - g(e^{i\theta})|^p}{(t+|h|)^{sp+1+2\epsilon}} dh
\end{equation}
for almost every $\theta$ and every $0<t<1$. 
Indeed, by using $2\epsilon < p(1-s)$ the direct computation shows the complementary H\"older factor is bounded by 
$$
\left(\int_{-\pi}^\pi(t+|h|)^{(-2p+sp+1+2\epsilon)/(p-1)}dh\right)^{p-1}\lesssim t^{p(s-1)+2\epsilon}. 
$$

\noindent\textbf{Step 3.2: Estimating $|t\varphi'((1-t)e^{i\theta})|^{1-sp}$}. \\
Recall that $w=(1-t)e^{i\theta}$. For $0<|h|\le\pi t$ we show that
$|\varphi(e^{i(\theta+h)})-\varphi(e^{i\theta})|\lesssim t|\varphi'(w)|$.
For $\pi t<|h|\le\pi$, the estimate also depends on $|h|/t$.

Case 1: $0<|h|\leq \pi t$. We have 
$$
|e^{i\theta}-w|=t, \qquad |e^{i(\theta + h)}-w|\leq |e^{i(\theta+h)}-e^{i\theta}|+|e^{i\theta}-w|\leq |h|+t  \leq (\pi + 1)t.
$$
Choose $\xi\in\mathbb T$ such that $|\varphi(\xi)-\varphi(w)| = \delta(\varphi(w))$. Since $|\xi|=1$, $|\xi - w|\geq 1-|w|=t$. Consequently, for either $\zeta = e^{i\theta}$ or $\zeta=e^{i(\theta+h)}$, we have 
$$
\frac{|\zeta - w|}{|\xi-w|}\leq \pi+1.
$$
Since $\Gamma$ is a quasicircle, the conformal map $\varphi$ can be extended to $\mathbb C$ quasiconformally. The quasiconformal map in the plane is quasisymmetric, which implies that there is an increasing function $\eta$ such that 
\begin{equation}
    \frac{|\varphi(\zeta)-\varphi(w)|}
     {|\varphi(\xi)-\varphi(w)|}\le
\eta\!\left(\frac{|\zeta-w|}{|\xi-w|}\right)\leq\eta(\pi+1) 
\end{equation}
and hence
$$
|\varphi(\zeta)-\varphi(w)|
\le \eta(\pi+1)\,\delta(\varphi(w)).
$$
 Combining it with the Koebe distance estimate yields 
\begin{align}
    |\varphi(e^{i(\theta+h)})-\varphi(e^{i\theta})| &\le
|\varphi(e^{i(\theta+h)})-\varphi(w)|
+
|\varphi(e^{i\theta})-\varphi(w)|\notag\\
&\le 2\eta(\pi+1)\,\delta(\varphi(w))\asymp 2\eta(\pi+1)(1-|w|)|\varphi'(w)|\notag\\
&\lesssim t|\varphi'((1-t)e^{i\theta})|.\label{eq:step32} 
\end{align}
Since $1-sp<0$, this implies
\begin{equation}\label{eq:estimate1}
    \left(t|\varphi'((1-t)e^{i\theta})|\right)^{1-sp}\lesssim \frac{1}{|\varphi(e^{i(\theta+h)})-\varphi(e^{i\theta})|^{sp-1}}, \quad 0<|h|\leq\pi t. 
\end{equation}

Case 2: $\pi t<|h|\leq\pi$. Set $T := |h|/\pi \in (t, 1]$. Since $|h| =\pi T$, the estimate \eqref{eq:step32} from Case 1, with $t$ replaced by $T$, yields   
\begin{equation}\label{eq:step1}
    |\varphi(e^{i(\theta+h)})-\varphi(e^{i\theta})| \lesssim  T|\varphi'((1-T)e^{i\theta})|.
\end{equation}
Put $\beta := \epsilon/(sp-1)$. By the little Bloch property \eqref{eq:littleBloch}, there exists $r_0\in (0, 1)$ such that 
\[
\left|\frac{\varphi''(re^{i\theta})}
{\varphi'(re^{i\theta})}\right|
\le\frac{\beta}{1-r},
\qquad r_0<r<1.
\]
uniformly in $\theta$. 
Set $M:= \max_{|z|\leq r_0}|\varphi''(z)/\varphi'(z)|<\infty$. Then, throughout the unit disk, 
\[
\left|\frac{\varphi''(re^{i\theta})}
{\varphi'(re^{i\theta})}\right|
\le M + \frac{\beta}{1-r}.
\]
Therefore,  
$$
\frac{d}{dr}\log|\varphi'(re^{i\theta})|
=
\operatorname{Re}\left(
e^{i\theta}\frac{\varphi''(re^{i\theta})}
{\varphi'(re^{i\theta})} \right)\geq -M - \frac{\beta}{1-r}.
$$
Integrating from \(1-T\) to \(1-t\), we obtain
\[
\begin{aligned}
\log\frac{|\varphi'((1-t)e^{i\theta})|}
         {|\varphi'((1-T)e^{i\theta})|}
&\ge -M(T-t)
-\beta\int_{1-T}^{1-t}\frac{dr}{1-r}\\
&=-M(T-t)-\beta\log\frac{T}{t}\\
&\ge -M-\beta\log\frac{T}{t}.
\end{aligned}
\]
Then  
\begin{equation}\label{eq:step2}
    t|\varphi'((1-t)e^{i\theta})|
\ge e^{-M}
\left(\frac tT\right)^{1+\varepsilon/(sp-1)}
T|\varphi'((1-T)e^{i\theta})|.
\end{equation}
Combining \eqref{eq:step1} and \eqref{eq:step2} yields  
\begin{equation}\label{eq:estimate2}
   \left( t|\varphi'((1-t)e^{i\theta})|\right)^{1-sp}\lesssim \frac{(|h|/t)^{sp-1+\epsilon}}{|\varphi(e^{i(\theta+h)})-\varphi(e^{i\theta})|^{sp-1}}, \quad \pi t<|h|\leq\pi.
\end{equation}
Finally, combining \eqref{eq:estimate1} and \eqref{eq:estimate2} yields
\begin{equation}\label{eq:estimate3}
   \left( t|\varphi'((1-t)e^{i\theta})|\right)^{1-sp}\lesssim \frac{\max\{1, \,(|h|/t)^{sp-1+\epsilon}\}}{|\varphi(e^{i(\theta+h)})-\varphi(e^{i\theta})|^{sp-1}}, \quad 0<|h|\leq\pi.
\end{equation}

\noindent\textbf{Step 3.3: Final computation}.\\
By inserting the estimates from Steps 3.1 and 3.2 into \eqref{eq:left} and applying Tonelli's theorem, we are led to estimate the following integral in $t$. For $0<|h|\leq\pi$, 
\begin{align*}
&\quad\int_0^1
  \frac{t^{sp+2\varepsilon-2}
        \max\{1,(|h|/t)^{sp-1+\varepsilon}\}}
       {(t+|h|)^{sp+1+2\varepsilon}}\,dt\\
&\le
  |h|^{-2-\varepsilon}
  \int_0^{|h|}t^{\varepsilon-1}\,dt
  +\int_{|h|}^{\infty}t^{-3}\,dt
  =\left(\frac1\varepsilon+\frac12\right)|h|^{-2}.
\end{align*}
For this estimate we extended the integration range to
$(0,\infty)$ and split at $t=|h|$; thus it also covers
$|h|>1$. Consequently,
\begin{align*}
  \mathcal E_{s,p;\Omega_i}(U)^p
  &\lesssim
  \int_0^{2\pi}\int_{-\pi}^{\pi}
  \frac{|g(e^{i(\theta+h)})-g(e^{i\theta})|^p}
       {|\varphi(e^{i(\theta+h)})-\varphi(e^{i\theta})|^{sp-1}
        |h|^2}\,dh\,d\theta\\
  &\le
  \int_0^{2\pi}\int_{-\pi}^{\pi}
  \frac{|g(e^{i(\theta+h)})-g(e^{i\theta})|^p}
       {|\varphi(e^{i(\theta+h)})-\varphi(e^{i\theta})|^{sp+1}}
       \,dh\,d\theta
  =[f]_{\dot B^s_{p,p}(\Gamma)}^p.
\end{align*}
In the last line we used \eqref{upper-param} and $|d\varphi(e^{i\theta})|=d\theta$.  
Taking \(p\)th roots gives
\[
\mathcal E_{s,p;\Omega_i}(U)
\lesssim [f]_{\dot B^s_{p,p}(\Gamma)},
\]completing the proof for the range \(sp>1\).

\section{Reverse trace inclusion I: Proposition \ref{conj} and Theorem \ref{thm:reverseinclusion}  (a)}
\label{sec4:conjugation-lifting}

The main goal of this section is to prove part~\textup{(a)} of Theorem \ref{thm:reverseinclusion}. 
Throughout this section, $\Omega=\Omega_i$ is a bounded chord-arc
domain with boundary $\Gamma$. Fix $z_0\in\Omega$, and choose a conformal
map $\varphi:\mathbb D\to\Omega$ with $\varphi(0)=z_0$.
We first establish boundedness of harmonic conjugation on boundary spaces and then use
a Cauchy integral lifting and a weighted Beurling estimate to obtain
the reverse trace inclusion asserted in part~\textup{(a)} of Theorem \ref{thm:reverseinclusion}.

\subsection{Preliminaries on interpolation}

For $1<p<\infty$, the first-order periodic Sobolev space
$W^{1,p}(\mathbb T)$ consists of all functions $g\in L^p(\mathbb T)$
whose distributional derivative $g'$ belongs to $L^p(\mathbb T)$. It is a Banach space equipped
with the norm
\[
  \|g\|_{W^{1,p}(\mathbb T)}
  :=
  \left(
    \|g\|_{L^p(\mathbb T)}^p
    +\|g'\|_{L^p(\mathbb T)}^p
  \right)^{1/p}.
\]
Smooth periodic functions are dense in $W^{1,p}(\mathbb T)$.
This follows from periodic mollification; compare the approximation
argument in \cite[p.~211]{Brez}.

The real interpolation identity
\begin{equation}\label{p1p}
  \bigl(L^p(\mathbb T),W^{1,p}(\mathbb T)\bigr)_{s,p}
  =B^s_{p,p}(\mathbb T),
  \qquad 0<s<1
\end{equation}
holds with equivalent norms. For the corresponding identity on
$\mathbb R$, see \cite[p.~142]{BL}; the version on
$\mathbb T$ follows by standard localization. We emphasize here that \eqref{p1p} is a real interpolation identity.
By contrast, Calder\'on's complex interpolation method gives $[L^p(\mathbb T),W^{1,p}(\mathbb T)]_s$ equal to the period Bessel-Potential space with equivalent norms.  For $p\ne2$, this space does not coincide with $B^s_{p,p}(\mathbb T)$.

The following functorial property of interpolation spaces will be
used repeatedly; see \cite[Theorems~3.1.2 and~4.1.2]{BL}.

\begin{theo}[Interpolation of linear operators]\label{functor}
Let $(X_0,X_1)$ and $(Y_0,Y_1)$ be compatible couples of
complex Banach spaces. Suppose that
\[
  T:X_0+X_1\longrightarrow Y_0+Y_1
\]
is a complex-linear operator such that $T(X_j)\subset Y_j$ and
\[
  \|Tf\|_{Y_j}\le M_j\|f\|_{X_j},
  \qquad f\in X_j,\quad j=0,1,
\]
for some constants $M_0,M_1\ge0$.
Then, for every $0<\theta<1$, the following assertions hold.

\begin{enumerate}
\item \textup{(Complex interpolation)}
The restriction of $T$ defines a bounded operator
\[
  T:[X_0,X_1]_\theta\longrightarrow[Y_0,Y_1]_\theta,
\]
and
\[
  \|Tf\|_{[Y_0,Y_1]_\theta}
  \le M_0^{1-\theta}M_1^\theta
       \|f\|_{[X_0,X_1]_\theta},
  \qquad f\in[X_0,X_1]_\theta.
\]

\item \textup{(Real interpolation)}
For every $1\le q\le\infty$, the restriction of $T$ defines
a bounded operator
\[
  T:(X_0,X_1)_{\theta,q}
    \longrightarrow(Y_0,Y_1)_{\theta,q},
\]
and
\[
  \|Tf\|_{(Y_0,Y_1)_{\theta,q}}
  \le M_0^{1-\theta}M_1^\theta
       \|f\|_{(X_0,X_1)_{\theta,q}},
  \qquad f\in(X_0,X_1)_{\theta,q}.
\]
\end{enumerate}
Here, $[\cdot, \cdot]_\theta$ denotes Calder\'on's complex
interpolation method, and $(\cdot,\cdot)_{\theta,q}$ denotes
real interpolation by the $K$-method. Both interpolation
spaces are equipped with their standard interpolation norms.
\end{theo}

\subsection{Conjugation on boundary spaces}

For a real-valued harmonic function $u$ in $\Omega$, fix $z_0\in \Omega$, and let $v$ be
its harmonic conjugate, normalized by $v(z_0)=0$. Thus $u+iv$ is
holomorphic and the Cauchy--Riemann equations give
$|\nabla v|=|\nabla u|$, which in particular implies that  the harmonic energy spaces
are invariant under conjugation. Conjugation of complex-valued
functions is defined by complex linearity.

Let $\Gamma$ be a chord-arc curve. 
For a real-valued measurable  function $f$ on $\Gamma$ such that $f\circ \varphi\in L^1(\mathbb T)$.  
Define the conjugation operator 
\[
  K_\Gamma f:=\operatorname{Tr}_\Gamma v,\quad u=\mathcal P_\Omega(f).
\]
At this stage, the trace means the nontangential boundary value and it exists almost everywhere with respect to arclength measure. Recall that the trace $\operatorname{Tr}_\Gamma v$ defined in this way  coincides with the Sobolev trace  of  $v$ defined by completion when $v\in \dot W^{1,p}(\Omega,\omega)$ where $\omega(z)=\delta(z)^\alpha \in A_p$; see Section \ref{subsec:traceoperators}.  
On the unit circle $\mathbb T$, conjugation normalized at $0$ is the Hilbert transform
\[
  Hg(e^{it})
  =\operatorname{p.v.}\frac{1}{2\pi}
    \int_{-\pi}^{\pi}
       \cot\!\left(\frac{t-\theta}{2}\right)g(e^{i\theta})\,d\theta.
\]
Consequently,
\begin{equation}\label{sec4:conformal-conjugation}
  (K_\Gamma f)\circ\varphi=H(f\circ\varphi)
\end{equation}
almost everywhere on $\mathbb T$. In particular, $K_\Gamma(f+c)=K_\Gamma f$ for every constant $c$.

To express the conjugation operator in arclength coordinates, let
$\ell=\operatorname{length}(\Gamma)$, and
let $\gamma:\mathbb R/\ell\mathbb Z\to\Gamma$ be an arc-length parametrization of $\Gamma$. Define
\[
  \lambda(e^{it})=\gamma(t\ell/(2\pi)),\qquad
  h=\lambda^{-1}\circ\varphi|_{\mathbb T}.
\]
The chord-arc condition makes $\lambda$ bi-Lipschitz; see \eqref{t1t2}.
The circle
homeomorphism $h$ has an increasing absolutely continuous lift
$\eta$ with $\eta(t+2\pi)=\eta(t)+2\pi$. Writing
\[
  h(e^{it})=e^{i\eta(t)},\qquad
  \rho(t)=\eta'(t),\qquad
  w(t)=|\varphi'(e^{it})|,
\]
we have $w=\ell/(2\pi)\rho$ for almost every $t$. Both $h$ and $h^{-1}$
are absolutely continuous, and $\rho>0$ almost everywhere.
For $g=f\circ\lambda$, \eqref{sec4:conformal-conjugation}
becomes
\[
  (K_\Gamma f)\circ\lambda=T_hg,
  \qquad T_h=V_h^{-1}HV_h,
  \qquad V_hg=g\circ h.
\]

\begin{theo}\label{sec4:conjugation}
Let $1<p<\infty$, $0<s<1$ and $p'=p/(p-1)$,  and let $\Gamma$ be chord-arc. 
Assume that
\begin{equation}\label{sec4:weight-assumption}
  \begin{cases}
    |\varphi'|\in A_p,&1<p\le2,\\
    |\varphi'|\in A_{p'},&2<p<\infty
  \end{cases}
\end{equation}
on $\mathbb T$. 
Then, the conjugation operator $K_\Gamma$ is bounded on $B^s_{p,p}(\Gamma)$:
\begin{equation}\label{sec4:inhomogeneous-conjugation}
  \|K_\Gamma f\|_{B^s_{p,p}(\Gamma)}
  \lesssim \|f\|_{B^s_{p,p}(\Gamma)}.
\end{equation}
\end{theo}

\begin{proof}
Since the Muckenhoupt classes increase with their index, the assumption
\eqref{sec4:weight-assumption} implies  
\[
  \rho\in A_p \cap A_{p'}.
\]
Equivalently, both $\rho$ and $\nu:=\rho^{1-p}$ belong to $A_p$ since a weight $\omega\in A_p$ is equivalent to a weight $\omega^{-1/(p-1)} \in A_{p'}$.
The weighted Hilbert transform theorem therefore gives boundedness
of $H$ on both $L^p(\rho\,dt)$ and  $L^p(\nu\,dt)$; see
\cite{HMW}.

First, the  change of variables gives
\[
  \|g\circ h\|_{L^p(\rho\,dt)}=\|g\|_{L^p(dt)}.
\]
Applying the weighted Hilbert transform Theorem and then changing
variables back, we obtain
\begin{equation}\label{sec4:Lp-conjugation}
  \|T_hg\|_{L^p(\mathbb T)}\lesssim\|g\|_{L^p(\mathbb T)}, 
\end{equation}
which implies the boundedness of $T_h$ on $L^p(\mathbb T)$. 

Next, we prove the boundedness of $T_h$ on $W^{1,p}(\mathbb T)$.
Let $g\in C^\infty(\mathbb T)$ and set $G=g\circ h$.
Hence, 
\[
  \|G'\|_{L^p(\nu\,dt)}^p
  =\int_0^{2\pi}|g'(h(e^{it}))|^p\rho(t)\,dt
  =\|g'\|_{L^p(dt)}^p.
\]
It follows from $(HG)'=H(G')$ that 
\begin{equation}\label{sec4:derivative-conjugation}
  \|(T_hg)'\|_{L^p(dt)}^p
  =\int_0^{2\pi}|H(G')|^p\rho^{1-p}\,dt
  \lesssim\|G'\|_{L^p(\nu\,dt)}^p
  =\|g'\|_{L^p(dt)}^p.
\end{equation}
Together with \eqref{sec4:Lp-conjugation} and density of smooth functions, this proves that $T_h$ is bounded on
$W^{1,p}(\mathbb T)$.

Applying the real interpolation identity \eqref{p1p} and Theorem \ref{functor} gives $T_h$ is bounded on $B^s_{p,p}(\mathbb T)$. Finally, we have
\[
  \|K_\Gamma f\|_{B^s_{p,p}(\Gamma)}
  \asymp\|T_h(f\circ\lambda)\|_{B^s_{p,p}(\mathbb T)}
  \lesssim\|f\circ\lambda\|_{B^s_{p,p}(\mathbb T)}
  \asymp\|f\|_{B^s_{p,p}(\Gamma)}.
\]
This  proves
\eqref{sec4:inhomogeneous-conjugation}.
\end{proof}

We next pass from the inhomogeneous estimate to the homogeneous quotient space $\dot B^s_{p,p}(\Gamma)$.

\begin{prop}\label{sec4:homogeneous-conjugation}
Under the assumptions of Theorem~\ref{sec4:conjugation}, the map
$f\mapsto K_\Gamma f$ is a well-defined bounded operator on
$\dot B^s_{p,p}(\Gamma)$. More precisely,
\begin{equation}\label{sec4:homogeneous-K-bound}
  [K_\Gamma f]_{\dot B^s_{p,p}(\Gamma)}
  \lesssim [f]_{\dot B^s_{p,p}(\Gamma)}.
\end{equation}
\end{prop}

\begin{proof}
Let $D=\operatorname{diam}(\Gamma)$ and $\ell=\operatorname{length}(\Gamma)$ as above, and set 
\[
  f_\Gamma:=\frac1\ell\int_\Gamma f\,|dz|,
  \qquad f_0 := f-f_\Gamma.
\]
Recall that  Jensen's inequality gives 
\begin{equation}\label{sec4:mean-zero-norm}
  \|f_0\|_{B^s_{p,p}(\Gamma)}
  \lesssim[f]_{\dot B^s_{p,p}(\Gamma)}.
\end{equation}
Since $K_\Gamma f=K_\Gamma f_0$, 
Theorem~\ref{sec4:conjugation} now yields
\[
  [K_\Gamma f]_{\dot B^s_{p,p}(\Gamma)}
  \le\|K_\Gamma f_0\|_{B^s_{p,p}(\Gamma)}
  \lesssim\|f_0\|_{B^s_{p,p}(\Gamma)}
  \lesssim[f]_{\dot B^s_{p,p}(\Gamma)}.
\]
This proves \eqref{sec4:homogeneous-K-bound}. 
\end{proof}

Since $(K_\Gamma f)\circ\varphi = H(f\circ\varphi)$,   $H^2 g=-g+\widehat g(0)$ implies that
\[
  K_\Gamma^2f=-f+c_f,\qquad
  c_f=\frac1{2\pi}\int_0^{2\pi}f(\varphi(e^{it}))\,dt.
\]
Thus $K_\Gamma^2=-\rm{id}$ in the homogeneous quotient space $\dot B^s_{p,p}(\Gamma)$. In particular,
$K_\Gamma$ is an automorphism of $\dot B^s_{p,p}(\Gamma)$ with
inverse $-K_\Gamma$.  Normalization $v(z_0)=0$ generally does not make $K_\Gamma f$
have a zero arclength mean. If the homogeneous space is realized as
the zero-mean subspace, its conjugation operator is therefore
\[
  f\longmapsto K_\Gamma f-(K_\Gamma f)_\Gamma.
\]
This represents the same operator on this homogeneous quotient space.

\subsection{The analytic boundary function and its Cauchy integral}

Let $\Gamma$ be a rectifiable Jordan curve, let $b\in L^p(\Gamma)$, $1<p<\infty$, and define its Cauchy integral away from the curve $\Gamma$ by
\[
  \mathcal C_\Gamma b(z)
  :=\frac{1}{2\pi i}\int_\Gamma
       \frac{b(\zeta)}{\zeta-z}\,d\zeta,
  \qquad z\in\mathbb C\setminus\Gamma.
\]
All boundary formulas below hold for $\mathcal H^1$-almost every $\xi\in\Gamma$. 
The associated principal-value singular integral on the boundary is
$$
\mathcal T_\Gamma b(\xi) :=  \frac{1}{\pi i}\text{p.v.}\int_\Gamma\frac{b(\zeta)}{\zeta - \xi}d\zeta, \quad \xi\in \Gamma.
$$
Define $P_i b$ to be the interior boundary trace of the analytic function $\mathcal C_\Gamma b$:
$$
P_ib(\xi) := \lim_{\Omega\ni z\to\xi} \mathcal C_\Gamma b(z)
$$
where the limits are non-tangential limits.  Similarly, define $-P_eb$ to be the exterior boundary trace of  $\mathcal C_\Gamma b$. Then the Plemelj formulas take the particularly simple form
\begin{align*}
    P_ib(\xi) &= \frac{1}{2}(b(\xi) + \mathcal T_\Gamma b(\xi)),\\
   P_eb(\xi) &= \frac{1}{2}(b(\xi) - \mathcal T_\Gamma b(\xi))
   \end{align*}
which in particular imply 
\begin{equation}\label{eq:sum11}
    b(\xi)= P_ib(\xi) + P_eb(\xi). 
\end{equation}

\begin{lemma}\label{sec4:Smirnov}
Assume \eqref{sec4:weight-assumption}. Let $f\in L^p(\Gamma)$
be real-valued, and let
\[
  u=\mathcal P_\Omega f,\qquad A=u+iv,\qquad a=f+iK_\Gamma f.
\]
Then $A\in E^p(\Omega)$, its nontangential boundary function is
$a\in L^p(\Gamma)$, and
\begin{equation}\label{sec4:Cauchy-inside-outside}
  \mathcal C_\Gamma a(z)
  =\begin{cases}
      A(z),&z\in\Omega,\\
      0,&z\in\Omega_e.
    \end{cases}
\end{equation}
In particular, $P_i a=a$.
\end{lemma}

\begin{proof}
To show $A\in E^p(\Omega)$, which is equivalent to $A\circ\varphi (\varphi')^{1/p} \in H^p(\mathbb D)$ (see \cite[p. 169]{Duren}), it is enough to show $A\circ\varphi (\varphi')^{1/p}$ belongs to the Smirnov class $N^+(\mathbb D)$ and on the boundary it belongs to $L^p(\mathbb T)$ (see \cite[Theorem 2.11]{Duren}). Note that every $H^q$-function, $0<q<\infty$, belongs to the Smirnov class, but the converse is false.

Suppose $f\in L^p(\Gamma)$. 
The $L^p$ estimate in the proof of Theorem~\ref{sec4:conjugation}
gives $a\in L^p(\Gamma)$. On the other hand,  it implies the $L^1$-norm
$$
\int_0^{2\pi}|f\circ\varphi|dt = \int_0^{2\pi}(|f\circ\varphi|\,|\varphi'|^{1/p})|\varphi'|^{-1/p}dt\leq \Vert f \Vert_{L^p(\Gamma)}\left( \int_0^{2\pi}|\varphi'|^{-1/(p-1)}dt\right)^{1/p'}
$$
 is finite since  $|\varphi'|\in A_p$,  then  $g=f\circ\varphi\in L^1(\mathbb T)$.

Since $u$ is the Poisson extension of $f$ and $v(z_0)=0$, the holomorphic function $G=A\circ\varphi$ is the
Schwarz integral of $g$
\[
  G(z)=\frac{1}{2\pi}\int_0^{2\pi}
       \frac{e^{it}+z}{e^{it}-z}\,g(e^{it})\,dt,
  \qquad z\in\mathbb D.
\]
Since $\operatorname{Re}G=\mathcal P_{\mathbb D}g\in h^1(\mathbb D)$,
Kolmogorov's theorem \cite[Theorem~4.2]{Duren} gives
$G\in H^q(\mathbb D)$ for every $0<q<1$. Its boundary function is
$G^*=a\circ\varphi$ almost everywhere.

Since $\Gamma$ is rectifiable, $\varphi'\in H^1(\mathbb D)$.
The derivative has no zeros, so an analytic branch of
$h=(\varphi')^{1/p}$ exists and belongs to $H^p(\mathbb D)$.
Thus, applying  H\"older's inequality gives
\[
  Q:=(A\circ\varphi)(\varphi')^{1/p}=Gh
\]
belongs to $H^r(\mathbb D)$ for some $r>0$.  In particular, $Q$ belongs to the Smirnov
class $N^+(\mathbb D)$. On the boundary,
\[
  \int_0^{2\pi}|Q^*(e^{it})|^p\,dt
  =\int_0^{2\pi}|a(\varphi(e^{it}))|^p
                         |\varphi'(e^{it})|\,dt
  =\int_\Gamma|a|^p\,|d\zeta|<\infty.
\]
This completes the proof of $A\in E^p(\Omega)$. 

Because $p>1$ and $\Gamma$ has finite length, $E^p(\Omega)$ is
contained in $E^1(\Omega)$. The Cauchy representation theorem
for $E^1$-functions on rectifiable Jordan domains gives
\eqref{sec4:Cauchy-inside-outside}; see
\cite[Theorem~10.4]{Duren}. Taking interior nontangential limits
yields $P_i a=a$.
\end{proof}

\subsection{A weighted Cauchy lifting estimate}

The next lemma is the weighted extension step. Unlike the preceding conjugation result, it requires no Muckenhoupt assumption on $|\varphi'|$.

\begin{lemma}\label{sec4:Cauchy-lifting}
Let $\Gamma$ be a bounded chord-arc curve, 
let $0<s<1$, $1<p<\infty$, and set $\alpha=(1-s)p-1$. Then, 
for every $b\in B^s_{p,p}(\Gamma)$,
\begin{equation}\label{sec4:Cauchy-energy}
  \left(\int_\Omega| (\mathcal C_\Gamma b)'(z)|^p
                         \delta(z)^\alpha\,dA(z)\right)^{1/p}
  \lesssim [b]_{\dot B^s_{p,p}(\Gamma)}.
\end{equation}
\end{lemma}

\begin{proof}
Set $b_0:=b-b_\Gamma$, where $b_\Gamma:=\operatorname{length}(\Gamma)^{-1}\int_\Gamma b\,|dz|$. Since $(\mathcal C_\Gamma b_0)'=(\mathcal C_\Gamma b)'$, it suffices to estimate the Cauchy integral of $b_0$. Use the inhomogeneous extension estimate \eqref{eq:weighted-trace} for $b_0$, then multiply the extension by a fixed smooth cutoff equal to $1$ near $\Gamma$ and supported in $\{\delta<1\}$. By the product rule and the mean-zero estimate $\|b_0\|_{B^s_{p,p}(\Gamma)}\lesssim[b]_{\dot B^s_{p,p}(\Gamma)}$, the resulting compactly supported function $F$ satisfies 
\begin{equation}\label{sec4:compact-extension}
  \operatorname{Tr}_\Gamma F=b_0,\qquad
  [F]_{\dot W^{1,p}(\mathbb C,\delta^\alpha)} 
  \lesssim [b]_{\dot B^s_{p,p}(\Gamma)}.
\end{equation}

  Recall that the chord-arc hypothesis and $-1<\alpha<p-1$ imply
$\delta^\alpha\in A_p$ on $\mathbb C$. In particular,
$\delta^{-\alpha/(p-1)}$ is locally integrable. Set $q=\mathbf1_{\Omega_e}\,\bar\partial F$ where $\mathbf{1}_{\Omega_e}$ denotes the characteristic function of $\Omega_e$.
Since $F$ has compact support, $q$ also has compact support and $q\in L^1(\mathbb C)$ by using H\"older's inequality. 
Consider its Cauchy potential
\[
  \mathfrak Cq(z)
  :=-\frac1\pi\int_{\mathbb C}\frac{q(\zeta)}{\zeta-z}\,dA(\zeta).
\]
It is absolutely convergent for almost every $z\in\mathbb C$ and belongs to $L^1_{\rm loc}(\mathbb C)$, and in particular, absolutely convergent  and holomorphic  outside $\text{supp} ~q$. 
The generalized Green's formula (see p.150 in \cite{Lehto-Virtanen}) says that if  $g:D\to\mathbb C$ is continuous in the domain $D$ and absolutely continuous on lines,
with $g_x,g_y\in L^1_{\mathrm{loc}}(D)$, and if $\Omega$ is a Jordan domain with $\overline{\Omega}\subset D$ and rectifiable boundary $\Gamma=\partial\Omega$, then
\begin{equation}\label{eq:greenfor}
    \int_\Omega g_{\bar z} dA = -\frac{i}{2}\int_\Gamma g dz.
\end{equation}
For each fixed $z\in\Omega_i$, it gives
\begin{equation}\label{sec4:Green-lifting}
  \mathcal C_\Gamma b_0(z)
  =\frac{1}{2\pi i}\int_\Gamma
       \frac{b_0(\zeta)}{\zeta-z}\,d\zeta
  =-\frac{1}{\pi}\int_{\Omega_e}
       \frac{\bar\partial F(\zeta)}{\zeta-z}\,dA(\zeta)
  =\mathfrak Cq(z).
\end{equation}
To justify the formula for the Sobolev extension $F$, choose
$F_n\in C_c^\infty(\mathbb C)$ supported in a common compact set $K$
such that $F_n\to F$ in $W^{1,p}(\mathbb C,\delta^\alpha)$.
Choose $R>0$ so that $K\cup\Gamma\subset B(0,R)$, and apply \eqref{eq:greenfor} to $F_n(\zeta)/(\zeta-z)$ on $\Omega_e\cap B(0,R)$. 
By the trace estimate,
$F_n|_\Gamma\to b_0$ in $L^1(\Gamma)$; by weighted H\"older's inequality,
$\bar\partial F_n\to\bar\partial F$ in $L^1(K)$. For fixed $z\in\Omega_i$, we have $|\zeta-z|\ge\delta(z)>0$ on
$\Gamma\cup\Omega_e$. Hence
\[
\begin{aligned}
\left|\int_\Gamma \frac{F_n(\zeta)-b_0(\zeta)}{\zeta-z}\,d\zeta\right|
&\le \frac{\|F_n|_\Gamma-b_0\|_{L^1(\Gamma)}}{\delta(z)}
\longrightarrow 0,\\
\left|\int_{\Omega_e}
\frac{\bar\partial F_n(\zeta)-\bar\partial F(\zeta)}{\zeta-z}\,dA(\zeta)\right|
&\le \frac{\|\bar\partial F_n-\bar\partial F\|_{L^1(K)}}{\delta(z)}
\longrightarrow 0.
\end{aligned}
\]
Passing to the limit in Green's formula for $F_n$ therefore gives
\eqref{sec4:Green-lifting}.

The distributional identities for the Cauchy potential are
\[
  \bar\partial\mathfrak Cq=q,\qquad
  \partial\mathfrak Cq=\mathcal Sq,
  \qquad
  \mathcal Sq(z)
  :=-\frac1\pi\operatorname{p.v.}
       \int_{\mathbb C}\frac{q(\zeta)}{(z-\zeta)^2}\,dA(\zeta),
\]
where $\mathcal S$ is the Beurling transform. Since $\delta^\alpha\in A_p$, 
the weighted
Calder\'on--Zygmund estimate \cite{CoF} gives
\begin{equation}\label{sec4:Beurling}
  \|\mathcal Sq\|_{L^p(\mathbb C,\delta^\alpha)}
  \lesssim\|q\|_{L^p(\mathbb C,\delta^\alpha)}.
\end{equation}
Since $q=0$ on $\Omega$, \eqref{sec4:Green-lifting} yields
$(\mathcal C_\Gamma b)'=\mathcal Sq$ there.  Therefore,
\begin{align*}
  \|(\mathcal C_\Gamma b)'\|_{L^p(\Omega,\delta^\alpha)}
  &\le\|\mathcal Sq\|_{L^p(\mathbb C,\delta^\alpha)}\\
  &\lesssim\|q\|_{L^p(\mathbb C,\delta^\alpha)}
   \lesssim\|\nabla F\|_{L^p(\mathbb C,\delta^\alpha)}
   \lesssim[b]_{\dot B^s_{p,p}(\Gamma)}.
\end{align*}
This proves \eqref{sec4:Cauchy-energy}.
\end{proof}

A similar argument shows that \eqref{sec4:Cauchy-energy} remains valid when the integral over $\Omega$ on the left-hand side is replaced by an integral over $\Omega_e$. Denote $\mathcal B_{p,p}^{s,\text{hol}}(\Omega)$ to be the subset of  $\mathcal B_{p,p}^{s}(\Omega)$ consisting of    holomorphic functions in $\Omega$.  Combining Lemma \ref{sec4:Cauchy-lifting} and  \eqref{eq:sum11} yields the following. 
\begin{theo}
under the assumptions of Lemma \ref{sec4:Cauchy-lifting}, 
    $B^s_{p,p}(\Gamma)$ is the sum of ${\rm Tr}_\Gamma\mathcal B_{p,p}^{s,{\rm hol}}(\Omega_i)$ and ${\rm Tr}_\Gamma\mathcal B_{p,p}^{s,{\rm hol}}(\Omega_e)$. 
\end{theo}
A corollary of this theorem is the following.
\begin{cor}\label{cor:threeequal}
    Let $\Gamma$ be a bounded chord-arc curve, and 
let $0<s<1$, $1<p<\infty$. If 
$$
{\rm Tr}_\Gamma\mathcal B_{p,p}^{s}(\Omega_i) = {\rm Tr}_\Gamma\mathcal B_{p,p}^{s}(\Omega_e)
$$
then these two spaces are both equal to $B^s_{p,p}(\Gamma)$. 
\end{cor}

\subsection{The reverse trace inclusion for general \texorpdfstring{$p$}{p}}
The following theorem is a restate of part~\textup{(a)} of Theorem \ref{thm:reverseinclusion}. 
\begin{theo}\label{sec4:reverse-trace}
Under the assumptions of Theorem~\ref{sec4:conjugation}, every 
$f\in B^s_{p,p}(\Gamma)$ has the Poisson extension
satisfying
\begin{equation}\label{sec4:reverse-energy}
  \mathcal E_{s,p;\Omega}(\mathcal P_\Omega f)
  \lesssim [f]_{\dot B^s_{p,p}(\Gamma)},
  \qquad
  \operatorname{Tr}_\Gamma\mathcal P_\Omega f=f.
\end{equation}
The trace identity holds $\mathcal H^1$-almost everywhere.  Consequently,
\begin{equation}\label{sec4:reverse-inclusion}
   B^s_{p,p}(\Gamma)
  \subset\operatorname{Tr}_\Gamma\mathcal B^s_{p,p}(\Omega).
\end{equation}
\end{theo}

\begin{proof}
First, suppose that $f$ is real-valued, and set
\[
  a=f+iK_\Gamma f,\qquad
  A=\mathcal P_\Omega f+iv,
\]
where $v$ is the normalized harmonic conjugation. By
Proposition~\ref{sec4:homogeneous-conjugation},
\[
  [a]_{\dot B^s_{p,p}(\Gamma)}
  \le[f]_{\dot B^s_{p,p}(\Gamma)}
       +[K_\Gamma f]_{\dot B^s_{p,p}(\Gamma)}
  \lesssim[f]_{\dot B^s_{p,p}(\Gamma)}.
\]
Lemma~\ref{sec4:Smirnov} identifies $A$ with $\mathcal C_\Gamma a$
in $\Omega$. Therefore Lemma~\ref{sec4:Cauchy-lifting} and the
identity $|\nabla\operatorname{Re}A|=|A'|$ give
\[
  \mathcal E_{s,p;\Omega}(\mathcal P_\Omega f)
  =\|A'\|_{L^p(\Omega,\delta^\alpha)}
  \lesssim[a]_{\dot B^s_{p,p}(\Gamma)}
  \lesssim[f]_{\dot B^s_{p,p}(\Gamma)}.
\]
For complex-valued $f$, apply this argument separately to its
real and imaginary parts and use the triangle inequality.
This proves the energy estimate for every $f$.

Next, it remains to identify the weighted trace. 
Recall that $B^s_{p,p}(\Gamma)$ has a dense subspace $C_c^\infty(\mathbb C)|_\Gamma$. 
Choose a sequence of continuous functions $f_n$ on $\Gamma$ converging to $f$ in
$B^s_{p,p}(\Gamma)$. Their Poisson extensions $u_n=\mathcal P_\Omega f_n$ are continuous on
$\overline\Omega$, so their non-tangential limits on $\Gamma$ are exactly
$f_n$. The energy estimate applied to $f_n-f$ gives
\[
  \mathcal E_{s,p;\Omega}
       (u_n-\mathcal P_\Omega f)\longrightarrow0.
\]
By the argument on weighted trace  in Section~\ref{subsec:traceoperators},
\[
  [\operatorname{Tr}_\Gamma\mathcal P_\Omega f-\operatorname{Tr}_\Gamma u_n]
       _{\dot B^s_{p,p}(\Gamma)}\longrightarrow 0, \quad {\rm and} \quad \operatorname{Tr}_\Gamma u_n = f_n.
\]
Thus, $\operatorname{Tr}_\Gamma\mathcal P_\Omega f=f$ in the
homogeneous boundary space. To see that the additive constant is zero, note that 
$(\mathcal P_\Omega f)(\varphi(r\xi))\to f(\varphi(\xi))$ non-tangentially as $r\uparrow1$ for almost every $\xi\in\mathbb T$. 
 Since harmonic measure and arclength on $\Gamma$ have the same null sets, part~$(T_3)$ of Appendix~\ref{app:weighted-traces} identifies the trace with $f$ for $\mathcal H^1$-almost everywhere.
This proves \eqref{sec4:reverse-inclusion}.
\end{proof}

\section{Radial-Lipschitz curves: Corollaries \ref{thm:radial}, \ref{cor:almostB} and  Theorem \ref{thm:counter}}
\label{sec:radial-weights}

We first prove that every radial-Lipschitz curve satisfies the
two-sided $A_2$ condition. Together with
Theorems~\ref{thm:inclusion} and~\ref{thm:reverseinclusion},
this yields Corollary~\ref{thm:radial}.
We then show that the weight exponent \(2\) cannot be
lowered uniformly and prove the sharpness of both smoothness bounds
in Theorem~\ref{thm:counter}. Finally, we prove the range $[4/3, 4]$ of $p$  in Corollary \ref{cor:almostB} is sharp. 

For later use,  we now describe the geometric meaning of  radial-Lipschitz regularity. 
Let $\rho:\mathbb R\to(0,\infty)$ be a continuous
$2\pi$-periodic function. The curve
\[
  \Gamma
  =\{z_0+\rho(t)e^{it}:
       t\in\mathbb R\}
\]
is called a star-shaped Jordan curve with respect to $z_0$,
and its interior domain $\Omega_i$ is called a star-shaped 
domain. If, in addition, $\rho$ is Lipschitz, then $\Gamma$
and $\Omega_i$ are called a radial-Lipschitz curve and a
radial-Lipschitz domain, respectively. 
 More precisely, there exists $L<\infty$ such that
\[
  |\rho(\theta)-\rho(t)|\le Ld_{\mathbb T}(\theta,t),
  \qquad
  d_{\mathbb T}(\theta,t)
  :=\min_{k\in\mathbb Z}|\theta-t+2\pi k|.
\]
Geometrically, every ray issuing from $z_0$ meets $\Gamma$ exactly
once, and the distance from $z_0$ to the intersection varies
Lipschitz continuously with the angle. Its interior,
\[
  \Omega_i=\{z_0+re^{it}:\ 0\le r<\rho(t),\quad
                               0\le t<2\pi\},
\]
is star-shaped with respect to $z_0$.
In this case, the
radial parametrization $\lambda(e^{it}) := z_0+\rho(t)e^{it}$ is bi-Lipschitz  on $\mathbb T$, and hence $\Gamma$
is chord-arc. 

An equivalent geometric characterization is the following. Let $\Gamma$ be a bounded rectifiable Jordan curve and $z_0\in\Omega_i$. Then $\Gamma$ is radial-Lipschitz with respect to $z_0$ if and only if it is uniformly strictly star-shaped with respect to $z_0$, in the sense that there exists a constant $c>0$ such that 
\begin{equation}\label{eq:dotdot}
    (\zeta-z_0)\cdot n(\zeta) \geq c
\end{equation}
for $\mathcal H^1$-almost every $\zeta\in\Gamma$. Here,  $n(\zeta)$ denotes the outward unit normal, and the dot denotes the Euclidean scalar product in $\mathbb R^2$. To interpret this condition geometrically, let $\psi(\zeta)$ be the angle between the radial vector $\zeta-z_0$ and  $n(\zeta)$, and set $M := \max_{\zeta\in\Gamma}|\zeta-z_0|$. The condition \eqref{eq:dotdot} implies 
$$
\psi(\zeta) \leq \arccos (c/M) < \frac{\pi}{2}
$$
for $\mathcal H^1$-almost every $\zeta\in\Gamma$. In particular, at such points, the angle between the radial direction and the tangent line is bounded below by the positive constant $\arcsin(c/M)$.

\subsection{Radial-Lipschitz curves and $A_2$ weights}
\label{subsec:radial-A2}
In this subsection, our goal is to prove that every radial-Lipschitz curve satisfies the two-sided
$A_2$ condition. 

Let
$\varphi_i:\mathbb D_i\to\Omega_i$ be a conformal map
normalized by $\varphi_i(0)=z_0$. By the Helson--Szeg\"o
theorem (see \cite[Chapter~IV]{Gar} and \cite{HMW}),  the condition
$|\varphi_i'|\in A_2$ on $\mathbb T$ is equivalent to the
representation
\begin{equation}\label{eq:HS}
  \log|\varphi_i'|=b_1+Hb_2
  \quad\text{a.e. on }\mathbb T,
  \qquad
  b_1,b_2\in L^\infty(\mathbb T),
  \qquad
  \|b_2\|_\infty<\frac{\pi}{2}. 
\end{equation}

We now verify a radial-Lipschitz curve $\Gamma$ satisfies \eqref{eq:HS}.  
 Since $\varphi_i$ extends homeomorphically to
$\overline{\mathbb D_i}$ and
$\varphi_i(z)=z_0$ only at $z=0$, the quotient $(\varphi_i(z)-z_0)/z$ 
with value $\varphi_i'(0)$ at $0$, extends to a continuous,
nowhere-vanishing function on $\overline{\mathbb D_i}$.
Consequently,
\[
  \log\left|\frac{\varphi_i(z)-z_0}{z}\right|
\]
is bounded there.  Since $\rho$ is Lipschitz and bounded below by a positive constant,  tangent geometry gives that
$$
\left| \text{Arg}\frac{z\varphi_i'(z)}               
{\varphi_i(z)-z_0}\right|\leq \arctan\left\Vert\frac{\rho'}{\rho}\right\Vert_\infty < \frac{\pi}{2}, \quad \text{a.e. on } \mathbb T.
$$
Since $\Omega_i$ is starlike with respect to $z_0$, the function
$q(z)=z\varphi_i'(z)/(\varphi_i(z)-z_0)$, with $q(0)=1$, has positive
real part in $\mathbb D_i$ (see \cite[p.66]{Pom}). Thus $\operatorname{Im}\log q$ is bounded
in the disk and is the Poisson integral of its boundary values.
Since $\log q(0)=0$, harmonic conjugation gives
$\log|q|=-H(\text{Arg}\, q)$ almost everywhere on $\mathbb T$. 
This verifies \eqref{eq:HS} and proves that
$|\varphi_i'|\in A_2$ on $\mathbb T$. 

Let $\varphi_e:\mathbb D_e\to\Omega_e$ be a conformal map
normalized by $\varphi_e(\infty)=\infty$.
Applying the above interior result to the inverted
radial-Lipschitz domain gives $|\varphi_e'|\in A_2$ on $\mathbb T$ as well; see  Section~\ref{subsec: exte}. Thus, every radial-Lipschitz curve satisfies the two-sided
$A_2$ condition.

\subsection{Sharpness of the Muckenhoupt exponent $2$}
\label{subsec:radial-A2-sharp}

For a fixed radial-Lipschitz curve, the open property of
Muckenhoupt classes gives an $\varepsilon\in(0,1)$ such that
both $|\varphi_i'|$ and $|\varphi_e'|$ belong to 
$A_{2-\varepsilon}$ on $\mathbb T$; see
\cite[Chapter~VI, Corollary~6.10(b)]{Gar}. The following elementary example
shows that $\varepsilon$ must depend on the curve.

\begin{prop}\label{prop:radial-A2-sharp}
For every integer $n\ge2$, there exists a radial-Lipschitz
curve $\Gamma_n$, with the interior domain $\Omega_i$ and a conformal map
$\varphi_n:\mathbb D_i\to \Omega_i$, such that
\[
  |\varphi_n'|\in A_2
  \setminus A_{2-1/n}\qquad \text{on} \;\;\mathbb T.
\]
Consequently, no exponent $p_0\in(1,2)$ works for all
radial-Lipschitz curves.
\end{prop}

\begin{proof}
Set $h_n=\tan(\pi/(2n))$, and let $\Gamma_n$ be the
quadrilateral obtained by joining, in order,  $ 0,\; 1+ih_n,\; -1,\; 1-ih_n,\; 0$. 
Its inward-pointing corner at $0$ has interior angle
\[
  \omega_n=2\pi-2\arctan h_n
  =\left(2-\frac1n\right)\pi;
\]
see Figure~\ref{fig:radial-A2-quadrilateral}.
The point $a=-1/2$ lies in the interior domain. For the outward unit normal $n$ on
each open edge,
\[
  (\zeta-a)\cdot n(\zeta)
  \ge\frac{h_n}{2\sqrt{4+h_n^2}}>0.
\]
Thus, $\Gamma_n$ is radial-Lipschitz.

\begin{figure}[htbp]
\centering
\begin{tikzpicture}[scale=3.1,font=\small,
  line cap=round,line join=round,
  >={Stealth[length=1.4mm,width=1.1mm]}]
  \coordinate (O) at (0,0);
  \coordinate (P) at (1,0.5773502692);
  \coordinate (Q) at (-1,0);
  \coordinate (R) at (1,-0.5773502692);
  \coordinate (A) at (-0.5,0);
  \filldraw[fill=black!7,draw=black,line width=0.9pt]
    (O)--(P)--(Q)--(R)--cycle;
  \draw[thin,->] (30:0.20) arc[start angle=30,end angle=330,radius=0.20];
  \node at (-0.105,0.025) {$\omega_n$};
  \fill (O) circle[radius=0.45pt];
  \fill (A) circle[radius=0.55pt];
  \node[right=3pt] at (O) {$0$};
  \node[above=3pt] at (A) {$a$};
  \node[above right=2pt] at (P) {$1+ih_n$};
  \node[left=3pt] at (Q) {$-1$};
  \node[below right=2pt] at (R) {$1-ih_n$};
  \node at (-0.45,0.31) {$\Gamma_n$};
\end{tikzpicture}
\caption{A radial-Lipschitz quadrilateral with radial center
$a=-1/2$ (shown for $n=3$).}
\label{fig:radial-A2-quadrilateral}
\end{figure}

Choose $\varphi_n$ so that $\varphi_n(0)=a$ and
$\varphi_n(1)=0$. The corner factor in the
Schwarz--Christoffel formula gives
\[
  |\varphi_n'(e^{it})|
  \asymp |t|^{\omega_n/\pi-1}
  =|t|^{1-1/n},\qquad t\to0;
\]
see \cite[Sections~2.2 and~4.1]{DT2002}. Put $p_n=2-1/n$.
For all sufficiently small $\eta>0$, we therefore have
\[
  \int_{-\eta}^{\eta}
  |\varphi_n'(e^{it})|^{-1/(p_n-1)}\,dt
  \asymp\int_{-\eta}^{\eta}\frac{dt}{|t|}=\infty.
\]
By \eqref{AP}, this shows $|\varphi_n'|\notin A_{p_n}$ on $\mathbb T$, whereas
Section~\ref{subsec:radial-A2} gives
$|\varphi_n'|\in A_2$ on $\mathbb T$.
Finally, given $1<p_0<2$, choose $n$ as large as
$p_0<p_n$. Since $A_{p_0}\subset A_{p_n}$, the same example
also satisfies $|\varphi_n'|\notin A_{p_0}$ on $\mathbb T$.
\end{proof}

\subsection{Sharpness of the smoothness range $(0,1/2+1/p]$}
\label{sec:sharpness}

We first prove Theorem~\ref{thm:counter}\textup{(i)}. Fix
\[
  p>2,\qquad \frac12+\frac1p<s<1,
\]
and choose
\begin{equation}\label{eq:lam}
    \frac12<\lambda<s-\frac1p.
\end{equation}
Set
\[
  \omega=\frac{\pi}{\lambda},
  \qquad \beta=2\pi-\omega,
\]
so that $\pi<\omega<2\pi$ and $0<\beta<\pi$.
We construct a radial-Lipschitz domain $\Omega_i$ whose boundary
$\Gamma$ is smooth except at a single corner of the interior angle
$\omega$, and a function $f\in B^s_{p,p}(\Gamma)$ whose interior
Poisson extension has infinite weighted energy. This shows that
the range in Corollary~\ref{thm:radial}(i) cannot, in general, be extended
beyond $s=1/2+1/p$.
\begin{figure}[htbp]
\centering
\resizebox{0.58\textwidth}{!}{%
\begin{tikzpicture}[
  font=\large,line cap=round,line join=round,
  >={Stealth[length=2mm,width=1.4mm]}
]
\begin{scope}[shift={(3.9,3.75)},scale=.76]
  \filldraw[fill=black!5,draw=black,line width=1.05pt]
    plot[smooth,tension=.3] coordinates {
(0,0) (0.02357,0) (0.04715,0) (0.07072,0) (0.09430,0)
(0.11787,0) (0.14144,0) (0.16502,0) (0.18859,0) (0.21217,0)
(0.23574,0) (0.25932,0) (0.28289,0) (0.30646,0) (0.33004,0)
(0.35361,0) (0.37719,0) (0.40076,0) (0.42433,0) (0.44791,0)
(0.47148,0) (0.49506,0) (0.51863,0) (0.54220,0) (0.56578,0)
(0.58935,0) (0.61293,0) (0.63650,0) (0.66007,0) (0.68365,0)
(0.70722,0) (0.73080,0) (0.75437,0) (0.77795,0) (0.80152,0)
(0.82509,0) (0.84867,0) (0.87224,0) (0.89582,0) (0.91939,0)
(0.94296,0) (0.96654,0) (0.99011,0) (1.01369,0) (1.03726,0)
(1.06083,0) (1.08441,0) (1.10798,0) (1.13156,0) (1.15513,0)
(1.17871,0) (1.20228,0) (1.22585,0) (1.24943,0) (1.27300,0)
(1.29658,0) (1.32015,0) (1.34372,0) (1.36730,0) (1.39087,0)
(1.41445,0) (1.43802,0) (1.46159,0) (1.48517,0) (1.50874,0)
(1.53232,0) (1.55589,0) (1.57947,0) (1.60304,0) (1.62661,0)
(1.65019,0) (1.67376,0) (1.69734,0) (1.72091,0) (1.74448,0)
(1.76806,0) (1.79163,0) (1.81521,0) (1.83878,0) (1.86235,0)
(1.88593,0) (1.90950,0) (1.93308,0) (1.95665,0) (1.98022,0)
(2.00380,0) (2.02737,0) (2.05095,0) (2.07452,0) (2.09810,0)
(2.12167,0) (2.14524,0) (2.16882,0) (2.19239,0) (2.21597,0)
(2.23954,0) (2.26311,0) (2.28669,0) (2.31026,0) (2.33384,0)
(2.35741,0) (2.38098,0) (2.40456,0) (2.42813,0) (2.45171,0)
(2.47528,0) (2.49886,0) (2.52243,0) (2.54600,0) (2.56958,0)
(2.59315,0) (2.61673,0) (2.64030,0) (2.66387,0) (2.68745,0)
(2.71102,0) (2.73460,0) (2.75817,0) (2.78174,0) (2.80532,0)
(2.82889,0) (2.85247,0) (2.87604,0) (2.89961,0) (2.92319,0)
(2.94676,0) (2.97034,0) (2.99391,0) (3.01749,0) (3.04106,0)
(3.06463,0) (3.08821,0) (3.11178,0) (3.13536,0) (3.15893,0)
(3.18250,0) (3.20608,0) (3.22965,0) (3.25323,0) (3.27680,0)
(3.30037,0) (3.32395,0) (3.34752,0) (3.37110,0) (3.39467,0)
(3.41825,0) (3.44182,0) (3.46539,0) (3.48897,0) (3.51254,0)
(3.53612,0) (3.55969,0) (3.58326,0) (3.60684,0) (3.63041,0)
(3.65399,0) (3.67756,0) (3.70113,0) (3.72471,0.00003) (3.74828,0.00022)
(3.77184,0.00079) (3.79537,0.00199) (3.81884,0.00397) (3.84223,0.00677) (3.86550,0.01040)
(3.88863,0.01485) (3.91159,0.02013) (3.93434,0.02622) (3.95686,0.03313) (3.97912,0.04083)
(4.00109,0.04932) (4.02274,0.05859) (4.04406,0.06861) (4.06501,0.07937) (4.08558,0.09085)
(4.10575,0.10303) (4.12550,0.11588) (4.14480,0.12938) (4.16365,0.14351) (4.18203,0.15824)
(4.19994,0.17355) (4.21735,0.18942) (4.23427,0.20582) (4.25068,0.22272) (4.26658,0.24011)
(4.28195,0.25796) (4.29681,0.27625) (4.31114,0.29496) (4.32494,0.31406) (4.33821,0.33353)
(4.35094,0.35335) (4.36315,0.37351) (4.37481,0.39398) (4.38595,0.41475) (4.39655,0.43579)
(4.40662,0.45710) (4.41616,0.47865) (4.42517,0.50042) (4.43365,0.52241) (4.44161,0.54459)
(4.44904,0.56695) (4.45595,0.58948) (4.46235,0.61216) (4.46822,0.63498) (4.47359,0.65793)
(4.47844,0.68099) (4.48279,0.70416) (4.48663,0.72741) (4.48997,0.75074) (4.49281,0.77413)
(4.49516,0.79758) (4.49702,0.82108) (4.49839,0.84461) (4.49928,0.86816) (4.49971,0.89172)
(4.49972,0.91529) (4.49935,0.93886) (4.49863,0.96242) (4.49760,0.98597) (4.49623,1.00950)
(4.49455,1.03301) (4.49256,1.05650) (4.49026,1.07996) (4.48764,1.10338) (4.48473,1.12677)
(4.48152,1.15013) (4.47802,1.17344) (4.47423,1.19670) (4.47015,1.21992) (4.46579,1.24308)
(4.46116,1.26620) (4.45626,1.28925) (4.45108,1.31225) (4.44565,1.33519) (4.43995,1.35806)
(4.43400,1.38087) (4.42780,1.40361) (4.42135,1.42628) (4.41466,1.44889) (4.40773,1.47142)
(4.40056,1.49387) (4.39317,1.51626) (4.38554,1.53856) (4.37769,1.56079) (4.36962,1.58293)
(4.36133,1.60500) (4.35283,1.62699) (4.34412,1.64889) (4.33520,1.67071) (4.32607,1.69245)
(4.31675,1.71410) (4.30722,1.73566) (4.29751,1.75714) (4.28760,1.77853) (4.27750,1.79983)
(4.26722,1.82104) (4.25675,1.84216) (4.24610,1.86319) (4.23527,1.88413) (4.22427,1.90498)
(4.21310,1.92573) (4.20175,1.94640) (4.19024,1.96697) (4.17856,1.98744) (4.16671,2.00783)
(4.15471,2.02811) (4.14254,2.04830) (4.13022,2.06840) (4.11774,2.08840) (4.10511,2.10830)
(4.09233,2.12811) (4.07940,2.14782) (4.06632,2.16743) (4.05309,2.18694) (4.03972,2.20636)
(4.02620,2.22567) (4.01255,2.24489) (3.99875,2.26400) (3.98482,2.28302) (3.97075,2.30193)
(3.95654,2.32074) (3.94220,2.33945) (3.92773,2.35806) (3.91313,2.37657) (3.89839,2.39497)
(3.88353,2.41326) (3.86854,2.43146) (3.85342,2.44954) (3.83817,2.46753) (3.82281,2.48540)
(3.80731,2.50317) (3.79170,2.52083) (3.77596,2.53838) (3.76011,2.55582) (3.74413,2.57316)
(3.72803,2.59038) (3.71182,2.60749) (3.69549,2.62449) (3.67904,2.64137) (3.66247,2.65815)
(3.64579,2.67480) (3.62900,2.69135) (3.61209,2.70777) (3.59506,2.72408) (3.57793,2.74027)
(3.56068,2.75634) (3.54332,2.77228) (3.52585,2.78811) (3.50827,2.80381) (3.49057,2.81939)
(3.47277,2.83484) (3.45486,2.85017) (3.43684,2.86536) (3.41871,2.88043) (3.40047,2.89536)
(3.38212,2.91016) (3.36366,2.92483) (3.34510,2.93936) (3.32643,2.95375) (3.30765,2.96800)
(3.28877,2.98211) (3.26978,2.99607) (3.25068,3.00989) (3.23147,3.02356) (3.21216,3.03708)
(3.19274,3.05045) (3.17322,3.06366) (3.15359,3.07671) (3.13386,3.08960) (3.11402,3.10233)
(3.09407,3.11490) (3.07402,3.12729) (3.05386,3.13951) (3.03360,3.15156) (3.01323,3.16343)
(2.99276,3.17512) (2.97218,3.18662) (2.95150,3.19793) (2.93072,3.20905) (2.90983,3.21997)
(2.88883,3.23069) (2.86774,3.24121) (2.84656,3.25157) (2.82533,3.26181) (2.80405,3.27197)
(2.78276,3.28208) (2.76144,3.29213) (2.74009,3.30214) (2.71873,3.31211) (2.69734,3.32202)
(2.67593,3.33189) (2.65450,3.34170) (2.63304,3.35147) (2.61157,3.36119) (2.59007,3.37086)
(2.56855,3.38049) (2.54700,3.39006) (2.52544,3.39958) (2.50385,3.40906) (2.48225,3.41849)
(2.46062,3.42786) (2.43897,3.43719) (2.41729,3.44647) (2.39560,3.45569) (2.37389,3.46487)
(2.35215,3.47400) (2.33039,3.48307) (2.30862,3.49210) (2.28682,3.50107) (2.26500,3.51000)
(2.24316,3.51887) (2.22129,3.52769) (2.19941,3.53646) (2.17751,3.54517) (2.15558,3.55384)
(2.13364,3.56245) (2.11167,3.57100) (2.08969,3.57951) (2.06768,3.58796) (2.04565,3.59636)
(2.02360,3.60470) (2.00154,3.61299) (1.97945,3.62122) (1.95734,3.62940) (1.93521,3.63753)
(1.91306,3.64560) (1.89089,3.65361) (1.86870,3.66157) (1.84648,3.66947) (1.82425,3.67731)
(1.80200,3.68509) (1.77973,3.69282) (1.75744,3.70049) (1.73513,3.70810) (1.71280,3.71566)
(1.69045,3.72315) (1.66807,3.73058) (1.64568,3.73796) (1.62327,3.74527) (1.60084,3.75253)
(1.57839,3.75972) (1.55592,3.76685) (1.53343,3.77392) (1.51092,3.78092) (1.48840,3.78787)
(1.46585,3.79475) (1.44328,3.80157) (1.42070,3.80832) (1.39809,3.81501) (1.37547,3.82163)
(1.35282,3.82819) (1.33016,3.83468) (1.30748,3.84110) (1.28478,3.84746) (1.26206,3.85375)
(1.23932,3.85997) (1.21656,3.86613) (1.19379,3.87221) (1.17099,3.87823) (1.14818,3.88417)
(1.12535,3.89005) (1.10250,3.89585) (1.07964,3.90158) (1.05675,3.90724) (1.03385,3.91282)
(1.01093,3.91834) (0.98799,3.92377) (0.96504,3.92914) (0.94206,3.93443) (0.91907,3.93964)
(0.89606,3.94477) (0.87304,3.94983) (0.85000,3.95481) (0.82694,3.95972) (0.80386,3.96454)
(0.78077,3.96928) (0.75766,3.97395) (0.73454,3.97853) (0.71140,3.98303) (0.68824,3.98745)
(0.66507,3.99178) (0.64188,3.99603) (0.61868,4.00020) (0.59546,4.00428) (0.57223,4.00827)
(0.54898,4.01218) (0.52572,4.01600) (0.50244,4.01974) (0.47915,4.02338) (0.45585,4.02693)
(0.43253,4.03040) (0.40920,4.03377) (0.38585,4.03705) (0.36250,4.04023) (0.33913,4.04333)
(0.31574,4.04633) (0.29235,4.04923) (0.26894,4.05203) (0.24553,4.05474) (0.22210,4.05736)
(0.19866,4.05987) (0.17521,4.06228) (0.15175,4.06459) (0.12828,4.06680) (0.10480,4.06891)
(0.08131,4.07092) (0.05781,4.07282) (0.03431,4.07461) (0.01079,4.07630) (-0.01273,4.07789)
(-0.03625,4.07936) (-0.05979,4.08073) (-0.08333,4.08198) (-0.10687,4.08313) (-0.13043,4.08416)
(-0.15398,4.08508) (-0.17754,4.08589) (-0.20110,4.08658) (-0.22467,4.08716) (-0.24824,4.08762)
(-0.27181,4.08796) (-0.29538,4.08818) (-0.31896,4.08829) (-0.34253,4.08827) (-0.36610,4.08813)
(-0.38968,4.08787) (-0.41325,4.08748) (-0.43682,4.08697) (-0.46038,4.08634) (-0.48394,4.08557)
(-0.50750,4.08468) (-0.53105,4.08366) (-0.55460,4.08251) (-0.57813,4.08123) (-0.60166,4.07979)
(-0.62518,4.07819) (-0.64869,4.07641) (-0.67218,4.07444) (-0.69565,4.07227) (-0.71911,4.06991)
(-0.74254,4.06736) (-0.76596,4.06462) (-0.78935,4.06170) (-0.81271,4.05859) (-0.83606,4.05530)
(-0.85937,4.05182) (-0.88266,4.04817) (-0.90592,4.04434) (-0.92915,4.04034) (-0.95235,4.03616)
(-0.97552,4.03181) (-0.99866,4.02729) (-1.02176,4.02260) (-1.04483,4.01774) (-1.06786,4.01272)
(-1.09085,4.00753) (-1.11381,4.00218) (-1.13673,3.99667) (-1.15961,3.99101) (-1.18246,3.98518)
(-1.20526,3.97919) (-1.22802,3.97306) (-1.25073,3.96676) (-1.27341,3.96032) (-1.29604,3.95373)
(-1.31863,3.94698) (-1.34117,3.94009) (-1.36367,3.93305) (-1.38613,3.92587) (-1.40853,3.91854)
(-1.43089,3.91107) (-1.45320,3.90346) (-1.47546,3.89571) (-1.49768,3.88782) (-1.51984,3.87979)
(-1.54196,3.87163) (-1.56402,3.86333) (-1.58603,3.85489) (-1.60800,3.84633) (-1.62990,3.83763)
(-1.65176,3.82880) (-1.67357,3.81984) (-1.69532,3.81075) (-1.71701,3.80153) (-1.73866,3.79218)
(-1.76024,3.78271) (-1.78177,3.77311) (-1.80325,3.76339) (-1.82467,3.75355) (-1.84603,3.74358)
(-1.86734,3.73349) (-1.88859,3.72328) (-1.90978,3.71295) (-1.93091,3.70251) (-1.95198,3.69194)
(-1.97299,3.68125) (-1.99395,3.67045) (-2.01484,3.65953) (-2.03567,3.64850) (-2.05644,3.63735)
(-2.07715,3.62608) (-2.09779,3.61470) (-2.11838,3.60321) (-2.13890,3.59161) (-2.15935,3.57989)
(-2.17975,3.56807) (-2.20007,3.55613) (-2.22034,3.54408) (-2.24053,3.53192) (-2.26066,3.51965)
(-2.28073,3.50728) (-2.30072,3.49479) (-2.32065,3.48220) (-2.34051,3.46950) (-2.36030,3.45669)
(-2.38002,3.44378) (-2.39967,3.43075) (-2.41925,3.41763) (-2.43876,3.40439) (-2.45820,3.39105)
(-2.47756,3.37761) (-2.49685,3.36406) (-2.51607,3.35041) (-2.53521,3.33665) (-2.55428,3.32279)
(-2.57327,3.30882) (-2.59218,3.29475) (-2.61102,3.28058) (-2.62978,3.26630) (-2.64846,3.25192)
(-2.66706,3.23743) (-2.68558,3.22285) (-2.70401,3.20816) (-2.72237,3.19336) (-2.74064,3.17847)
(-2.75882,3.16347) (-2.77693,3.14837) (-2.79494,3.13316) (-2.81287,3.11786) (-2.83071,3.10245)
(-2.84846,3.08693) (-2.86612,3.07132) (-2.88369,3.05560) (-2.90116,3.03978) (-2.91854,3.02385)
(-2.93583,3.00783) (-2.95302,2.99169) (-2.97011,2.97546) (-2.98711,2.95912) (-3.00400,2.94268)
(-3.02079,2.92613) (-3.03748,2.90948) (-3.05406,2.89272) (-3.07053,2.87586) (-3.08690,2.85890)
(-3.10316,2.84183) (-3.11930,2.82465) (-3.13534,2.80737) (-3.15125,2.78998) (-3.16705,2.77249)
(-3.18273,2.75488) (-3.19829,2.73717) (-3.21373,2.71936) (-3.22904,2.70143) (-3.24422,2.68340)
(-3.25927,2.66526) (-3.27419,2.64700) (-3.28897,2.62864) (-3.30362,2.61017) (-3.31812,2.59159)
(-3.33248,2.57289) (-3.34670,2.55409) (-3.36076,2.53517) (-3.37467,2.51614) (-3.38843,2.49700)
(-3.40203,2.47774) (-3.41546,2.45837) (-3.42873,2.43888) (-3.44182,2.41928) (-3.45474,2.39957)
(-3.46749,2.37974) (-3.48005,2.35979) (-3.49242,2.33973) (-3.50460,2.31954) (-3.51659,2.29925)
(-3.52837,2.27883) (-3.53995,2.25829) (-3.55131,2.23764) (-3.56245,2.21687) (-3.57338,2.19598)
(-3.58407,2.17497) (-3.59452,2.15385) (-3.60473,2.13260) (-3.61471,2.11124) (-3.62448,2.08979)
(-3.63406,2.06825) (-3.64350,2.04665) (-3.65281,2.02499) (-3.66199,2.00328) (-3.67103,1.98151)
(-3.67996,1.95969) (-3.68876,1.93782) (-3.69744,1.91590) (-3.70599,1.89394) (-3.71443,1.87193)
(-3.72275,1.84987) (-3.73096,1.82777) (-3.73905,1.80563) (-3.74703,1.78345) (-3.75489,1.76122)
(-3.76265,1.73896) (-3.77030,1.71666) (-3.77784,1.69433) (-3.78527,1.67196) (-3.79259,1.64955)
(-3.79981,1.62711) (-3.80693,1.60463) (-3.81394,1.58213) (-3.82085,1.55959) (-3.82766,1.53702)
(-3.83437,1.51442) (-3.84098,1.49179) (-3.84749,1.46914) (-3.85390,1.44645) (-3.86021,1.42374)
(-3.86643,1.40100) (-3.87255,1.37823) (-3.87857,1.35544) (-3.88450,1.33262) (-3.89033,1.30978)
(-3.89607,1.28692) (-3.90171,1.26403) (-3.90726,1.24112) (-3.91272,1.21819) (-3.91808,1.19523)
(-3.92336,1.17225) (-3.92854,1.14926) (-3.93362,1.12624) (-3.93862,1.10320) (-3.94352,1.08014)
(-3.94834,1.05706) (-3.95306,1.03397) (-3.95769,1.01085) (-3.96223,0.98772) (-3.96669,0.96457)
(-3.97105,0.94140) (-3.97532,0.91822) (-3.97950,0.89502) (-3.98359,0.87180) (-3.98759,0.84857)
(-3.99150,0.82533) (-3.99532,0.80206) (-3.99905,0.77879) (-4.00269,0.75549) (-4.00624,0.73219)
(-4.00970,0.70887) (-4.01307,0.68554) (-4.01635,0.66220) (-4.01954,0.63884) (-4.02264,0.61547)
(-4.02565,0.59209) (-4.02856,0.56869) (-4.03139,0.54529) (-4.03412,0.52188) (-4.03677,0.49845)
(-4.03932,0.47502) (-4.04178,0.45157) (-4.04415,0.42812) (-4.04642,0.40465) (-4.04860,0.38118)
(-4.05069,0.35770) (-4.05269,0.33421) (-4.05459,0.31071) (-4.05639,0.28721) (-4.05810,0.26370)
(-4.05972,0.24018) (-4.06124,0.21665) (-4.06267,0.19312) (-4.06399,0.16958) (-4.06522,0.14604)
(-4.06636,0.12250) (-4.06739,0.09895) (-4.06833,0.07539) (-4.06916,0.05183) (-4.06990,0.02827)
(-4.07054,0.00470) (-4.07107,-0.01886) (-4.07150,-0.04243) (-4.07184,-0.06601) (-4.07206,-0.08958)
(-4.07219,-0.11315) (-4.07220,-0.13673) (-4.07212,-0.16030) (-4.07192,-0.18387) (-4.07162,-0.20744)
(-4.07121,-0.23101) (-4.07070,-0.25458) (-4.07007,-0.27815) (-4.06933,-0.30171) (-4.06848,-0.32527)
(-4.06752,-0.34882) (-4.06644,-0.37237) (-4.06525,-0.39592) (-4.06395,-0.41945) (-4.06252,-0.44298)
(-4.06098,-0.46651) (-4.05932,-0.49002) (-4.05754,-0.51353) (-4.05564,-0.53703) (-4.05362,-0.56051)
(-4.05147,-0.58399) (-4.04920,-0.60745) (-4.04680,-0.63090) (-4.04428,-0.65434) (-4.04162,-0.67777)
(-4.03884,-0.70117) (-4.03592,-0.72457) (-4.03287,-0.74794) (-4.02969,-0.77130) (-4.02637,-0.79464)
(-4.02291,-0.81796) (-4.01931,-0.84125) (-4.01557,-0.86453) (-4.01169,-0.88778) (-4.00767,-0.91101)
(-4.00350,-0.93421) (-3.99918,-0.95738) (-3.99471,-0.98053) (-3.99009,-1.00365) (-3.98532,-1.02673)
(-3.98039,-1.04978) (-3.97530,-1.07280) (-3.97006,-1.09578) (-3.96465,-1.11873) (-3.95908,-1.14164)
(-3.95335,-1.16450) (-3.94745,-1.18732) (-3.94137,-1.21010) (-3.93513,-1.23283) (-3.92871,-1.25552)
(-3.92212,-1.27815) (-3.91535,-1.30073) (-3.90840,-1.32325) (-3.90128,-1.34573) (-3.89401,-1.36815)
(-3.88662,-1.39054) (-3.87913,-1.41289) (-3.87154,-1.43521) (-3.86385,-1.45749) (-3.85607,-1.47974)
(-3.84820,-1.50197) (-3.84024,-1.52416) (-3.83219,-1.54631) (-3.82406,-1.56844) (-3.81584,-1.59053)
(-3.80753,-1.61259) (-3.79914,-1.63462) (-3.79067,-1.65662) (-3.78212,-1.67859) (-3.77348,-1.70053)
(-3.76477,-1.72243) (-3.75598,-1.74431) (-3.74711,-1.76615) (-3.73816,-1.78796) (-3.72914,-1.80974)
(-3.72004,-1.83148) (-3.71086,-1.85320) (-3.70161,-1.87488) (-3.69228,-1.89653) (-3.68288,-1.91815)
(-3.67340,-1.93973) (-3.66385,-1.96129) (-3.65423,-1.98280) (-3.64453,-2.00429) (-3.63476,-2.02574)
(-3.62491,-2.04716) (-3.61499,-2.06855) (-3.60500,-2.08990) (-3.59494,-2.11122) (-3.58480,-2.13250)
(-3.57459,-2.15375) (-3.56431,-2.17496) (-3.55395,-2.19614) (-3.54352,-2.21728) (-3.53302,-2.23839)
(-3.52244,-2.25946) (-3.51179,-2.28049) (-3.50107,-2.30148) (-3.49028,-2.32244) (-3.47941,-2.34336)
(-3.46846,-2.36424) (-3.45745,-2.38508) (-3.44635,-2.40588) (-3.43519,-2.42664) (-3.42395,-2.44736)
(-3.41263,-2.46804) (-3.40124,-2.48868) (-3.38977,-2.50928) (-3.37823,-2.52983) (-3.36661,-2.55034)
(-3.35491,-2.57081) (-3.34314,-2.59124) (-3.33129,-2.61162) (-3.31936,-2.63195) (-3.30736,-2.65224)
(-3.29527,-2.67248) (-3.28311,-2.69267) (-3.27086,-2.71281) (-3.25854,-2.73291) (-3.24613,-2.75296)
(-3.23365,-2.77295) (-3.22108,-2.79290) (-3.20843,-2.81279) (-3.19570,-2.83263) (-3.18288,-2.85241)
(-3.16998,-2.87214) (-3.15700,-2.89182) (-3.14393,-2.91144) (-3.13077,-2.93100) (-3.11753,-2.95050)
(-3.10420,-2.96995) (-3.09078,-2.98933) (-3.07727,-3.00865) (-3.06368,-3.02791) (-3.04999,-3.04710)
(-3.03621,-3.06623) (-3.02235,-3.08529) (-3.00839,-3.10429) (-2.99433,-3.12321) (-2.98018,-3.14207)
(-2.96594,-3.16085) (-2.95160,-3.17957) (-2.93717,-3.19820) (-2.92263,-3.21676) (-2.90800,-3.23525)
(-2.89327,-3.25365) (-2.87844,-3.27198) (-2.86351,-3.29022) (-2.84847,-3.30837) (-2.83334,-3.32645)
(-2.81809,-3.34443) (-2.80275,-3.36232) (-2.78730,-3.38013) (-2.77174,-3.39784) (-2.75607,-3.41545)
(-2.74029,-3.43296) (-2.72440,-3.45038) (-2.70841,-3.46769) (-2.69230,-3.48490) (-2.67607,-3.50201)
(-2.65973,-3.51900) (-2.64328,-3.53588) (-2.62671,-3.55265) (-2.61002,-3.56930) (-2.59321,-3.58583)
(-2.57629,-3.60223) (-2.55924,-3.61852) (-2.54207,-3.63467) (-2.52478,-3.65069) (-2.50736,-3.66657)
(-2.48982,-3.68232) (-2.47215,-3.69793) (-2.45436,-3.71339) (-2.43643,-3.72870) (-2.41838,-3.74386)
(-2.40020,-3.75886) (-2.38189,-3.77370) (-2.36344,-3.78838) (-2.34486,-3.80289) (-2.32615,-3.81723)
(-2.30731,-3.83139) (-2.28833,-3.84537) (-2.26921,-3.85916) (-2.24996,-3.87277) (-2.23057,-3.88617)
(-2.21105,-3.89938) (-2.19138,-3.91238) (-2.17158,-3.92518) (-2.15165,-3.93775) (-2.13157,-3.95011)
(-2.11136,-3.96223) (-2.09101,-3.97413) (-2.07052,-3.98579) (-2.04990,-3.99721) (-2.02914,-4.00837)
(-2.00825,-4.01928) (-1.98722,-4.02994) (-1.96606,-4.04032) (-1.94476,-4.05043) (-1.92334,-4.06026)
(-1.90179,-4.06981) (-1.88011,-4.07907) (-1.85831,-4.08802) (-1.83638,-4.09668) (-1.81434,-4.10502)
(-1.79218,-4.11305) (-1.76990,-4.12076) (-1.74752,-4.12814) (-1.72503,-4.13519) (-1.70243,-4.14190)
(-1.67974,-4.14827) (-1.65695,-4.15428) (-1.63407,-4.15994) (-1.61110,-4.16524) (-1.58805,-4.17018)
(-1.56493,-4.17475) (-1.54173,-4.17895) (-1.51848,-4.18276) (-1.49516,-4.18620) (-1.47179,-4.18926)
(-1.44837,-4.19193) (-1.42491,-4.19421) (-1.40142,-4.19611) (-1.37789,-4.19761) (-1.35435,-4.19871)
(-1.33079,-4.19943) (-1.30722,-4.19981) (-1.28365,-4.19995) (-1.26008,-4.19995) (-1.23650,-4.19990)
(-1.21293,-4.19982) (-1.18935,-4.19972) (-1.16578,-4.19958) (-1.14221,-4.19940) (-1.11863,-4.19919)
(-1.09506,-4.19894) (-1.07149,-4.19864) (-1.04792,-4.19830) (-1.02435,-4.19791) (-1.00078,-4.19746)
(-0.97721,-4.19695) (-0.95364,-4.19638) (-0.93008,-4.19574) (-0.90651,-4.19503) (-0.88295,-4.19424)
(-0.85940,-4.19336) (-0.83584,-4.19239) (-0.81229,-4.19133) (-0.78875,-4.19016) (-0.76521,-4.18888)
(-0.74168,-4.18748) (-0.71815,-4.18595) (-0.69464,-4.18428) (-0.67113,-4.18246) (-0.64764,-4.18048)
(-0.62417,-4.17832) (-0.60071,-4.17598) (-0.57728,-4.17343) (-0.55387,-4.17067) (-0.53049,-4.16767)
(-0.50714,-4.16441) (-0.48384,-4.16088) (-0.46058,-4.15704) (-0.43738,-4.15287) (-0.41425,-4.14835)
(-0.39120,-4.14343) (-0.36824,-4.13809) (-0.34540,-4.13227) (-0.32270,-4.12595) (-0.30017,-4.11906)
(-0.27783,-4.11156) (-0.25574,-4.10338) (-0.23394,-4.09446) (-0.21250,-4.08472) (-0.19148,-4.07411)
(-0.17099,-4.06253) (-0.15112,-4.04993) (-0.13199,-4.03623) (-0.11375,-4.02140) (-0.09654,-4.00539)
(-0.08051,-3.98821) (-0.06580,-3.96990) (-0.05254,-3.95053) (-0.04080,-3.93018) (-0.03065,-3.90900)
(-0.02208,-3.88713) (-0.01508,-3.86469) (-0.00958,-3.84183) (-0.00548,-3.81867) (-0.00269,-3.79531)
(-0.00105,-3.77182) (-0.00028,-3.74828) (-0.00003,-3.72471) (0,-3.70113) (0,-3.67756)
(0,-3.65399) (0,-3.63041) (0,-3.60684) (0,-3.58326) (0,-3.55969)
(0,-3.53612) (0,-3.51254) (0,-3.48897) (0,-3.46539) (0,-3.44182)
(0,-3.41825) (0,-3.39467) (0,-3.37110) (0,-3.34752) (0,-3.32395)
(0,-3.30037) (0,-3.27680) (0,-3.25323) (0,-3.22965) (0,-3.20608)
(0,-3.18250) (0,-3.15893) (0,-3.13536) (0,-3.11178) (0,-3.08821)
(0,-3.06463) (0,-3.04106) (0,-3.01749) (0,-2.99391) (0,-2.97034)
(0,-2.94676) (0,-2.92319) (0,-2.89961) (0,-2.87604) (0,-2.85247)
(0,-2.82889) (0,-2.80532) (0,-2.78174) (0,-2.75817) (0,-2.73460)
(0,-2.71102) (0,-2.68745) (0,-2.66387) (0,-2.64030) (0,-2.61673)
(0,-2.59315) (0,-2.56958) (0,-2.54600) (0,-2.52243) (0,-2.49886)
(0,-2.47528) (0,-2.45171) (0,-2.42813) (0,-2.40456) (0,-2.38098)
(0,-2.35741) (0,-2.33384) (0,-2.31026) (0,-2.28669) (0,-2.26311)
(0,-2.23954) (0,-2.21597) (0,-2.19239) (0,-2.16882) (0,-2.14524)
(0,-2.12167) (0,-2.09810) (0,-2.07452) (0,-2.05095) (0,-2.02737)
(0,-2.00380) (0,-1.98022) (0,-1.95665) (0,-1.93308) (0,-1.90950)
(0,-1.88593) (0,-1.86235) (0,-1.83878) (0,-1.81521) (0,-1.79163)
(0,-1.76806) (0,-1.74448) (0,-1.72091) (0,-1.69734) (0,-1.67376)
(0,-1.65019) (0,-1.62661) (0,-1.60304) (0,-1.57947) (0,-1.55589)
(0,-1.53232) (0,-1.50874) (0,-1.48517) (0,-1.46159) (0,-1.43802)
(0,-1.41445) (0,-1.39087) (0,-1.36730) (0,-1.34372) (0,-1.32015)
(0,-1.29658) (0,-1.27300) (0,-1.24943) (0,-1.22585) (0,-1.20228)
(0,-1.17871) (0,-1.15513) (0,-1.13156) (0,-1.10798) (0,-1.08441)
(0,-1.06083) (0,-1.03726) (0,-1.01369) (0,-0.99011) (0,-0.96654)
(0,-0.94296) (0,-0.91939) (0,-0.89582) (0,-0.87224) (0,-0.84867)
(0,-0.82509) (0,-0.80152) (0,-0.77795) (0,-0.75437) (0,-0.73080)
(0,-0.70722) (0,-0.68365) (0,-0.66007) (0,-0.63650) (0,-0.61293)
(0,-0.58935) (0,-0.56578) (0,-0.54220) (0,-0.51863) (0,-0.49506)
(0,-0.47148) (0,-0.44791) (0,-0.42433) (0,-0.40076) (0,-0.37719)
(0,-0.35361) (0,-0.33004) (0,-0.30646) (0,-0.28289) (0,-0.25932)
(0,-0.23574) (0,-0.21217) (0,-0.18859) (0,-0.16502) (0,-0.14144)
(0,-0.11787) (0,-0.09430) (0,-0.07072) (0,-0.04715) (0,-0.02357)
(0,0)
    } -- cycle;
  \fill[black!17] (0,0)--(3.67,0)
    arc[start angle=0,end angle=270,radius=3.67]--cycle;
  \draw[black!60,dashed,line width=.65pt]
    (3.67,0) arc[start angle=0,end angle=270,radius=3.67];

  \draw[black,line width=1.05pt] (3.68,0)--(0,0)--(0,-3.68);
  \fill (-1.1,1.1) circle[radius=1.7pt];
  \node[anchor=south east,inner sep=3pt] at (-1.1,1.1) {$a$};
  \fill (0,0) circle[radius=1.7pt];
  \node[anchor=north east,inner sep=3pt] at (0,0) {$0$};
  \node at (-1.3,2.78) {$\Omega_i$};
  \node[font=\normalsize,align=center] at (-2.03,-1.2)
    {$\Omega_i\cap B(0,r_0)$};
  \node at (2.08,-1.5) {$\Omega_e$};
  \node[anchor=south west] at (3.32,3.25) {$\Gamma$};

  \node[font=\normalsize,anchor=south] at (2.28,.09)
    {$\theta=0$};
  \node[font=\normalsize,anchor=east] at (-.10,-2.45)
    {$\theta=\omega$};
  \draw[line width=.6pt,->]
    (.72,0) arc[start angle=0,end angle=270,radius=.72];
  \node[font=\normalsize] at (-1.02,.35) {$\omega$};
  \draw[line width=.5pt,->]
    (0,-.48) arc[start angle=270,end angle=360,radius=.48];
  \node[font=\normalsize] at (.82,-.70) {$\beta$};
\end{scope}
\end{tikzpicture}
}
\caption{A one-corner radial-Lipschitz domain, illustrated for
$\omega=3\pi/2$.}
\label{fig:corner-domain}
\end{figure}


\textbf{Geometry of the corner}. 
Choose $r_0>0$ and a point $a=\rho e^{i\omega/2},\; 0<\rho<r_0.$
We construct a bounded Jordan domain $\Omega_i$, contained in the
sector $0<\arg z<\omega$, with the following properties:
\begin{enumerate}
\item[(1)] Near the origin, $\Omega_i$ coincides with the sector:
\[
  \Omega_i\cap B(0,r_0)
  =\{re^{i\theta}:0<r<r_0,\ 0<\theta<\omega\}.
\]
\item[(2)] Its boundary $\Gamma=\partial\Omega_i$ is  $C^\infty$-smooth except
at the single corner $0$.
\item[(3)] The domain $\Omega_i$ is uniformly strictly star-shaped
with respect to $a$: there exists a constant $c>0$ such that $(\zeta-a)\cdot n(\zeta)\ge c$ for $\zeta\in\Gamma\setminus\{0\}$. 
\end{enumerate}
See Figure~\ref{fig:corner-domain} for an illustration.
For background on domains with corners; see
\cite[Preface, pp.~x--xi]{Gri}.

To ensure property~(3), observe first that on both straight sides
near $0$,
\begin{equation}\label{eq:straight}
  (\zeta-a)\cdot n(\zeta)
  =\rho\sin(\omega/2)>0.  
\end{equation}
We choose the remaining compact smooth arc $\Gamma_1$ so that
$(\zeta-a)\cdot n(\zeta)>0$ at every point of $\Gamma_1$.
By continuity and compactness, $c_1:=\min_{\zeta\in\Gamma_1}
  (\zeta-a)\cdot n(\zeta)>0.$ 
Combining this bound with \eqref{eq:straight} yields property~(3).

\textbf{The singular interior extension}. 
Since $\Omega_i$ lies in the sector $0<\arg z<\omega$, the branch
$\log z=\log r+i\theta$, $0<\theta<\omega$, is holomorphic throughout
$\Omega_i$. Define
\begin{equation}\label{eq:corner-function}
  u_i(z)=\operatorname{Im}(z^\lambda)
        =r^\lambda\sin(\lambda\theta),
  \qquad z=re^{i\theta}\in\Omega_i.
\end{equation}
 The function $u_i$ is harmonic in $\Omega_i$,
continuous on $\overline{\Omega_i}$ after setting $u_i(0)=0$, and
vanishes on the two straight sides. Moreover,
\begin{equation}\label{eq:corner-gradient}
  |\nabla u_i(z)|=|(z^\lambda)'|
  =\lambda |z|^{\lambda-1}.
\end{equation}
Put $f=u_i|_\Gamma$. Uniqueness of the continuous Dirichlet problem
gives $u_i=\mathcal P_{\Omega_i}f$.

For sufficiently small $\epsilon>0$, the sector geometry gives
\begin{equation}\label{eq:momega}
    \delta(re^{i\theta})\asymp r\,m_\omega(\theta),
  \qquad
  m_\omega(\theta)=\min\{1,\theta,\omega-\theta\},
  \qquad 0<r<\epsilon.
\end{equation}
Consequently, by \eqref{eq:corner-gradient},
\begin{align}
  \int_{\Omega_i\cap B(0,\epsilon)}
       |\nabla u_i(z)|^p\delta(z)^{(1-s)p-1}\,dA(z)
  &\asymp
  \int_0^\epsilon\int_0^\omega
       r^{p(\lambda-1)}\bigl(rm_\omega(\theta)\bigr)^{(1-s)p-1}
       r\,d\theta\,dr\notag\\
  &=
  \left(\int_0^\omega m_\omega(\theta)^{(1-s)p-1}\,d\theta\right)
  \int_0^\epsilon r^{p(\lambda-s)}\,dr.\label{eq:sector1}
\end{align}
The angular integral is  positive and finite because both endpoints are integrable when  $(1-s)p-1 >-1$.
The radial integral is finite precisely when
\begin{equation}\label{eq:interior-threshold}
  p(\lambda-s)>-1
  \quad\Longleftrightarrow\quad
  s<\lambda+\frac1p;
\end{equation}
at equality it diverges logarithmically. Our choice
$\lambda<s-1/p$ (see \eqref{eq:lam}) therefore yields
\[
  \mathcal E_{s,p;\Omega_i}(\mathcal P_{\Omega_i}f)=\infty.
\]

\textbf{Regularity of the boundary data}. 
Although $\nabla u_i$ is singular at the corner, the boundary function
$f$ has no corresponding singularity. Indeed,
\[
  f(r)=r^\lambda\sin0=0,\qquad
  f(re^{i\omega})=r^\lambda\sin(\lambda\omega)=0,
\]
and hence $f\equiv0$ on $\Gamma\cap B(0,r_0)$. Away from $0$, both
$\Gamma$ and $z^\lambda$ are smooth.

To estimate the Besov seminorm, let
$\ell=\operatorname{length}(\Gamma)$, let
$\gamma:\mathbb R/\ell\mathbb Z\to\Gamma$ be an arclength
parametrization, and set $g=f\circ\gamma$.
Since $g$ vanishes in a neighborhood of the parameter corresponding
to the corner and is smooth elsewhere, it is a smooth periodic
function. In particular, $g$ is Lipschitz with $ L_g:=\|g'\|_{L^\infty(\mathbb R/\ell\mathbb Z)}<\infty.$

For $t,\tau\in[0,\ell]$, put
\[
  d_\ell(t,\tau)=\min\{|t-\tau|,\ell-|t-\tau|\}.
\]
The chord-arc condition and the Lipschitz continuity of $g$ imply
\[
  d_\ell(t,\tau)\le K|\gamma(t)-\gamma(\tau)|,
  \qquad
  |g(t)-g(\tau)|\le L_g\,d_\ell(t,\tau).
\]
It follows that
\begin{align*}
  [f]_{\dot B^s_{p,p}(\Gamma)}^p
  &=
  \int_0^\ell\int_0^\ell
  \frac{|g(t)-g(\tau)|^p}
       {|\gamma(t)-\gamma(\tau)|^{1+sp}}\,dt\,d\tau\\
  &\le
  K^{1+sp}L_g^p
  \int_0^\ell\int_0^\ell
       d_\ell(t,\tau)^{p(1-s)-1}\,dt\,d\tau\\
  &=
  \frac{2K^{1+sp}L_g^p\ell}{p(1-s)}
  \left(\frac{\ell}{2}\right)^{p(1-s)}
  <\infty.
\end{align*}
Since $f$ is continuous on the compact curve $\Gamma$, it also belongs
to $L^p(\Gamma)$. Thus, $f\in B^s_{p,p}(\Gamma)$. This proves Theorem~\ref{thm:counter}\textup{(i)}.

\begingroup

\subsection{Sharpness of the lower smoothness bound}
\label{subsec:lower-smoothness-example}
We prove Theorem~\ref{thm:counter}\textup{(ii)}. Fix \(1<p<2\)
and \(0<s<1/p-1/2\), and set
\[
 \lambda=\frac1p-s,\qquad \omega=\frac{\pi}{\lambda}.
\]
Then \(1/2<\lambda<1\) and \(\pi<\omega<2\pi\).
Choose \(\Omega_i\) as in Section~\ref{sec:sharpness}; see
Figure~\ref{fig:corner-domain}. After a dilation about the corner
\(0\), we may assume that \(\overline{\Omega_i}\subset\mathbb D\).
This dilation preserves the opening angle \(\omega\), so for some
\(r_0>0\),
\[
 \Omega_i\cap B(0,r_0)
 =\{re^{i\theta}:0<r<r_0,\ 0<\theta<\omega\}.
\]

\medskip
\noindent\textbf{The singular interior extension.}
Use the branch \(\log z=\log r+i\theta\), \(0<\theta<\omega\),
and put \(L(r)=1+\log(1/r)\). Define
\begin{equation}\label{eq:explicit-lower-data}
 u_i(z)=-\operatorname{Im}\frac{z^{-\lambda}}{\sqrt{1-\log z}},\quad z\in\Omega_i;
 \qquad f=u_i|_{\Gamma\setminus\{0\}},\quad f(0)=0 .
\end{equation}
Since \(\operatorname{Re}(1-\log z)=L(r)\ge1\), we see that $1-\log z$ stays in the right half-plane where the principal square root is holomorphic and never zero, and thus it gives a harmonic \(u_i\). In polar coordinates,
\[
 u_i(re^{i\theta})
 =\frac{\sin\!\left(\lambda\theta-\frac12
                 \arctan\frac{\theta}{L(r)}\right)}
        {r^\lambda(L(r)^2+\theta^2)^{1/4}}.
\]
The sine is positive  because
\(\tfrac12\arctan(\theta/L(r))\le\theta/2<\lambda\theta<\pi\).
As in \eqref{eq:corner-gradient}, the gradient is computed by
differentiating the holomorphic function in \eqref{eq:explicit-lower-data}:
\[
 \left(\frac{z^{-\lambda}}{\sqrt{1-\log z}}\right)'
 =\frac{z^{-\lambda-1}}{\sqrt{1-\log z}}
   \left(-\lambda+\frac{1}{2(1-\log z)}\right),
 \qquad |\nabla u_i(z)|\asymp r^{-\lambda-1}L(r)^{-1/2}
\]
near \(0\). The same sector geometry as in \eqref{eq:momega}
gives \(\delta(re^{i\theta})\asymp r m_\omega(\theta)\). Thus,
for sufficiently small \(\epsilon>0\),
\begin{equation*}
 \int_{\Omega_i\cap B(0,\epsilon)}
       |\nabla u_i|^p\delta^{(1-s)p-1}\,dA
 \asymp
 \left(\int_0^\omega m_\omega(\theta)^{(1-s)p-1}\,d\theta\right)
 \int_0^\epsilon r^{-p(\lambda+s)}L(r)^{-p/2}\,dr.
\end{equation*}
The angular integral on the right-hand side is positive and finite as we observed above, and
$$
\int_0^\epsilon r^{-p(\lambda+s)}L(r)^{-p/2}\,dr = \int_0^\epsilon\frac{dr}{rL(r)^{p/2}} = \int_{L(\epsilon)}^\infty t^{-p/2}\,dt
 =\infty
$$
since \(p(\lambda+s)=1\) and \(p/2<1\).

To identify \(u_i\) with \(\mathcal P_{\Omega_i}f\), take a
Riemann map \(\varphi:\mathbb D_i\to\Omega_i\) with \(\varphi(1)=0\).
The power map \(z\mapsto z^\lambda\) sends \(\Omega_i\) into the upper
half-plane. It maps the sector \(\Omega_i\cap B(0,r_0)\) onto
the upper half-disk of radius \(r_0^\lambda\), with the two straight
sides mapping to the real diameter.
Thus \(F(\zeta):=(\varphi(\zeta))^\lambda\) extends holomorphically
across the unit-circle arc near \(1\) by Schwarz reflection.
The extension is locally one-to-one: the original and reflected
sides map into opposite half-planes, and the arc itself maps
one-to-one onto a real interval. Hence \(F'(1)\ne0\).
Since \(F(1)=0\), Taylor expansion gives
\[
 |\varphi(\zeta)|^\lambda=|F(\zeta)|\asymp|1-\zeta|
 \qquad\text{near }1.
\]

\begin{figure}[!htbp]
\centering
\color{black}
\resizebox{0.98\linewidth}{!}{%
\begin{tikzpicture}[
  x=1pt,y=1pt,yscale=-1,
  font=\fontsize{22}{25}\selectfont,
  line cap=round,line join=round,
  >={Stealth[length=11pt,width=9pt]},
  domain outline/.style={fill=black!9,draw=black,line width=2.6pt},
  reflected patch/.style={fill=black!25,draw=black,line width=1.7pt,
    dash pattern=on 6pt off 5pt},
  approach path/.style={->,line width=1.8pt,dash pattern=on 4pt off 4pt},
  reflection path/.style={line width=1.8pt,dash pattern=on 6pt off 5pt}
]
  \path[reflected patch]
    (226,139) -- (268.4,107.2)
    arc[start angle=-36.87,end angle=36.87,radius=143pt] --
    (226,247) arc[start angle=36.87,end angle=-36.87,radius=90pt] -- cycle;
  \path[reflected patch]
    (851,194) arc[start angle=180,end angle=0,radius=67pt] -- cycle;
  \path[domain outline] (154,193) circle[radius=90pt];
  \path[domain outline]
    (554,194) -- (630,194)
    .. controls (651,194) and (658,162) .. (632,135)
    .. controls (611,113) and (581,103) .. (551,103)
    .. controls (500,103) and (468,141) .. (468,194)
    .. controls (468,238) and (493,284) .. (531,284)
    .. controls (548,284) and (554,283) .. (554,270) -- cycle;
  \path[domain outline]
    (837,194) .. controls (820,194) and (813,179) .. (816,160)
    .. controls (823,115) and (867,96) .. (918,96)
    .. controls (969,96) and (1013,115) .. (1020,160)
    .. controls (1023,179) and (1016,194) .. (999,194) -- cycle;
  \path[fill=black!17]
    (226,139) arc[start angle=-36.87,end angle=36.87,radius=90pt]
    arc[start angle=126.87,end angle=233.13,radius=67.5pt] -- cycle;
  \path[fill=black!17]
    (554,194) -- (616,194) arc[start angle=0,end angle=-270,radius=62pt] -- cycle;
  \path[fill=black!17]
    (851,194) arc[start angle=180,end angle=360,radius=67pt] -- cycle;
  \draw[reflection path]
    (226,247) arc[start angle=126.87,end angle=233.13,radius=67.5pt]
    (616,194) arc[start angle=0,end angle=-270,radius=62pt]
    (851,194) arc[start angle=180,end angle=360,radius=67pt];
  \draw[line width=2.6pt]
    (226,139) arc[start angle=-36.87,end angle=36.87,radius=90pt]
    (616,194) -- (554,194) -- (554,256)
    (851,194) -- (985,194);

  \node[font=\fontsize{23}{25}\selectfont\bfseries] at (154,75) {Unit disk};
  \node[font=\fontsize{23}{25}\selectfont\bfseries] at (554,75) {One-corner domain \(\Omega_i\)};
  \node[font=\fontsize{23}{25}\selectfont\bfseries] at (918,70) {Straightened domain};
  \node at (117,228) {$\mathbb D_i$};

  \draw[->,line width=2.4pt] (316,144) -- (440,144)
    node[midway,above=6pt] {$\varphi$};
  \draw[->,line width=2.4pt] (668,144) -- (806,144)
    node[midway,above=6pt] {$z\mapsto z^\lambda$};
  \draw[approach path] (206,193) -- (234,193);
  \draw[approach path] (513,153) -- (548,188);
  \draw[approach path] (918,136) -- (918,184);
  \draw[reflection path] (250,193) -- (284,193);
  \draw[reflection path] (918,200) -- (918,239);
  \draw[reflection path,->]
    (285,202) .. controls (294,268) and (324,323) .. (387,323) --
    (825,323) .. controls (885,323) and (907,285) .. (916,250);

  \foreach \markpos in {(216.1,193),(530,170),(918,149),(284.4,193),(918,239)}
    \fill[black] \markpos circle[radius=4.4pt];
  \foreach \markpos in {(244,193),(554,194),(918,194)}
    \fill[black] \markpos circle[radius=5.7pt];

  \node at (215,172) {$\rho$};
  \node[anchor=west] at (528,147) {$\varphi(\rho)$};
  \node[anchor=south,inner sep=0pt] at (918,121) {$w=F(\rho)$};
  \node[anchor=west] at (301,193) {$\rho^*=1/\rho$};
  \node[anchor=west] at (928,239) {$\overline w$};
  \node at (607,354) {$\widetilde F(\rho^*)=\overline{F(\rho)}$};
  \node at (254,217) {$1$};
  \node at (571,215) {$0$};
  \node at (892,215) {$0$};
\end{tikzpicture}%
}
\caption{The dashed arc in \(\mathbb D_i\) is the preimage under \(\varphi\)
of \(\Omega_i\cap\partial B(0,\epsilon)\), where \(0<\epsilon<r_0\).
The power map sends the middle arc to the upper semicircle
\(|w|=\epsilon^\lambda\). The darkest shading marks the reflected sides.}
\label{fig:lower-corner-reflection}
\end{figure}

Set \(v=u_i\circ\varphi\) and \(g=f\circ\varphi\), and let \(dm\)
denote normalized arclength measure on \(\mathbb T\). By the Herglotz
representation theorem, the positive harmonic function \(v\) is
the Poisson integral of a unique finite positive measure \(\mu\)
on \(\mathbb T\) \cite[Corollary~6.15]{AxlerBourdonRamey}.
Moreover, the measures \(v(R\xi)\,dm(\xi)\) converge weakly to
\(\mu\) as \(R\uparrow1\) \cite[Theorem~6.9(a)]{AxlerBourdonRamey}; 
that is, for every continuous \(h\) on \(\mathbb T\),
\[
\lim_{R\uparrow1}\int_{\mathbb T}h(\xi)v(R\xi)\,dm(\xi)
=\int_{\mathbb T}h(\xi)\,d\mu(\xi).
\]

On each closed arc away from \(1\), \(v(R\xi)\to g(\xi)\)
uniformly. Thus \(d\mu=g\,dm\) away from \(1\), and the only
possible extra mass is \(c=\mu(\{1\})\ge0\). In particular,
\(g\in L^1(\mathbb T)\), and
\[
 d\mu=g\,dm+c\,d\delta_1,\qquad
 v(\zeta)=\mathcal P_{\mathbb D_i}g(\zeta)
          +c\,\frac{1-|\zeta|^2}{|1-\zeta|^2},
\]
where \(\delta_1\) is the unit point mass at \(1\).
Since \(g\ge0\), along the radius \(\zeta=\rho\uparrow1\)
this gives
\[
 (1-\rho)v(\rho)\ge c(1+\rho).
\]
On the other hand, put \(r=|\varphi(\rho)|\).
By \eqref{eq:explicit-lower-data} and
\(|1-\log z|\ge L(|z|)\), we have
\[
 0\le (1-\rho)v(\rho)
 \le \frac{(1-\rho)r^{-\lambda}}{\sqrt{L(r)}}.
\]
Since \(r^\lambda\asymp1-\rho\), the numerator is bounded,
while
\[
 L(r)=1+\log(1/r)
     =\frac{1}{\lambda}\log\frac{1}{1-\rho}+O(1)
     \longrightarrow\infty.
\]
Therefore \((1-\rho)v(\rho)\to0\), which forces \(c=0\).
Thus \(v=\mathcal P_{\mathbb D_i}g\) so that \(u_i=\mathcal P_{\Omega_i}f\). This completes the proof of $\mathcal E_{s,p;\Omega_i}(\mathcal P_{\Omega_i}f)=\infty$. 

\medskip
\noindent\textbf{Regularity of the boundary data.}
On the two straight sides near \(0\), the identity
\(\lambda\omega=\pi\) gives
\[
 f(r)=0,\qquad f(re^{i\omega})=b(r)
 =\frac{\sin\!\left(\frac12\arctan\frac{\omega}{L(r)}\right)}
        {r^\lambda(L(r)^2+\omega^2)^{1/4}}.
\]
The sine factor is comparable to \(L(r)^{-1}\). Differentiating
with \(rL'(r)=-1\) gives
\[
 b(r)\asymp r^{-\lambda}L(r)^{-3/2},\qquad
 r b'(r)=b(r)\left(
 -\lambda+\frac{2L(r)+\sqrt{L(r)^2+\omega^2}}
                 {2(L(r)^2+\omega^2)}\right).
\]
In particular, \(|b'(r)|\lesssim r^{-\lambda-1}L(r)^{-3/2}\). These estimates will be used later.

Fix \(0<\epsilon<r_0\) and write
\(\Gamma_\epsilon=\Gamma\cap B(0,\epsilon)\). We estimate
\[
 [f]_{\dot B^s_{p,p}(\Gamma)}^p
 =\iint_{\Gamma\times\Gamma}
   \frac{|f(z)-f(w)|^p}{|z-w|^{1+sp}}\,|dz|\,|dw|
\]
in three cases.

\medskip

\noindent\textbf{Case I: Both points lie in \(\Gamma_\epsilon\).} In this case, the estimate requires a lot of work. There are two reasons:  the boundary
trace \(f\) vanishes on the ray
\(\theta=0\) but \(f(re^{i\omega})=b(r)\to+\infty\) as \(r\downarrow0\)
on the other ray, and  \(|z-w|^{-1-sp}\)
becomes singular as  \(z\) and \(w\)
approach each other.

Take two points $r,\, t$ on the nonzero ray and assume $r<t$. First, suppose  \(r<t<2r\).
Putting \(h=t-r\), we have
\[
 |b(t)-b(r)|\le \int_r^t |b'(x)|\,dx\lesssim h\,r^{-\lambda-1}L(r)^{-3/2}
\]
since \(x\asymp r\) and \(L(x)\asymp L(r)\) throughout \([r,t]\).
The contribution of these pairs 
is therefore bounded by
\[
 C\int_0^\epsilon r^{-p(\lambda+1)}L(r)^{-3p/2}
       \int_0^r h^{p(1-s)-1}\,dh\,dr
 \lesssim \int_0^\epsilon\frac{dr}{rL(r)^{3p/2}},
\]
since \(p(\lambda+s)=1\).  Next, suppose 
\(t\geq 2r\), i.e., \(t-r\ge t/2\). The contribution of these pairs is bounded by
\[
C\iint_{0<2r\le t<\epsilon}
\frac{|b(t)|^p}{t^{1+sp}}\,dr\,dt + C\iint_{0<2r\le t<\epsilon}
\frac{|b(r)|^p}{t^{1+sp}}\,dr\,dt.
\]
Then, first integrating the variable that does not appear in $b$  gives a bound by
\(\int_0^\epsilon |b(r)|^p r^{-sp}\,dr\). Pairs on the zero ray contribute nothing. For points on different rays,
\[
 |re^{i\omega}-t|\ge \sin(\omega/2)(r+t).
\]
 These pairs contribute at most
\[
 C\int_0^\epsilon\!\!\int_0^\epsilon
 \frac{|b(r)|^p+|b(t)|^p}{\max\{r,t\}^{1+sp}}\,dt\,dr
 \lesssim \int_0^\epsilon |b(r)|^p r^{-sp}\,dr
 \lesssim \int_0^\epsilon\frac{dr}{rL(r)^{3p/2}}.
\]
Here we used 
\[
 \int_0^\epsilon\frac{dt}{\max\{r,t\}^{1+sp}}
 =r^{-sp}+\int_r^\epsilon t^{-1-sp}\,dt
 \lesssim r^{-sp}.
\]
The common bound is finite because \(3p/2>1\).

The finiteness of this local seminorm also gives local
\(L^p\)-integrability. Indeed, by Fubini we may choose \(w_0\)
on the zero ray in \(\Gamma_\epsilon\) for which the inner
Besov integral is finite. Then 
\[
 \int_{\Gamma_\epsilon}|f(z)|^p\,|dz|
 \le (2\epsilon)^{1+sp}
 \int_{\Gamma_\epsilon}
 \frac{|f(z)-f(w_0)|^p}{|z-w_0|^{1+sp}}\,|dz|
 <\infty.
\]
On the other hand, the trace is bounded on \(\Gamma\setminus\Gamma_\epsilon\). Thus,  \(f\in L^p(\Gamma)\).

\medskip
\noindent\textbf{Case II: Both points lie outside \(B(0,\epsilon/2)\).}
The trace is Lipschitz there, so the integrand is bounded by
a constant times \(|z-w|^{p(1-s)-1}\). By the chord-arc
condition, writing \(\ell=\operatorname{length}(\Gamma)\), we get
\[
 \iint_{\Gamma\times\Gamma}
 |z-w|^{p(1-s)-1}\,|dz|\,|dw|
 \lesssim 2\ell\int_0^{\ell/2}h^{p(1-s)-1}\,dh<\infty.
\]

\medskip
\noindent\textbf{Case III: One point lies in \(B(0,\epsilon/2)\)
and the other lies outside \(B(0,\epsilon)\).}
These are the pairs not covered by Cases I and II. Their
distance is at least \(\epsilon/2\). Thus, the contribution is bounded by
\[
 2\left(\frac{2}{\epsilon}\right)^{1+sp}
 \int_{\Gamma_{\epsilon/2}}\!
 \int_{\Gamma\setminus\Gamma_\epsilon}
 |f(z)-f(w)|^p\,|dw|\,|dz|
 \lesssim \|f\|_{L^p(\Gamma)}^p<\infty,
\]
using the \(L^p\) bound obtained in Case I. Thus, \([f]_{\dot B^s_{p,p}(\Gamma)}<\infty\), and hence
\(f\in B^s_{p,p}(\Gamma)\). This proves
Theorem~\ref{thm:counter}\textup{(ii)}.

\subsection{Sharpness of $p$-range for almost-Dirichlet principle}
In this subsection, we prove sharpness of $p$-range $[4/3,\, 4]$ in Corollary~\ref{cor:almostB}. 

Fix $p>4$ and take $s=1-1/p$ in Theorem~\ref{thm:counter}\textup{(i)}.
Since $f$ vanishes near the corner and is smooth on the rest of $\Gamma$,
it admits an extension $F\in C_c^\infty(\mathbb C)$. In particular,
$\int_{\Omega_i}|\nabla F|^p\,dA<\infty$.

Fix \(1<p<4/3\) and take \(s=1-1/p\) in \textup{(ii)}. Then
\[
 \lambda=\frac2p-1,\qquad (1-s)p-1=0,\qquad p(\lambda+1)=2.
\]
The boundary function \(f\) is unbounded at \(0\), whereas the
almost-Dirichlet principle concerns continuous boundary data.
We therefore cut off \(f\) near \(0\), then prove that their extension
energies stay bounded: this will be our main effort in the following.  However,  those of their Poisson extensions
tend to infinity. 
Thus, the examples in Theorem 1.6 establish sharpness for the interior domains.

\noindent\textbf{Continuous data with bounded extension energies.}
Let \(\chi\) be a smooth radial cutoff, equal to \(1\) for
\(r\le r_0/2\) and supported in \(r<r_0\). Set
\[
 G(re^{i\theta})
 =\chi(r)\frac{\theta}{\omega}b(r)
  +(1-\chi(r))u_i(re^{i\theta}).
\]
Notice that the first term interpolates between \(0\) on \(\theta=0\)
and \(b(r)\) on \(\theta=\omega\) near the corner, and the function 
\(G\) has boundary values \(f\) away from \(0\).
For \(r\le r_0/2\), we have
\[
 |G|\lesssim r^{-\lambda}L(r)^{-3/2},\quad
 |\nabla G|=\sqrt{\left|\partial_rG\right|^2
+\frac{1}{r^2}\left|\partial_\theta G\right|^2}\lesssim |b'(r)|+\frac{|b(r)|}{r}
 \lesssim r^{-\lambda-1}L(r)^{-3/2}.
\]
Choose a smooth cutoff \(0\le\eta\le1\), equal to \(0\) on
\([0,1]\) and to \(1\) on \([2,\infty)\). Define
\[
 G_\epsilon(z)=\eta(|z|/\epsilon)G(z),
 \qquad f_\epsilon=G_\epsilon|_\Gamma.
\]
Since \(G_\epsilon=0\) for \(|z|\le\epsilon\), it is continuous
on \(\overline{\Omega_i}\) and smooth in \(\Omega_i\). 
Write \(r=|z|\). Applying  the product rule gives, for \(r>0\),
\[
 \nabla G_\epsilon(z)
 =\eta(r/\epsilon)\nabla G(z)
  +\frac{1}{\epsilon}\eta'(r/\epsilon)G(z)\nabla r.
\]
Notice that the second term vanishes outside the annulus \(\epsilon<r<2\epsilon\). Here, 
for \(z=x+iy\), we have \(\nabla r=(x,y)/r\), so
\(|\nabla r|=\sqrt{x^2+y^2}/r=1\). Thus, on this annulus,
\[
 \left|\frac{1}{\epsilon}\eta'(r/\epsilon)G(z)\nabla r\right|
 =\frac{|\eta'(r/\epsilon)|}{\epsilon}|G(z)|
 \le 2\|\eta'\|_\infty\,\frac{|G(z)|}{r}. 
\]
Thus, \(|\nabla G_\epsilon|\) satisfies the same bound
near \(0\) as \(|\nabla G|\), uniformly in \(\epsilon\). On the other hand, away from \(0\)
the gradients are uniformly bounded. Hence
\[
 \int_{\Omega_i}|\nabla G_\epsilon|^p\,dA
 \lesssim 1+\int_0^{r_0/2}
      r^{1-p(\lambda+1)}L(r)^{-3p/2}\,dr
 =1+\int_0^{r_0/2}\frac{dr}{rL(r)^{3p/2}}<\infty,
\]
with a bound independent of \(\epsilon\).

\smallskip
\noindent\textbf{Divergence of the Poisson energies.}
On \(\Gamma\), we have \(0\le f_\epsilon\le f\) and 
for each fixed boundary point \(z\in\Gamma\setminus\{0\}\),
\(\lim_{\epsilon\downarrow0}f_\epsilon(z)=f(z)\). 

Since \(f\) is integrable
for harmonic measure, dominated convergence in the Poisson
formula, also after differentiation on compact subsets, gives
\[
 \nabla\mathcal P_{\Omega_i}f_\epsilon
 \longrightarrow \nabla u_i
 \qquad\text{locally uniformly in }\Omega_i.
\] 
 Set
\[
 g=f\circ\varphi,\qquad g_\epsilon=f_\epsilon\circ\varphi
 \qquad\text{on }\mathbb T.
\]
Since \(g_\epsilon\to g\) almost everywhere on \(\mathbb T\),   \(|g_\epsilon-g|\le |g|\) and $g\in L^1$ on $\mathbb T$. By the dominated convergence theorem,
\[
 \int_{\mathbb T}|g_\epsilon-g|\,dm\longrightarrow0.
\]
We now estimate the gradients. Since
\(u_i=\mathcal P_{\Omega_i}f\), the Poisson formula gives
\[
 \bigl(\mathcal P_{\Omega_i}f_\epsilon-u_i\bigr)(\varphi(\zeta))
 =\int_{\mathbb T}P(\zeta,\xi)
       \bigl(g_\epsilon(\xi)-g(\xi)\bigr)\,dm(\xi),
 \qquad
 P(\zeta,\xi)=\frac{1-|\zeta|^2}{|\zeta-\xi|^2}.
\]
The direct computation shows 
\[
 C_R:=\sup_{\substack{|\zeta|\le R\\ \xi\in\mathbb T}}
             |\nabla_\zeta P(\zeta,\xi)|<\infty.
\]
This justifies differentiation under the integral. Consequently,
\[
 \sup_{|\zeta|\le R}
 \left|\nabla_\zeta
   \bigl[(\mathcal P_{\Omega_i}f_\epsilon-u_i)\circ\varphi\bigr]
   (\zeta)\right|
 \le C_R\int_{\mathbb T}|g_\epsilon-g|\,dm
 \longrightarrow0.
\]

Finally, let \(K\subset\Omega_i\) be compact. Then
\(\varphi^{-1}(K)\subset\{|\zeta|\le R\}\) for some \(R<1\),
and \((\varphi^{-1})'\) is bounded on \(K\). By the chain rule
and conformality,
\[
 \sup_{z\in K}
 \bigl|\nabla\mathcal P_{\Omega_i}f_\epsilon(z)-\nabla u_i(z)\bigr|
 \le \sup_{z\in K}|(\varphi^{-1})'(z)|\,
      C_R\int_{\mathbb T}|g_\epsilon-g|\,dm
 \longrightarrow0.
\]
This proves the local uniform convergence of the sequence $\nabla\mathcal P_{\Omega_i}f_\epsilon$ in \(\Omega_i\).  
By Fatou's lemma,
\[
 \infty=\int_{\Omega_i}|\nabla u_i|^p\,dA
 \le \liminf_{\epsilon\downarrow0}
      \int_{\Omega_i}|\nabla\mathcal P_{\Omega_i}f_\epsilon|^p\,dA.
\]
Thus the Poisson energies tend to infinity while the energies
of the continuous extensions \(G_\epsilon\) stay bounded.
This proves the lower-bound sharpness for the interior domains in
Corollary~\ref{cor:almostB}.

Invert about the radial center $a$. By Proposition~\ref{ie}(1), the
exterior Poisson energy is infinite for $p>4$ and tends to infinity
as $\epsilon\downarrow0$ for $1<p<4/3$.   In contrast, for $p>4$ the same boundary function has an exterior
extension of finite $p$-energy. For $1<p<4/3$, the corresponding
boundary functions have exterior extensions whose $p$-energies are
bounded independently of $\epsilon$. To prove both assertions, let $F$ be the smooth extension for $p>4$
or the function $G_\epsilon$ for $1<p<4/3$. Multiply $F$ by a fixed smooth cutoff $\psi$ that
vanishes near $a$ and equals $1$ near $\Gamma$. This preserves the
boundary values and adds only bounded energy, since $F$ is uniformly
bounded on $\operatorname{supp}\nabla\psi$. Inversion is bi-Lipschitz
on $\operatorname{supp}(\psi F)$, so the inverted extension also has
bounded $p$-energy.
\par
\endgroup

\section{Reverse trace inclusion II: Theorems \ref{thm:chordarc},  \ref{thm:reverseinclusion} (b)(c) and \ref{thm:Bren}}
Throughout this section, we write $\Omega:=\Omega_i$ and
$\delta(z):=d(z,\Gamma)$. Recall that $\Gamma$ being chord-arc implies its inverted curve $\Gamma'=1/\overline{\Gamma}$ is also chord-arc, and vice versa. 
By the reduction of the exterior case to the interior case in Section~\ref{subsec: exte}, to prove Theorem \ref{thm:chordarc} it suffices to prove part~\textup{(1)} of  the following theorem.

For $1<p<2$, the assumption in Theorem \ref{sec4:reverse-trace} is $|\varphi'|\in A_p$; for
$p>2$, it is $|\varphi'|\in A_{p'}$. Thus, 
for $p\neq2$, the assumption is stronger than $A_2$, but
it yields the reverse estimate throughout the full range $0<s<1$.
Under the weaker \(A_2\) assumption, we use
interpolation for part~\textup{(2)} of the following theorem  and a direct Poisson estimate
for  part~\textup{(3)}. These are
Theorem~\ref{thm:reverseinclusion}\textup{(b),(c)}.
The proof of Theorem~\ref{thm:Bren} is given in the last subsection.

\begin{theo}\label{thm:chordarc-1}
Let $\Gamma$ be a bounded chord-arc curve. Suppose that one of the following three conditions holds: 
\begin{enumerate}
    \item[\rm (1)] $1/p\le s<1/2+1/(2p)$, $1<p<\infty$;  
    \item[\rm (2)] $0<s\le \frac1p+\frac12$,  $p > 2$ and  $|\varphi'|\in A_2$ on $\mathbb T$;
    \item[\rm (3)] $\frac1p-\frac12\le s < 1$, $1<p < 2$ and $|\varphi'|\in A_2$ on $\mathbb T$.
\end{enumerate}
 Then  every $f\in B_{p,p}^s(\Gamma)$ has the Poisson extension satisfying 
\begin{equation}\label{sec7:energy-comparison}
\mathcal E_{s,p;\Omega}(\mathcal P_{\Omega}f)  
\asymp[f]_{\dot B_{p,p}^s(\Gamma)}, \qquad  
\operatorname{Tr}_\Gamma\mathcal P_\Omega f=f.
\end{equation}
The trace identity holds $\mathcal H^1$-almost everywhere.
\end{theo}

Theorem~\ref{thm:inclusion} already gives
$[f]_{\dot B_{p,p}^s(\Gamma)}\lesssim\mathcal E_{s,p;\Omega}(u)$
for every $u\in\mathcal B_{p,p}^s(\Omega)$ with
$f:=\operatorname{Tr}_\Gamma u$, whenever $1<p<\infty$ and $0<s<1$. The second equality in \eqref{sec7:energy-comparison} can be proved using the same argument as the last step of the proof of Theorem \ref{sec4:reverse-trace}.  
Thus, it remains to prove the reverse estimate
\begin{equation}\label{1p12}
    \mathcal E_{s,p;\Omega}(\mathcal P_{\Omega}f)  
\lesssim [f]_{\dot B_{p,p}^s(\Gamma)}
\end{equation}
for every $f\in B_{p,p}^s(\Gamma)$.

\subsection{The supercritical case \texorpdfstring{$s>1/p$}{s > 1/p}}

\begin{prop}
Let $\Omega$ be a bounded Jordan domain with boundary $\Gamma$,
and let $1<p<\infty$ and $1/p<s<1$. There exists a constant
$C=C(p,s)>0$ such that every $u\in\mathcal B_{p,p}^s(\Omega)$ satisfies
\[
  |\nabla u(z)|\le C\mathcal E_{s,p;\Omega}(u)
       \delta(z)^{s-1/p-1},\qquad z\in\Omega.
\]
\end{prop}
\begin{proof}
Fix $z\in\Omega$ and set $D_z=B(z,\delta(z)/2)$. Since the
components of $\nabla u$ are harmonic, the mean value property gives
\[
  \nabla u(z)=\frac{1}{|D_z|}\iint_{D_z}\nabla u(\zeta)\,dA(\zeta).
\]
Moreover, $\delta(z)/2\le\delta(\zeta)\le3\delta(z)/2$ on $D_z$.
Hence H\"older's inequality yields
\begin{align*}
|\nabla u(z)|^p
&\le\frac{4}{\pi\delta(z)^2}
       \iint_{D_z}|\nabla u(\zeta)|^p\,dA(\zeta)\\
&\le C\delta(z)^{-1-(1-s)p}
       \iint_{D_z}|\nabla u(\zeta)|^p
       \delta(\zeta)^{(1-s)p-1}\,dA(\zeta)\\
&\le C\mathcal E_{s,p;\Omega}(u)^p
       \delta(z)^{-1-(1-s)p}.
\end{align*}
Taking $p$th roots proves the assertion.
\end{proof}
Thus, $|\nabla u(z)|\le C\delta(z)^{\beta-1}$ with
$\beta=s-1/p$. When $\Gamma$ is a quasicircle, we may apply the
following consequence of the Gehring--Martio theorem \cite{GM}.
\begin{prop}
Let $\Omega$ be a bounded domain whose boundary is a quasicircle,
let $u$ be harmonic in $\Omega$, and let $0<\beta\le1$.
The following statements are equivalent:
\begin{enumerate}
\item[\rm(1)] There exists $C>0$ such that
              $|\nabla u(z)|\le C\delta(z)^{\beta-1}$ for every
              $z\in\Omega$.
\item[\rm(2)] The function $u$ extends to an element of
              $\Lambda^\beta(\overline\Omega)$.
\end{enumerate}
Here, $\Lambda^\beta(\overline\Omega)$ denotes the space of
H\"older continuous functions of order $\beta$ on $\overline\Omega$.
\end{prop}

Since $1/p<s<1$ implies $0<\beta=s-1/p<1$, every
$u\in\mathcal B_{p,p}^s(\Omega)$ extends continuously to
$\overline\Omega$. Its boundary values are therefore well-defined
and characterize $u$, as being the unique harmonic extension of this continuous boundary value.

In the critical case $s=1/p$, part~(1) of Theorem~\ref{thm:chordarc-1} is
already known; see \cite{LS}. Its proof relies primarily on the
comparability of the critical energies under conformal pullback and the existence of a
transmission operator.

Another endpoint we will use is $p=\infty$. On a quasidisk, the preceding
proposition identifies harmonic H\"older functions of order
$0<s<1$ with the space
\[
  \mathcal B_{\infty,\infty}^s(\Omega)
  :=\left\{u:\ u\text{ is harmonic in }\Omega,\quad
       \sup_{z\in\Omega}\delta(z)^{1-s}|\nabla u(z)|<\infty\right\}
\]
understood modulo constants. A theorem of Hinkkanen \cite[Theorem~1]{Hin}
states that, for any bounded Jordan domain
$\Omega$ and any $0<s<1/2$, the harmonic extension of
$f\in\Lambda^s(\Gamma)$ belongs to
$\Lambda^s(\overline\Omega)$, with norm controlled by that of $f$.
In particular, it belongs to $\mathcal B_{\infty,\infty}^s(\Omega)$.
The upper bound $1/2$ is sharp for general Jordan domains.
For quasicircles whose Riemann maps admit a $K$-quasiconformal
extension, with $1\le K\le2$, this range was improved to
$0<s<1/K$ in \cite{Zi1}.

The case $p=2$, $0<s<1$ is also an endpoint that we will use. In this case, if $|\varphi'|\in A_2$ on $\mathbb T$ then we already know  from   Theorem \ref{sec4:reverse-trace} that \eqref{1p12} holds. 

We interpolate between  two  of these three endpoint estimates to establish
\eqref{1p12} in the stated range. The main tools are the
functorial property of Calder\'on complex interpolation and the
Stein--Weiss interpolation theorem for weighted $L^p$ spaces.

\subsection{Calder\'on complex interpolation}

Now we recall the weighted $L^p$ interpolation theorem originating
in the work of Stein \cite{Ste56} and Stein--Weiss \cite{SW58}.
Although these results predate Calder\'on's theory of complex
interpolation spaces, they can be formulated in that framework as
follows; for this formulation, including the infinite endpoints,
see \cite[Section~1.1]{CwE}.
\begin{theo}[Stein--Weiss theorem]
Let $\Omega$ be an open subset of $\mathbb C$, and let $\omega_0$
and $\omega_1$ be weights that are positive and finite almost
everywhere on $\Omega$. Given $1\le p_0,p_1\le\infty$ and
$0<\theta<1$, define
\begin{equation}\label{p0p1} 
\begin{split}        
\frac{1}{p} &= \frac{1-\theta}{p_0}+\frac{\theta}{p_1},\\        
\omega^{\frac{1}{p}} & = \omega_0^{\frac{1-\theta}{p_0}}\omega_1^{\frac{\theta}{p_1}}.    
\end{split}
\end{equation}
Then, 
\begin{equation}\label{Stein-interp}    
[L^{p_0}(\Omega, \omega_0), L^{p_1}(\Omega, \omega_1)]_\theta =  L^{p}(\Omega, \omega)
\end{equation}
with equality of norms.  When $p_0=p_1=\infty$,
\eqref{p0p1} does not determine $\omega$, but
\eqref{Stein-interp} holds for every weight $\omega$, since all three spaces
then coincide isometrically with $L^\infty(\Omega,dA)$.
\end{theo}

Let $(\Omega,\mu)$ be a $\sigma$-finite measure space, and let
$m:\Omega\to(0,\infty)$ be measurable. Define
\begin{equation*}    
L_m^p(\mu) := \{f:\; mf \in L^p(\mu)\}
\end{equation*}
with norm $\Vert f\Vert_{L_m^p(\mu)} :=\Vert mf\Vert_{L^p(\mu)}$.
Thus, for $1\le p<\infty$,
\begin{equation*}    
\Vert f \Vert^p_{L_m^p(\mu)} = \int_\Omega|f(z)|^pm(z)^pd\mu(z),
\end{equation*}
while 
\begin{equation*}    
\Vert f\Vert_{L_m^\infty(\mu)}
  =\operatorname*{ess\,sup}_{z\in\Omega}|f(z)|m(z),
\end{equation*}
where the essential supremum is taken with respect to $\mu$.
For $1\le p<\infty$, we have $L_m^p(\mu)=L^p(\nu)$ with
$d\nu=m^p\,d\mu$. At $p=\infty$, however, multiplication by $m$
must be distinguished from changing the underlying measure:
$L_m^\infty(\mu)$ need not coincide with $L^\infty(\nu)$ for
$d\nu=m\,d\mu$. Indeed, $\mu$ and $\nu$ have the same null sets,
so $L^\infty(\nu)=L^\infty(\mu)$ isometrically. The following
form of the Stein--Weiss theorem accommodates both finite
exponents and the endpoint $p=\infty$.
\begin{theo}\label{Lmp}
Let $(\Omega,\mu)$ be a $\sigma$-finite measure space, and let
$m_0,m_1:\Omega\to(0,\infty)$ be measurable. Define
$L^{p_j}_{m_j}(\mu)$, $j=0,1$, as above. Suppose that
$1\le p_0,p_1\le\infty$, $\min\{p_0,p_1\}<\infty$, and
$0<\theta<1$. Set
\begin{equation*}        
\frac{1}{p}=\frac{1-\theta}{p_0}+\frac{\theta}{p_1}, \quad m_\theta = m_0^{1-\theta}m_1^\theta.    
\end{equation*}    
Then,     
\begin{equation*}        
[L^{p_0}_{m_0}(\mu), \; L^{p_1}_{m_1}(\mu)]_\theta = L^p_{m_\theta}(\mu)   
\end{equation*}    
with equality of norms. 
\end{theo}
\begin{proof}
\noindent\emph{Case 1: $1\le p_0,p_1<\infty$.}
This is a direct application of the Stein--Weiss theorem. Indeed,
\begin{equation*}
    L^{p_i}_{m_i}(\mu) = L^{p_i}(\omega_id\mu),\quad \omega_i = m_i^{p_i}, \quad i=0, 1.
\end{equation*}
The Stein--Weiss theorem gives
\begin{equation*}
    [L^{p_0}(\omega_0d\mu),\, L^{p_1}(\omega_1d\mu)]_\theta = L^{p}(\omega d\mu), \quad \omega^{\frac{1}{p}} = \omega_0^{\frac{1-\theta}{p_0}} \omega_1^{\frac{\theta}{p_1}}
\end{equation*}
which is equivalent to
\begin{equation*}
    [L^{p_0}_{m_0}(\mu), \; L^{p_1}_{m_1}(\mu)]_\theta = L^p_{m_\theta}(\mu).
\end{equation*}

\noindent\emph{Case 2: $1\le p_0<\infty$ and $p_1=\infty$.}
Set
\[
p = \frac{p_0}{1-\theta},
\qquad
m_\theta = m_0^{1-\theta}m_1^\theta.
\]
Define
\[
d\nu = \left(\frac{m_0}{m_1}\right)^{p_0} d\mu,
\qquad
Uf = m_1f.
\]
Since $m_0,m_1$ are positive and finite almost everywhere,
$\mu$ and $\nu$ have the same null sets. Direct computation
shows that $U$ is an isometric isomorphism of Banach couples:
\[
U:\bigl(L^{p_0}_{m_0}(\mu),L^\infty_{m_1}(\mu)\bigr)
\longrightarrow
\bigl(L^{p_0}(\nu),L^\infty(\nu)\bigr).
\]
By invariance of complex interpolation under isometric
isomorphisms and the Stein--Weiss theorem,
\[
U\bigl([L^{p_0}_{m_0}(\mu),L^\infty_{m_1}(\mu)]_\theta\bigr)
=
[L^{p_0}(\nu),L^\infty(\nu)]_\theta
=
L^p(\nu)
\]
isometrically. Finally, using $p_0=(1-\theta)p$,
\[
\begin{aligned}
\|Uf\|_{L^p(\nu)}^p
&=
\int_\Omega |f|^p m_1^p
\left(\frac{m_0}{m_1}\right)^{p_0}\,d\mu\\
&=
\int_\Omega |f|^p
\bigl(m_0^{1-\theta}m_1^\theta\bigr)^p\,d\mu
=
\|f\|_{L^p_{m_\theta}(\mu)}^p.
\end{aligned}
\]
Therefore,
\[
[L^{p_0}_{m_0}(\mu),L^\infty_{m_1}(\mu)]_\theta
=
L^p_{m_\theta}(\mu)
\quad\text{with equality of norms.}
\]
The case $p_0=\infty$ and $p_1<\infty$ follows by interchanging the endpoints and replacing $\theta$ by $1-\theta$.
\end{proof}

\subsection{Proof of Theorem \ref{thm:chordarc-1} \textup{(1)}}
Let $\Gamma$ be chord-arc. We interpolate between
$B_{q,q}^{1/q}(\Gamma)$, with $1<q<\infty$, and
$B_{\infty,\infty}^{\sigma}(\Gamma)$, with $0<\sigma<1/2$,
using the linear operator
\[
  S_\Omega f=\nabla \mathcal P_\Omega f.
\]
The endpoint results recalled above give the bounded maps
\begin{align*}
 S_\Omega:
 \dot B_{q,q}^{1/q}(\Gamma)
 &\longrightarrow
 L^q_{\delta^{(q-2)/q}}(dA),\\
 S_\Omega:
 \dot B_{\infty,\infty}^{\sigma}(\Gamma)
 &\longrightarrow
 L^{\infty}_{\delta^{1-\sigma}}(dA).
\end{align*}
Set
\begin{equation}\label{eq:flex-parameters}
 \frac1p=\frac{1-\theta}{q},
 \qquad
 s=\frac{1-\theta}{q}+\theta\sigma
   =\frac1p+\theta\sigma.
\end{equation}
The Besov interpolation formula from Section~2, transported to
$\Gamma$ by its bi-Lipschitz arclength parametrization, and
Theorem~\ref{Lmp} yield
\begin{align*}
 [\dot B_{q,q}^{1/q}(\Gamma),\dot B_{\infty,\infty}^{\sigma}(\Gamma)]_\theta
 &=\dot B_{p,p}^{s}(\Gamma),\\
 [L^q_{\delta^{(q-2)/q}}(dA),L^{\infty}_{\delta^{1-\sigma}}(dA)]_\theta
 &=L^p_{\delta^{((1-s)p-1)/p}}(dA).
\end{align*}
The exponent in the second identity can be checked directly:
\begin{equation*}
    \frac{q-2}{q}(1-\theta)+(1-\sigma)\theta = -\frac{1}{p}+1 - s
\end{equation*}
where we used  \eqref{eq:flex-parameters}.
Since
\[
  \theta=1-\frac qp<1-\frac1p,
\]
the resulting parameters satisfy
\[
  \frac1p<s<\frac1p+\frac12\left(1-\frac1p\right)
  =\frac12+\frac1{2p}.
\]
Conversely, every pair $(p,s)$ in this open range can be obtained
in this way. Indeed, choose
\[
  1<q<p(1-2s)+2,
  \qquad \theta=1-\frac qp,
  \qquad \sigma=\frac{s-1/p}{\theta}.
\]
The assumptions on $s$ imply $1<p(1-2s)+2<p$, so
$0<\theta<1$ and $0<\sigma<1/2$.

Applying Theorem~\ref{functor} to $S_\Omega$ and the preceding
interpolation couples gives
\[
  \mathcal E_{s,p;\Omega}(\mathcal P_\Omega f)
  =\|S_\Omega f\|_{L^p_{\delta^{((1-s)p-1)/p}}(dA)}
  \lesssim [f]_{\dot B^s_{p,p}(\Gamma)}.
\]
this completes the proof of part~\textup{(1)} of Theorem~\ref{thm:chordarc-1} and hence 
of Theorem~\ref{thm:chordarc}.

Geometrically, these interpolation points lie between the critical
line $s=1/p$ and the upper boundary $s=1/2+1/(2p)$ in the
$(1/p,s)$-plane. This region is generated by interpolating the
critical line with the H\"older segment $1/p=0$, $0<s<1/2$.
Repeated interpolation between these endpoint families cannot
enlarge this region.

\subsection{Proof of Theorem \ref{thm:chordarc-1} \textup{(2)}}
\label{subsec:interpolation-A2}
We first treat \(p>2\) and \(s<1/2+1/p\). 
The equality case is treated below. Consider the complex interpolation identity
\begin{equation}\label{eq:interpolation}
\left[\dot B^{a}_{2,2}(\Gamma),\dot B^b_{\infty,\infty}(\Gamma)\right]_\theta
=\dot B^{s}_{p,p}(\Gamma),
\qquad
\frac1p=\frac{1-\theta}{2},
\qquad
s=(1-\theta)a+\theta b.
\end{equation}
Fix
\[
p>2,\qquad 0<s<\frac12+\frac1p.
\]
Set
\begin{equation}\label{eq:parameters}
\theta=1-\frac2p, \qquad a=\frac{s}{\frac12+\frac1p}, \qquad b=\frac a2.
\end{equation}
Then
\[
0<\theta<1, \qquad 0<a<1, \qquad 0<b<\frac12.
\]
This shows that every  parameter pair in the strict range of part~\textup{(2)} is attained. Correspondingly, we also use the complex interpolation
\begin{equation}\label{eq:LLL}
     [L^2_{\delta^{1/2-a}}(dA),L^{\infty}_{\delta^{1-b}}(dA)]_\theta
 =L^p_{\delta^{1-s-1/p}}(dA).
\end{equation}
With the assumption $|\varphi'|\in A_2$ on $\mathbb T$, Theorem \ref{sec4:reverse-trace} proved 
 the linear operator
\[  
S_\Omega :\; \dot B^{a}_{2,2}(\Gamma) \longrightarrow L^2_{\delta^{1/2-a}}(dA)
\]
is bounded. Recall that 
\[  
S_\Omega :\;  \dot B^b_{\infty,\infty}(\Gamma) \longrightarrow L^{\infty}_{\delta^{1-b}}(dA)
\]
is also bounded.  Combining with \eqref{eq:interpolation} and \eqref{eq:LLL}, these endpoint results give the bounded maps
\begin{align*} 
S_\Omega: 
\dot B^{s}_{p,p}(\Gamma) \longrightarrow 
L^p_{\delta^{1-s-1/p}}(dA).
\end{align*}
This proves part~\textup{(2)} in the strict range.

We now include \(s=1/2+1/p\) for \(p>2\). Suppose that
$|\varphi'|\in A_2$ on $\mathbb T$. By the openness of
the Muckenhoupt classes, there exists $r\in(1,2)$ such that
$|\varphi'|\in A_r$. Put
$r'=r/(r-1)>2$. 

If $r\le t\le r'$, then both $t$ and
$t'=t/(t-1)$ are at least $r$. Hence $|\varphi'|\in A_t\cap A_{t'}$,
and Theorem~\ref{sec4:reverse-trace} gives the bounded maps
\[
  S_\Omega:\dot B^a_{t,t}(\Gamma)
  \longrightarrow L^t_{\delta^{1-a-1/t}}(dA),
  \qquad r\le t\le r',\quad 0<a<1.
\]
Thus, the desired equality case already follows when
$r\le p\le r'$.

Suppose that $p>r'$, and set $s=1/2+1/p$. We choose \(\theta,a,b\) to satisfy the interpolation identities
\[
\frac1p=\frac{1-\theta}{r'},
\qquad
s=(1-\theta)a+\theta b,
\]
with \(0<\theta<1\), \(0<a<1\), and \(0<b<1/2\).
The first identity forces $\theta=1-r'/p$. Since \(p>r'>0\), we have \(0<\theta<1\). Choose $b=a/2$
for convenience: it ensures $0<b<1/2$ whenever $0<a<1$. The second interpolation identity then becomes
$$
s=(1-\theta)a+\theta\frac a2=\left(1-\frac{\theta}{2}\right)a
=\left(\frac12+\frac{r'}{2p}\right)a.
$$
Therefore, define 
\[
a=\frac{s}{\frac12+\frac{r'}{2p}}
=\frac{p+2}{p+r'},
\qquad
b=\frac a2=\frac{p+2}{2(p+r')}.
\]
Finally, because \(r'>2\),
\[
0<a=\frac{p+2}{p+r'}<1,
\qquad
0<b=\frac a2<\frac12.
\]
Thus, all parameters are admissible.  
The Besov interpolation formula and Theorem~\ref{Lmp} give
\begin{align*}
  [\dot B^a_{r',r'}(\Gamma),\dot B^b_{\infty,\infty}(\Gamma)]_\theta
  &=\dot B^s_{p,p}(\Gamma),\\
  [L^{r'}_{\delta^{1-a-1/r'}}(dA),L^\infty_{\delta^{1-b}}(dA)]_\theta
  &=L^p_{\delta^{1-s-1/p}}(dA).
\end{align*}
Consequently, \eqref{1p12} holds at $s=1/2+1/p$.

\begingroup

\subsection{Proof of Theorem \ref{thm:chordarc-1} \textup{(3)}}

 For \(1<p<2\),
the critical estimate covers \(s=1/p\). For \(1/p<s<1\),
interpolate the estimate at a critical pair \((q,1/q)\),
where \(1<q<p\), with the estimate at \((2,a)\). The
interpolation identities are
\[
 \frac1p=\frac{1-\theta}{2}+\frac\theta q,\qquad
 s=(1-\theta)a+\frac\theta q.
\]
They give \(0<\theta<1\), and \(0<a<1\) when \(q\) is close
enough to \(1\), since \(a\to(s-(2/p-1))/(2-2/p)\in(0,1)\).
The Besov spaces and weighted target spaces interpolate as in
Section~\ref{subsec:interpolation-A2}.
It remains to prove \eqref{1p12} for
\begin{equation}\label{eq:srange}
    1<p<2,\qquad \frac1p-\frac12\le s<\frac1p.
\end{equation}
The endpoint $s=1/p-1/2$ is included because the $A_2$ condition
improves to $|\varphi'|\in A_r$ for some $p<r<2$. Thus
$1/p-1/r<1/p-1/2\le s$, leaving room to choose
$1/p-1/r<\sigma<s$ in the local oscillation estimate below.

\paragraph{Step 1. A local oscillation estimate.}
\revision{Assume first that $f$ is bounded.}
Since \(|\varphi'|\in A_2\), the open property of the Muckenhoupt classes
gives \(|\varphi'|\in A_r\) for some \(p < r<2\). By \eqref{eq:srange} we can
choose
\[
 \frac1p-\frac1r\leq\sigma<s.
\]
For an arc \(I\subset\mathbb T\), write \(J :=\varphi(I)\) and denote 
\(m(I):=|J|=\int_I |\varphi'|\,d\theta\). Set \(g=f\circ\varphi\) on \(\mathbb T\). 
Write 
\[
 g_I=\frac{1}{|I|}\int_I g(e^{i\theta})\,d\theta,
 \qquad
 f_J=\frac{1}{|J|}\int_J f(z)\,|dz|. 
\]
Applying H\"older's inequality, \(|\varphi'|\in A_r\) and $g=f\circ\varphi$ gives
\begin{align*}
& \frac1{|I|}\int_I |g-f_J|\,d\theta 
\le
\left(\frac1{|I|}\int_I |g-f_J|^r |\varphi'|\,d\theta\right)^{1/r}
\left(\frac1{|I|}\int_I|\varphi'|^{-1/(r-1)}\,d\theta\right)^{(r-1)/r}\\
& 
=\left(\frac1{|I|}\int_I |\varphi'|\,d\theta\cdot\frac1{|J|}\int_J |f-f_J|^r\,|dz|\right)^{1/r} \left(\frac1{|I|}\int_I |\varphi'|^{-1/(r-1)}\,d\theta\right)^{(r-1)/r}\\
& \le [|\varphi'|]_{A_r}^{1/r}\left(\frac1{|J|}\int_J|f-f_J|^r\,|dz| \right)^{1/r}.
\end{align*}
Then, by using 
\[
|g_I-f_J|
=\left|\frac1{|I|}\int_I(g-f_J)\,d\theta\right|
\le \frac1{|I|}\int_I|g-f_J|\,d\theta
\]
and the triangle inequality \(|g-g_I|\le |g-f_J|+|f_J-g_I|\) we conclude that 
\[
 \frac1{|I|}\int_I|g-g_I|\,d\theta\lesssim \left(\frac1{|J|}\int_J|f-f_J|^r\,|dz| \right)^{1/r}=:M(I).
\]
Since \(\sigma p<1\) and \(r\le p/(1-\sigma p)\), 
We now show that the average oscillation of \(f\) on \(J\) is
controlled by its fractional difference on \(J\): 
\begin{equation}\label{eq:local-oscillation-A2}
 M(I)
 \lesssim
 \left(
 |J|^{\sigma p-1}\iint_{J\times J}
 \frac{|f(z)-f(w)|^p}{|z-w|^{1+\sigma p}}\,|dz|\,|dw|
 \right)^{1/p}.
\end{equation}

\revision{If $M(I)=0$, there is nothing to prove.} For \(z\in J\) with
\(f(z)\ne f_J\), choose a subarc \(H\subset J\) containing \(z\)
of length \(h|J|\), where
\[
 h=\min\left\{1,\left(\frac{2M(I)}{|f(z)-f_J|}\right)^r\right\}.
\]
This choice ensures that \(|f_H-f_J|\le |f(z)-f_J|/2\).
Indeed, if \(h=1\), then \(H=J\); otherwise, H\"older's
inequality gives
\[
 |f_H-f_J|
 \le \left(\frac1{|H|}\int_H|f-f_J|^r\,|dw|\right)^{1/r}
 \le h^{-1/r}M(I) =\frac12|f(z)-f_J|.
\]
In particular, \(|f(z)-f_H|\ge |f(z)-f_J|/2\).

For \(w\in H\), we have \(|z-w|\le |H|=h|J|\). Hence Jensen's
inequality gives
\[
 |J|^{\sigma p}\int_J
 \frac{|f(z)-f(w)|^p}{|z-w|^{1+\sigma p}}\,|dw|
 \ge h^{-\sigma p}\frac1{|H|}\int_H|f(z)-f(w)|^p\,|dw|
 \ge h^{-\sigma p}|f(z)-f_H|^p.
\]
The condition \(r\le p/(1-\sigma p)\) says exactly that
\(\sigma p\ge(r-p)/r\). Since \(h\le1\), it follows that
\[
 h^{-\sigma p}\ge h^{-(r-p)/r}
 \ge \left(\frac{|f(z)-f_J|}{2M(I)}\right)^{r-p}.
\]
Combining the last two estimates and averaging in \(z\) over \(J\),
we obtain
\[
 |J|^{\sigma p-1}\iint_{J\times J}
 \frac{|f(z)-f(w)|^p}{|z-w|^{1+\sigma p}}\,|dz|\,|dw|
 \ge \frac{1}{2^r M(I)^{r-p}}\,
       \frac1{|J|}\int_J|f-f_J|^r\,|dz|
 =\frac{M(I)^p}{2^r}.
\]
This proves \eqref{eq:local-oscillation-A2} for bounded $f$.

For general $f$, the average $f_J$ is well defined since
$f\in L^p(J)\subset L^1(J)$. However, as $r>p$, this alone does
not imply that $M(I)<\infty$. We therefore truncate $f$ by setting
\[
 f_N(z)=
 \begin{cases}
   f(z), & |f(z)|\le N,\\
   Nf(z)/|f(z)|, & |f(z)|>N.
 \end{cases}
\]
Since $f_N\to f$ almost everywhere and $|f_N|\le |f|$, dominated
convergence gives $(f_N)_J\to f_J$. Moreover, projection onto the
closed disk of radius $N$ is $1$-Lipschitz, so
\[
 |f_N(z)-f_N(w)|\le |f(z)-f(w)|.
\]
Thus Fatou's lemma and \eqref{eq:local-oscillation-A2} for $f_N$ give
\[
\begin{aligned}
 M(I)
 &=\left(\frac1{|J|}\int_J |f-f_J|^r\,|dz|\right)^{1/r}\\
 &\le \liminf_{N\to\infty}
 \left(\frac1{|J|}\int_J |f_N-(f_N)_J|^r\,|dz|\right)^{1/r}\\
 &\lesssim
 \left(
 |J|^{\sigma p-1}
 \iint_{J\times J}
 \frac{|f(z)-f(w)|^p}{|z-w|^{1+\sigma p}}\,|dz|\,|dw|
 \right)^{1/p}.
\end{aligned}
\]
This proves \eqref{eq:local-oscillation-A2} for general $f$.

Taking $I=\mathbb T$, so that $J=\Gamma$, and using
$|z-w|^{(s-\sigma)p}\le(\operatorname{diam}\Gamma)^{(s-\sigma)p}$
since $\sigma<s$, we obtain
\[
 M(\mathbb T)
 \lesssim |\Gamma|^{\sigma-1/p}
          (\operatorname{diam}\Gamma)^{s-\sigma}
          [f]_{\dot B^s_{p,p}(\Gamma)}<\infty.
\]
The weighted H\"older estimate with $I=\mathbb T$ now gives
\[
 \frac1{2\pi}\int_0^{2\pi}|g(e^{i\theta})|\,d\theta
 \le |f_\Gamma|
    +\frac1{2\pi}\int_0^{2\pi}|g(e^{i\theta})-f_\Gamma|\,d\theta
 \le |f_\Gamma|+[|\varphi'|]_{A_r}^{1/r}M(\mathbb T)<\infty.
\]
Hence $g\in L^1(\mathbb T)$, and its Poisson integral
$V=\mathcal P_{\mathbb D_i}g$ is well defined.

\paragraph{Step 2. From the Poisson integral to the oscillation on an arc.}
For \(x\in[0,2\pi)\), write
\[
I_t := I_t(x)=\{e^{i(x+h)}:-t\le h\le t\},\qquad 0<t<\pi.
\]
and if \(t\ge\pi\), taking \(I_t(x)=\mathbb T\).   
Put \(z=(1-t)e^{ix}\), $0<t<1$. 
Since the differentiated Poisson kernel has integral zero, we get \[
\nabla V(z)=\frac1{2\pi}\int_{-\pi}^{\pi}
\nabla_zP(z,e^{i\theta})\,g(e^{i\theta})\,d\theta = \frac1{2\pi}\int_{-\pi}^{\pi}
\nabla_zP(z,e^{i\theta})\,(g(e^{i\theta})-g_{I_t})\,d\theta.
\]
By \eqref{eq:Poisson1} we conclude that 
\begin{align*}
    |\nabla V((1-t)e^{ix})|&\leq\frac{1}{\pi}\int_{-\pi}^\pi\frac{|g(e^{i\theta})-g_{I_t}|}{|e^{i\theta}-z|^2}\,d\theta = \frac{1}{\pi}\int_{-\pi}^\pi\frac{|g(e^{i(x+h)})-g_{I_t}|}{|e^{ih}-(1-t)|^2}\,dh
\end{align*}
By \(|e^{ih}-(1-t)|\geq \pi^{-1}\max\{t,|h|\}\) for \(|h|\le\pi\), we finally get
\[
 t|\nabla V((1-t)e^{ix})|\lesssim D(x,t),\qquad
 D(x,t):=\int_{-\pi}^{\pi}
 \frac{t\,|g(e^{i(x+h)})-g_{I_t}|}{\max\{t,|h|\}^{2}}\,dh .
\]
The integral defining \(D\) makes sense for every \(t>0\). We use the comparison between the radii \(t\) and \(2t\), and claim that
\begin{equation}\label{eq:DMD}
     D(x,t)\le C M(I_{2t})+\tfrac12 D(x,2t).
\end{equation}
Indeed, since for every \(h\),
\[
 |g(e^{i(x+h)})-g_{I_t}|
 \leq |g(e^{i(x+h)})-g_{I_{2t}}|+|g_{I_t}-g_{I_{2t}}|,
\]
and moreover,  $|g_{I_t}-g_{I_{2t}}| \lesssim M(I_{2t})$ and
\[
\int_{-\pi}^{\pi}\frac{t}{\max\{t,|h|\}^2}\,dh
\le 2\int_0^t\frac{dh}{t}
   +2\int_t^\infty\frac{t}{h^2}\,dh=4,
\]
we get 
$$
D(x,t)\leq \int_{-\pi}^{\pi}
 \frac{t\,|g(e^{i(x+h)})-g_{I_{2t}}|}{\max\{t,|h|\}^{2}}\,dh + C_1M(I_{2t}).
$$
The contribution from \(|h|>2t\) is at most \(\tfrac12 D(x,2t)\), since on this region the kernel at radius \(t\) is exactly half the kernel at radius \(2t\):
\[
\frac{t}{\max\{t,|h|\}^2}
=\frac{t}{|h|^2}
=\frac12\,\frac{2t}{\max\{2t,|h|\}^2}.
\]
When $|h|\leq 2t$, the kernel at radius $t$ is at most $1/t$, and thus this integral is bounded by $C_2M(I_{2t})$. Combining the two cases gives \eqref{eq:DMD}. 

Now we  integrate the estimate \eqref{eq:DMD}. For fixed \(\epsilon>0\), define
\[
\|H(x,t)\|_\epsilon
:=
\left(
\int_\epsilon^\infty\int_0^{2\pi}
|H(x,t)|^p\,m(I_t(x))^{1-sp}\,
\frac{dx\,dt}{t^2}
\right)^{1/p}.
\]
By direct computation, we see $\Vert H(x, 2t)\Vert_\epsilon\leq 2^{1/p}\Vert H(x, t)\Vert_\epsilon$. Applying it to the estimate \eqref{eq:DMD} gives 
\[
 \|D\|_\epsilon\le C\|M(I_t)\|_\epsilon+2^{1/p-1}\|D\|_\epsilon
\]
where each $\Vert\cdot\Vert_\epsilon$ is finite. 
Moving the last term to the left yields 
\[
\displaystyle
\|D\|_\epsilon
\le \frac{C\,2^{1/p}}{1-2^{1/p-1}}\,\|M(I_t)\|_\epsilon.
\]
Letting
\(\epsilon\downarrow0\) and using monotone convergence give the same bound for all \(t>0\). Finally, we get \(\|D\|\lesssim\|M(I_t)\|\). 

For \(z=(1-t)e^{ix}\), \(0<t<1\), Koebe's distance estimate:  $\delta(\varphi(z))\asymp (1-|z|^2)|\varphi'(z)|$, quasisymmetry
of the quasiconformal extension of \(\varphi\) to the plane:
$$
\frac{|\varphi(z_1)-\varphi(z_2)|}{|\varphi(z_1)-\varphi(z_3)|}\leq\eta\left (\frac{|z_1-z_2|}{|z_1-z_3|}\right)
$$
where \(\eta:[0,\infty)\to[0,\infty)\) is an increasing homeomorphism, 
and the chord-arc condition of $\Gamma$ give
\begin{equation}\label{eq:deltatm}
    \delta(\varphi(z))\asymp t|\varphi'(z)|\asymp m(I_t),
 \qquad \varphi(I_t)\subset B(\varphi(z),C\delta(\varphi(z))).
\end{equation}

We now estimate $\mathcal E_{s,p;\Omega_i}(\mathcal P_{\Omega_i}f)^p$: by using the change of variables $w=\varphi(z)$ and the first estimate in \eqref{eq:deltatm} we can come to 
\[
\mathcal E_{s,p;\Omega_i}(\mathcal P_{\Omega_i}f)^p
\asymp\int_{\mathbb D_i}
|\nabla V(z)|^p\,t^{p-2}m(I_t)^{1-sp}\,dA(z).
\]
Writing the area element in $(x, t)$ coordinates: \(dA(z)=(1-t)\,dx\,dt\) gives 
\[
\mathcal E_{s,p;\Omega_i}(\mathcal P_{\Omega_i}f)^p
 \lesssim \int_0^1\int_0^{2\pi}
 |t\nabla V((1-t)e^{ix})|^p\,m(I_t)^{1-sp}\,\frac{dx\,dt}{t^2}.
\]
Since \(t|\nabla V|\lesssim D\) and \(\|D\|\lesssim\|M(I_t)\|\),
we get 
\begin{equation}\label{eq:integrated-oscillation-A2}
 \mathcal E_{s,p;\Omega_i}(\mathcal P_{\Omega_i}f)^p
 \lesssim \int_0^\infty\int_0^{2\pi}
 m(I_t)^{1-sp}M(I_t)^p\,\frac{dx\,dt}{t^2}.
\end{equation}

\paragraph{Step 3. Integrate for a fixed pair of boundary points.}
Write \(J=\varphi(I_t)\). Since \(m(I_t)=|J|\), \eqref{eq:local-oscillation-A2} gives
\[
 m(I_t)^{1-sp}M(I_t)^p
 \lesssim \iint_{J\times J}
 \frac{|f(a)-f(b)|^p}{|a-b|^{1+\sigma p}}\,
 \frac{|da|\,|db|}{|J|^{(s-\sigma)p}} .
\]
By Tonelli's theorem we may fix \(a\ne b\) and integrate first
over the arcs \(J\) containing both points. We show in the following that this
integration contributes at most \(C|a-b|^{-(s-\sigma)p}\).

For \(0<t<1/2\), put \(Z=\varphi((1-t)e^{ix})\). Since
\(dA(z)=(1-t)\,dx\,dt\), the comparisons in Step~2 give $|J|\asymp\delta(Z)$ and 
\[
\frac{dA(Z)}{\delta(Z)^2}
\asymp
\frac{|\varphi'(z)|^2(1-t)\,dx\,dt}
     {t^2|\varphi'(z)|^2}
=(1-t)\frac{dx\,dt}{t^2}
\asymp\frac{dx\,dt}{t^2}.
\]   
If \(a,b\in J\), then we conclude by the second estimate of \eqref{eq:deltatm} that \(\delta(Z)\le |Z-a|\le C\delta(Z)\) and
\(|a-b|\le2C\delta(Z)\). In particular, \(|Z-a|\ge |a-b|/(2C)\).
We then get 
\begin{align*}
 \int_{\substack{0<t<1/2\\a,b\in J}}
 \frac{dx\,dt}{t^2|J|^{(s-\sigma)p}} &\lesssim
\int_{|Z-a|\ge |a-b|/(2C)}
\frac{dA(Z)}{|Z-a|^{2+(s-\sigma)p}}\\
 &\lesssim \int_{|a-b|/(2C)}^\infty
             \frac{d\rho}{\rho^{1+(s-\sigma)p}}
 =\frac{(2C)^{(s-\sigma)p}}{(s-\sigma)p}\,|a-b|^{-(s-\sigma)p}.
 \end{align*}
For \(t\ge1/2\), 
\[
\int_{\substack{t\ge 1/2\\ a,b\in J}}
\frac{dx\,dt}{t^2|J|^{(s-\sigma)p}}
\le
\frac{1}{|a-b|^{(s-\sigma)p}}
\int_{1/2}^{\infty}\int_0^{2\pi}\frac{dx\,dt}{t^2}
=\frac{4\pi}{|a-b|^{(s-\sigma)p}}.
\]

Finally, \((1+\sigma p)+(s-\sigma)p=1+sp\). Inserting the pair
estimate into \eqref{eq:integrated-oscillation-A2} thus gives
\[
 \mathcal E_{s,p;\Omega_i}(\mathcal P_{\Omega_i}f)^p
 \lesssim \iint_{\Gamma\times\Gamma}
 \frac{|f(a)-f(b)|^p}{|a-b|^{1+sp}}\,|da|\,|db|
 =[f]_{\dot B^s_{p,p}(\Gamma)}^p .
\]
This proves \eqref{1p12}.
\par
\endgroup

\subsection{Endpoint integrability of the Riemann map}
\label{subsec:brennan-endpoint}
In this subsection, we prove Theorem \ref{thm:Bren}. \\
\noindent\textbf{Achievement of the exponent $4$.} We first prove the first assertion of Theorem~\ref{thm:Bren}. 
Set $\Omega := \Omega_i$.  Let $z_0=g^{-1}(0)$. By Carath\'eodory's theorem, $g$
extends homeomorphically to $\overline{\Omega}$, with
$|g|=1$ on $\Gamma$. Define the function
\[
  h(z):=\frac{g(z)}{z-z_0},
  \qquad h(z_0):=g'(z_0).
\]
It is holomorphic and nowhere vanishing in $\Omega$, and
extends continuously and without zeros to
$\overline{\Omega}$. Thus, $u(z):=-\log|h(z)|$ 
is harmonic in $\Omega$ and continuous on
$\overline{\Omega}$. It satisfies
\[
  u(z)=\log|z-z_0|=:f(z)\quad \text{on}\; \Gamma
\]
Consequently, $u=\mathcal P_\Omega f$.

Since $z_0\notin\Gamma$, we may choose
$F\in C_c^\infty(\mathbb C)$ agreeing with
$\log|z-z_0|$ in a neighborhood of $\Gamma$.
Corollary~\ref{cor:almostB}, applied at $p=4$,
therefore gives
\[
  \iint_\Omega|\nabla u|^4\,dA
  \lesssim
  \iint_\Omega|\nabla F|^4\,dA
  <\infty.
\]
Finally, differentiating $g(z)=(z-z_0)h(z)$ yields
\[
  g'=h+g\frac{h'}h.
\]
Since $|g|\le1$ and
$|\nabla u|=|h'/h|$, we obtain
\[
  |g'|
  \le \|h\|_{L^\infty(\Omega)}+|\nabla u|.
\]
The function $h$ is bounded and $\Omega$ has finite area.
Hence $g'\in L^4(\Omega)$, as required.

\medskip
\noindent\textbf{Optimality of the exponent $4$.}
We next prove the second assertion of Theorem~\ref{thm:Bren}.
We reuse the one-corner radial-Lipschitz example constructed in
the proof of Theorem~\ref{thm:counter} in
Section~\ref{sec:sharpness}, with a suitable choice of the
interior angle. Fix $p>4$ and set
\[
  \lambda=1-\frac{2}{p}\in(1/2,1),
  \qquad \omega=\frac{\pi}{\lambda}\in(\pi,2\pi).
\]
Let $\Omega_p$ be the corresponding bounded radial-Lipschitz
domain. Its boundary is smooth except at the corner $0$.
Let $g:\Omega_p\to\mathbb D_i$ be a Riemann map.  As in that construction, for some $r_0>0$,
\[
  \Omega_p\cap B(0,r_0)
  =\{re^{i\theta}:0<r<r_0,\ 0<\theta<\omega\}. 
\]

We first justify the estimate
$|g'(z)|\asymp |z|^{\lambda-1}$ near $0$.
Choose the branch of $z^\lambda$ with $0<\arg z<\omega$.
The change of variables $w=z^\lambda$ straightens the sector
into the upper half-disk
\[
  D_R^+:=\{w:|w|<R,\ \operatorname{Im}w>0\},
  \qquad R=r_0^\lambda.
\]
Define $G(w):=g(w^{1/\lambda})$ on $D_R^+$.
By Carath\'eodory's theorem, $g$ extends homeomorphically to
$\overline{\Omega_p}$. Thus $G$ extends continuously to the
real diameter, maps it injectively into $\mathbb T$, and
satisfies $G(0)=g(0)=:\xi\in\mathbb T$.
After decreasing $R$, we may assume that $G$ has no zeros
on the closed upper half-disk. Since $G$ maps the real diameter into the unit circle,
the Schwarz reflection principle extends $G$
holomorphically across the real diameter by 
\[
  \widetilde G(w)=
  \begin{cases}
    G(w),&\operatorname{Im}w\ge0,\\[1mm]
    \displaystyle\frac{1}{\overline{G(\overline w)}},
      &\operatorname{Im}w<0.
  \end{cases}
\]
This extension is injective: its restrictions to the upper
half-disk, the lower half-disk, and the real diameter are
injective, and their images lie, respectively, inside
$\mathbb D_i$, outside $\overline{\mathbb D_i}$, and on
$\mathbb T$. Consequently,
$A:=\widetilde G'(0)\ne0$.
Since $g(z)=\widetilde G(z^\lambda)$ in the sector, the chain
rule gives
\[
  g'(z)=\lambda z^{\lambda-1}
         \widetilde G'(z^\lambda),
  \qquad \widetilde G'(z^\lambda)\longrightarrow A\ne0
  \quad\text{as }z\to0.
\]
It follows that
\[
  |g'(z)|\asymp |z|^{\lambda-1}
  \qquad (z\in\Omega_p,\ z\to0),
\]
with the comparison uniform throughout the sector near $0$.

For sufficiently small $\varepsilon>0$, polar coordinates
therefore yield
\[
  \iint_{\Omega_p\cap B(0,\varepsilon)}
       |g'(z)|^p\,dA(z)
  \asymp
  \int_0^\varepsilon\int_0^\omega
       r^{p(\lambda-1)}r\,d\theta\,dr
  =\omega\int_0^\varepsilon\frac{dr}{r}
  =\infty.
\]
 Hence
$g'\notin L^p(\Omega_p)$, proving the asserted uniform
optimality of $4$.

\begin{appendices}
\section{Weighted traces and boundary values}
\label{app:weighted-traces}


Here we provide proofs of the facts that the trace $(\operatorname{Tr}_\Gamma U)$ defined by completion  also admits  point-wise descriptions in terms of the weighted full-ball and one-sided
averages, and moreover, for harmonic functions, the mean-value property yields non-tangential convergence. 

Let  $1<p<\infty$, and $0<s<1$. Set
\[
  \alpha=(1-s)p-1,\qquad
  \delta(z)=\operatorname{dist}(z,\Gamma),\qquad
  \omega(z)=\delta(z)^\alpha,\qquad d\mu=\omega\,dA.
\]
\begin{prop} 
Let $U\in W^{1,p}(\mathbb C, \omega)$, and $\Gamma$ be a bounded chord-arc curve.  Set $\Omega\in\{\Omega_i,\Omega_e\}$.  The following  statements hold:
\begin{itemize}
    \item[$(T_1)$] For
$\mathcal H^1$-almost every $x\in\Gamma$,
\begin{equation}\label{eq:app-fullballave}
  \lim_{r\downarrow0}
  \frac{1}{\mu(B(x,r))}
  \int_{B(x,r)}
  \bigl|U(z)-(\operatorname{Tr}_\Gamma U)(x)\bigr|
  \,d\mu(z)=0;
\end{equation}
\item[$(T_2)$] Moreover, at every such point $x\in\Gamma$,
\begin{equation}\label{eq:app-strong-one-sided}
  \lim_{r\downarrow0}
  \frac{1}{\mu(B(x,r)\cap\Omega)}
  \int_{B(x,r)\cap\Omega}|U(z)-(\operatorname{Tr}_\Gamma U)(x)|\,d\mu(z)=0.
\end{equation}
In particular, 
\begin{equation}\label{eq:app-equal-average-limits}
\begin{aligned}
  \lim_{r\downarrow0}
  \frac{1}{\mu(B(x,r))}
  \int_{B(x,r)}U(z)\,d\mu(z)
  &=
  \lim_{r\downarrow0}
  \frac{1}{\mu(B(x,r)\cap\Omega)}
  \int_{B(x,r)\cap\Omega}U(z)\,d\mu(z)\\
  &=(\operatorname{Tr}_\Gamma U)(x).
\end{aligned}
\end{equation}
\item[$(T_3)$] If further $U$ is harmonic in $\Omega$, then $U(z)\to (\operatorname{Tr}_\Gamma U)(x)$ as $z
\in\Omega$ tends to $x$ non-tangentially. 
\end{itemize}
\end{prop}

\begin{proof}
\textbf{We start with the proof of part $(T_1)$.} 
For it, we combine the weighted Sobolev representative results of \cite{Kil94} with the continuity of the trace operator mentioned above. 

According to the definition, the functions in $W^{1,p}(\mathbb C, \omega)$ are defined almost everywhere. Each $U\in W^{1,p}(\mathbb C, \omega)$ has a quasicontinuous refined  representative $\tilde U$ and an  approximation sequence 
$\varphi_j\in C_c^\infty(\mathbb C)$ such that both $\varphi_j\rightarrow U \; \text{in}\; W^{1,p}(\mathbb C, \omega)$ 
and $\varphi_j\rightarrow \tilde U$ quasiuniformly in $\mathbb C$, which in particular implies quasieverywhere convergence: 
\begin{equation}\label{eq:Cap0}
    \varphi_j(x) \rightarrow \tilde U(x),\;\; x\in\mathbb C\setminus E_1, \qquad C_{p,\omega}(E_1)=0.
\end{equation}
Here, $ C_{p,\omega}(E_1)$ denotes $(1, p, \omega)$-capacity of set $E_1$. Since $\omega$ is an  $A_p$-weight on $\mathbb C$, every set of zero
$(1,p,\omega)$-capacity has Hausdorff dimension strictly less
than one; see \cite[Corollary~2.12]{Kil94}.
Consequently,
\[
  C_{p,\omega}(E_1)=0
  \quad\Longrightarrow\quad
  \mathcal H^1(E_1)=0,
\]
and in particular, $\mathcal H^1(E_1\cap\Gamma)=0$.

 By continuity of the trace operator, $\varphi_j|_\Gamma \to \operatorname{Tr}_\Gamma U$ in $L^p(\Gamma)$. After passing to a subsequence $\varphi_{j_k}$, the boundary restrictions $\varphi_{j_k}|_\Gamma$
therefore converge to $\operatorname{Tr}_\Gamma U$  at $\mathcal H^1$-almost every point on $\Gamma$.
On the other hand, by \eqref{eq:Cap0} we have
\begin{equation*}    
\varphi_{j_k}(x) \rightarrow \tilde U(x) \;\; \text{at}\;\mathcal H^1\text{-almost every} \; x\in\Gamma.
\end{equation*}
Comparing these two point-wise limits gives
\begin{equation}\label{eq:UU}
    (\operatorname{Tr}_\Gamma U)(x) = \tilde U(x) \;\; \text{at}\;\mathcal H^1\text{-almost every} \; x\in\Gamma.
\end{equation} 
Moreover, 
 \cite[Corollary~3.6]{Kil94} provides a set $E_2$ with $C_{p,\omega}(E_2)=0$ such that 
\begin{equation}\label{eq:strong-full-ball-Sobolev-trace}
  \lim_{r\downarrow0}
  \frac{1}{\mu(B(x,r))}
  \int_{B(x,r)}|\tilde U(z)-\tilde U(x)|\,d\mu(z)=0
  \quad\text{for }x\in\mathbb C\setminus E_2.
\end{equation}
In this formula, $U(z)$ may be used in place of $\tilde U(z)$ because they agree $\mu$-almost everywhere. 
As above, $\mathcal H^1(E_2\cap\Gamma)=0$. Combining \eqref{eq:UU} and \eqref{eq:strong-full-ball-Sobolev-trace} proves  \eqref{eq:app-fullballave}. 

 \medskip
 \textbf{We now prove part $(T_2)$. } Since $\Gamma$ is chord-arc, it is a quasicircle and hence has the two-sided corkscrew property. Thus, there exist 
constants $\kappa\in(0,1)$ and $r_0>0$ such that, for every
$x\in\Gamma$ and $0<r<r_0$, there exists a point $y\in\Omega$ satisfying 
\[
  B(y,\kappa r)\subset B(x,r)\cap\Omega.
\]
Since $|y-x|<r$, we also have 
\[
  B(x,r)\subset B(y,2r).
\]
Since $d\mu = \omega dA$ where $\omega$ is an $A_p$-weight, we have $\mu$ satisfies the doubling property: there is a constant $C_\mu$ such that 
\[  
0<\mu(B(\xi,r))<\infty,  
\qquad  
\mu(B(\xi,2r))\le C_\mu\mu(B(\xi,r))
\]
for all $\xi\in\mathbb C$ and $r>0$. 
Choose an integer $m\ge1$ such that $2^m\kappa\ge2$.
By iterating the doubling inequality,
$$
\mu(B(x,r))  
\le \mu(B(y,2r)) \le \mu(B(y,2^m\kappa r)) \le C_\mu^m\mu(B(y,\kappa r)) \le C_\mu^m\mu(B(x,r)\cap\Omega).
$$
Thus,
\begin{equation}\label{eq:one-sided-density}
  \mu(B(x,r)\cap\Omega)
  \ge (C_\mu^m)^{-1}\mu(B(x,r)),
  \qquad x\in\Gamma,\quad 0<r<r_0.
\end{equation}

Fix a point $x\in\Gamma$ where
\eqref{eq:app-fullballave} holds.
Applying \eqref{eq:one-sided-density} gives
\begin{align*}
  \frac{1}{\mu(B(x,r)\cap\Omega)}&
   \int_{B(x,r)\cap\Omega}|U(z)-(\operatorname{Tr}_\Gamma U)(x)|\,d\mu(z) \\
 &\le
   \frac{C_\mu^m}{\mu(B(x,r))}
   \int_{B(x,r)}|U(z)-(\operatorname{Tr}_\Gamma U)(x)|\,d\mu(z)
\end{align*}
which tends to $0$ as $r\downarrow0$. 
This proves \eqref{eq:app-strong-one-sided}. Finally, 
\[
\begin{aligned}
  &\left|
  \frac{1}{\mu(B(x,r)\cap\Omega)}
  \int_{B(x,r)\cap\Omega}U(z)\,d\mu(z)-(\operatorname{Tr}_\Gamma U)(x)
  \right|\\
  &\qquad\le
  \frac{1}{\mu(B(x,r)\cap\Omega)}
  \int_{B(x,r)\cap\Omega}|U(z)-(\operatorname{Tr}_\Gamma U)(x)|\,d\mu(z),
\end{aligned}
\]
which implies the one-sided average converges to $(\operatorname{Tr}_\Gamma U)(x)$. The corresponding full-ball average
also converges to $(\operatorname{Tr}_\Gamma U)(x)$.
This proves \eqref{eq:app-equal-average-limits}.

\medskip

\textbf{Finally, we prove part $(T_3)$. } Fix a point $x\in\Gamma$ satisfying \eqref{eq:app-fullballave} and an aperture
$a>1$. Consider the non-tangential approach region
\[
  \mathcal N_a^\Omega(x)
  :=\{z\in\Omega:|z-x|<a\delta(z)\}.
\] 
For $z\in\mathcal N_a^\Omega(x)$, put
\[
  t=\delta(z),\qquad D_z=B(z,t/2),\qquad R=(a+1)t.
\]
Then, 
\[
D_z\subset B(x,R)\cap\Omega, \quad
  \frac t2\leq\delta(\zeta)\leq\frac{3t}{2}
  \quad(\zeta\in D_z),
  \quad
  \mu(D_z)\asymp t^\alpha|D_z|\asymp t^{\alpha+2},
\]
where the comparison constants depend only on $\alpha$. Moreover, $B(x,R)\subset B(z,(2a+1)t).$ 
\begin{figure}[htbp]
\centering
\begin{tikzpicture}[
  scale=1.05, >=Latex,
  every node/.style={font=\small},
  point/.style={circle,fill=black,inner sep=1.5pt}
]
  \coordinate (x) at (0,0);
  \coordinate (xi) at (0.796553,0.078196);
  \coordinate (z) at (0.670575,0.726061);

  \path[fill=black!12,draw=black,densely dashed,line width=.8pt,
        line join=round]
    (0.000000,0.000000)
    -- (0.034017,0.012184) -- (0.065036,0.023596) -- (0.096213,0.035355) -- (0.127713,0.047526)
    -- (0.160463,0.060464) -- (0.191826,0.073181) -- (0.223740,0.086410) -- (0.255785,0.099994)
    -- (0.287678,0.113825) -- (0.319616,0.127981) -- (0.351679,0.142497) -- (0.383397,0.157183)
    -- (0.415280,0.172248) -- (0.447165,0.187627) -- (0.478847,0.203235) -- (0.510498,0.219145)
    -- (0.542110,0.235355) -- (0.573654,0.251852) -- (0.605112,0.268630) -- (0.636472,0.285683)
    -- (0.667725,0.303009) -- (0.698864,0.320604) -- (0.729883,0.338467) -- (0.760777,0.356593)
    -- (0.791544,0.374983) -- (0.822190,0.393640) -- (0.852715,0.412564) -- (0.883091,0.431741)
    -- (0.913319,0.451172) -- (0.943424,0.470869) -- (0.973403,0.490831) -- (1.003190,0.511025)
    -- (1.032849,0.531482) -- (1.062335,0.552180) -- (1.091655,0.573121) -- (1.120795,0.594299)
    -- (1.149756,0.615713) -- (1.178510,0.637347) -- (1.207078,0.659213) -- (1.235426,0.681291)
    -- (1.263552,0.703577) -- (1.291456,0.726073) -- (1.319126,0.748769) -- (1.346550,0.771658)
    -- (1.373711,0.794729) -- (1.400608,0.817978) -- (1.427237,0.841403) -- (1.453592,0.864997)
    -- (1.479668,0.888755) -- (1.505459,0.912673) -- (1.530953,0.936741) -- (1.556147,0.960954)
    -- (1.581038,0.985308) -- (1.605618,1.009797) -- (1.629879,1.034412) -- (1.653805,1.059141)
    -- (1.677402,1.083985) -- (1.700650,1.108927) -- (1.723550,1.133965) -- (1.746084,1.159084)
    -- (1.768254,1.184282) -- (1.790042,1.209544) -- (1.811443,1.234861) -- (1.832450,1.260225)
    -- (1.853052,1.285625) -- (1.873698,1.311371) -- (1.894133,1.337286) -- (1.913841,1.363001)
    -- (1.932873,1.388544) -- (1.951273,1.413936) -- (1.969033,1.439165) -- (1.986218,1.464266)
    -- (2.002849,1.489250) -- (2.018905,1.514092) -- (2.034421,1.538811) -- (2.049413,1.563412)
    -- (2.063843,1.587859) -- (2.077756,1.612178) -- (2.091136,1.636351) -- (2.103965,1.660353)
    -- (2.116259,1.684192) -- (2.128006,1.707850) -- (2.139195,1.731308) -- (2.149811,1.754548)
    -- (2.159840,1.777548) -- (2.169269,1.800288) -- (2.178082,1.822746) -- (2.186255,1.844892)
    -- (2.193758,1.866689) -- (2.200548,1.888088) -- (2.206586,1.909044) -- (2.202312,1.921202)
    -- (2.194326,1.930149) -- (2.186346,1.939099) -- (2.178384,1.948065) -- (2.170448,1.957056)
    -- (2.162544,1.966076) -- (2.154675,1.975131) -- (2.146845,1.984226) -- (2.139055,1.993362)
    -- (2.131308,2.002545) -- (2.123607,2.011779) -- (2.115948,2.021062) -- (2.108338,2.030402)
    -- (2.100774,2.039797) -- (2.093256,2.049252) -- (2.085787,2.058768) -- (2.078366,2.068348)
    -- (2.070994,2.077995) -- (2.063670,2.087709) -- (2.056395,2.097494) -- (2.049169,2.107351)
    -- (2.041991,2.117283) -- (2.034863,2.127291) -- (2.027783,2.137379) -- (2.020752,2.147548)
    -- (2.013771,2.157801) -- (2.006839,2.168141) -- (1.999957,2.178570) -- (1.993125,2.189090)
    -- (1.986344,2.199706) -- (1.979611,2.210417) -- (1.972927,2.221226) -- (1.966295,2.232138)
    -- (1.959713,2.243157) -- (1.953184,2.254285) -- (1.946703,2.265522) -- (1.940274,2.276874)
    -- (1.933898,2.288345) -- (1.927575,2.299940) -- (1.921303,2.311657) -- (1.915085,2.323504)
    -- (1.908922,2.335486) -- (1.902813,2.347606) -- (1.896759,2.359865) -- (1.890762,2.372273)
    -- (1.884824,2.384834) -- (1.878941,2.397548) -- (1.873118,2.410424) -- (1.867357,2.423469)
    -- (1.861657,2.436687) -- (1.856019,2.450082) -- (1.850448,2.463664) -- (1.844944,2.477441)
    -- (1.839505,2.491412) -- (1.834137,2.505593) -- (1.828843,2.519991) -- (1.823624,2.534612)
    -- (1.818480,2.549463) -- (1.813416,2.564558) -- (1.808437,2.579908) -- (1.803544,2.595521)
    -- (1.798738,2.611405) -- (1.794027,2.627580) -- (1.789415,2.644057) -- (1.784905,2.660850)
    -- (1.780503,2.677974) -- (1.776212,2.695443) -- (1.766993,2.705551) -- (1.755391,2.712027)
    -- (1.743690,2.718337) -- (1.731885,2.724467) -- (1.719969,2.730403) -- (1.707932,2.736126)
    -- (1.695764,2.741613) -- (1.683452,2.746837) -- (1.670975,2.751757) -- (1.658302,2.756318)
    -- (1.638258,2.748469) -- (1.618433,2.740717) -- (1.598847,2.733101) -- (1.579510,2.725644)
    -- (1.560427,2.718365) -- (1.541599,2.711271) -- (1.523028,2.704375) -- (1.504708,2.697674)
    -- (1.486638,2.691174) -- (1.468811,2.684874) -- (1.451223,2.678770) -- (1.433870,2.672866)
    -- (1.416743,2.667153) -- (1.399838,2.661634) -- (1.383147,2.656296) -- (1.366665,2.651142)
    -- (1.350386,2.646167) -- (1.334304,2.641366) -- (1.318412,2.636734) -- (1.302705,2.632266)
    -- (1.287176,2.627959) -- (1.271821,2.623809) -- (1.256635,2.619812) -- (1.241611,2.615963)
    -- (1.226744,2.612255) -- (1.212030,2.608691) -- (1.197465,2.605263) -- (1.183042,2.601965)
    -- (1.168760,2.598803) -- (1.154609,2.595759) -- (1.140592,2.592845) -- (1.126699,2.590044)
    -- (1.112930,2.587365) -- (1.099279,2.584797) -- (1.085744,2.582339) -- (1.072321,2.579994)
    -- (1.059005,2.577749) -- (1.045795,2.575608) -- (1.032687,2.573571) -- (1.019678,2.571631)
    -- (1.006763,2.569785) -- (0.993942,2.568034) -- (0.981212,2.566377) -- (0.968570,2.564811)
    -- (0.956011,2.563328) -- (0.943534,2.561933) -- (0.931138,2.560623) -- (0.918820,2.559397)
    -- (0.906576,2.558253) -- (0.894406,2.557189) -- (0.882306,2.556203) -- (0.870274,2.555292)
    -- (0.858309,2.554457) -- (0.846408,2.553696) -- (0.834569,2.553008) -- (0.822791,2.552392)
    -- (0.811072,2.551847) -- (0.799410,2.551372) -- (0.787802,2.550965) -- (0.776248,2.550626)
    -- (0.764745,2.550353) -- (0.753293,2.550146) -- (0.741888,2.550004) -- (0.730530,2.549926)
    -- (0.719218,2.549911) -- (0.707949,2.549958) -- (0.696722,2.550067) -- (0.685535,2.550236)
    -- (0.674388,2.550466) -- (0.663279,2.550756) -- (0.652206,2.551105) -- (0.641168,2.551512)
    -- (0.630164,2.551978) -- (0.619192,2.552501) -- (0.608252,2.553081) -- (0.597341,2.553717)
    -- (0.586459,2.554410) -- (0.575604,2.555157) -- (0.564775,2.555959) -- (0.553970,2.556815)
    -- (0.543190,2.557727) -- (0.532432,2.558692) -- (0.521696,2.559712) -- (0.510980,2.560785)
    -- (0.500283,2.561913) -- (0.489604,2.563093) -- (0.478943,2.564327) -- (0.468297,2.565614)
    -- (0.457666,2.566953) -- (0.447048,2.568345) -- (0.436444,2.569790) -- (0.425851,2.571286)
    -- (0.415268,2.572833) -- (0.404694,2.574432) -- (0.394129,2.576083) -- (0.383572,2.577787)
    -- (0.373021,2.579543) -- (0.362476,2.581350) -- (0.351935,2.583211) -- (0.341397,2.585123)
    -- (0.330862,2.587088) -- (0.320328,2.589105) -- (0.309794,2.591174) -- (0.299261,2.593297)
    -- (0.288725,2.595474) -- (0.278187,2.597704) -- (0.267646,2.599989) -- (0.257101,2.602332)
    -- (0.246550,2.604732) -- (0.235993,2.607189) -- (0.225429,2.609704) -- (0.214856,2.612277)
    -- (0.204274,2.614909) -- (0.193682,2.617599) -- (0.183078,2.620350) -- (0.172461,2.623157)
    -- (0.161831,2.626026) -- (0.151186,2.628956) -- (0.140526,2.631947) -- (0.129848,2.635002)
    -- (0.119153,2.638119) -- (0.108439,2.641300) -- (0.097704,2.644546) -- (0.086949,2.647857)
    -- (0.076171,2.651234) -- (0.065369,2.654678) -- (0.054543,2.658190) -- (0.043691,2.661769)
    -- (0.032812,2.665419) -- (0.021904,2.669138) -- (0.010968,2.672929) -- (0.000000,2.676791)
    -- (-0.011000,2.680723) -- (-0.022032,2.684730) -- (-0.033100,2.688811) -- (-0.044203,2.692969)
    -- (-0.055344,2.697204) -- (-0.066523,2.701518) -- (-0.077742,2.705912) -- (-0.089002,2.710380)
    -- (-0.100305,2.714931) -- (-0.111652,2.719565) -- (-0.123045,2.724285) -- (-0.134484,2.729086)
    -- (-0.145973,2.733971) -- (-0.157511,2.738945) -- (-0.169102,2.744009) -- (-0.180745,2.749156)
    -- (-0.192444,2.754399) -- (-0.204198,2.759729) -- (-0.216011,2.765153) -- (-0.227884,2.770672)
    -- (-0.239818,2.776282) -- (-0.251816,2.781991) -- (-0.263878,2.787795) -- (-0.276007,2.793696)
    -- (-0.288205,2.799700) -- (-0.300472,2.805796) -- (-0.312812,2.811996) -- (-0.325225,2.818299)
    -- (-0.337714,2.824700) -- (-0.350281,2.831202) -- (-0.362926,2.837807) -- (-0.375653,2.844514)
    -- (-0.388462,2.851322) -- (-0.401356,2.858232) -- (-0.414335,2.865242) -- (-0.427403,2.872353)
    -- (-0.440560,2.879558) -- (-0.453807,2.886858) -- (-0.467147,2.894254) -- (-0.480579,2.901735)
    -- (-0.494106,2.909305) -- (-0.507727,2.916952) -- (-0.521444,2.924673) -- (-0.535257,2.932460)
    -- (-0.549165,2.940306) -- (-0.563167,2.948197) -- (-0.577263,2.956120) -- (-0.591449,2.964059)
    -- (-0.605724,2.971995) -- (-0.620081,2.979903) -- (-0.634514,2.987747) -- (-0.648282,2.992106)
    -- (-0.660525,2.989289) -- (-0.672733,2.986324) -- (-0.684910,2.983230) -- (-0.697055,2.980009)
    -- (-0.709170,2.976676) -- (-0.721256,2.973238) -- (-0.733313,2.969701) -- (-0.745342,2.966069)
    -- (-0.757343,2.962349) -- (-0.769317,2.958546) -- (-0.781264,2.954659) -- (-0.793184,2.950695)
    -- (-0.805078,2.946659) -- (-0.816944,2.942548) -- (-0.828785,2.938371) -- (-0.840600,2.934124)
    -- (-0.852390,2.929817) -- (-0.864152,2.925442) -- (-0.875890,2.921011) -- (-0.887603,2.916519)
    -- (-0.899289,2.911969) -- (-0.910952,2.907367) -- (-0.922589,2.902708) -- (-0.934200,2.897994)
    -- (-0.945787,2.893231) -- (-0.957350,2.888419) -- (-0.968887,2.883554) -- (-0.980399,2.878640)
    -- (-0.991888,2.873681) -- (-1.003352,2.868676) -- (-1.014792,2.863626) -- (-1.026206,2.858526)
    -- (-1.037597,2.853384) -- (-1.048963,2.848199) -- (-1.060306,2.842972) -- (-1.071625,2.837704)
    -- (-1.082919,2.832394) -- (-1.094189,2.827040) -- (-1.105435,2.821646) -- (-1.116657,2.816214)
    -- (-1.127856,2.810743) -- (-1.139031,2.805233) -- (-1.150183,2.799686) -- (-1.161309,2.794097)
    -- (-1.172412,2.788471) -- (-1.183491,2.782808) -- (-1.194547,2.777108) -- (-1.205579,2.771373)
    -- (-1.216588,2.765602) -- (-1.227571,2.759791) -- (-1.238532,2.753944) -- (-1.249468,2.748063)
    -- (-1.260382,2.742148) -- (-1.271271,2.736195) -- (-1.282135,2.730206) -- (-1.292975,2.724183)
    -- (-1.303793,2.718126) -- (-1.314584,2.712030) -- (-1.325352,2.705900) -- (-1.336096,2.699737)
    -- (-1.346814,2.693535) -- (-1.357507,2.687299) -- (-1.368177,2.681028) -- (-1.378819,2.674719)
    -- (-1.389438,2.668377) -- (-1.400028,2.661995) -- (-1.410594,2.655578) -- (-1.421133,2.649123)
    -- (-1.431644,2.642630) -- (-1.442128,2.636100) -- (-1.452584,2.629529) -- (-1.463011,2.622918)
    -- (-1.473407,2.616266) -- (-1.483774,2.609571) -- (-1.494109,2.602834) -- (-1.504410,2.596050)
    -- (-1.514676,2.589216) -- (-1.524906,2.582334) -- (-1.535152,2.575490) -- (-1.545424,2.568700)
    -- (-1.555673,2.561879) -- (-1.565898,2.555027) -- (-1.576095,2.548139) -- (-1.586264,2.541212)
    -- (-1.596407,2.534252) -- (-1.606516,2.527247) -- (-1.616598,2.520205) -- (-1.619633,2.502285)
    -- (-1.604255,2.456373) -- (-1.591817,2.415619) -- (-1.581554,2.378743) -- (-1.572984,2.344927)
    -- (-1.565758,2.313579) -- (-1.559633,2.284280) -- (-1.554422,2.256708) -- (-1.549983,2.230615)
    -- (-1.546211,2.205818) -- (-1.543017,2.182156) -- (-1.540324,2.159496) -- (-1.538084,2.137746)
    -- (-1.536242,2.116811) -- (-1.534756,2.096612) -- (-1.533594,2.077089) -- (-1.532727,2.058188)
    -- (-1.532122,2.039848) -- (-1.531764,2.022040) -- (-1.531633,2.004724) -- (-1.531706,1.987859)
    -- (-1.531969,1.971416) -- (-1.532411,1.955372) -- (-1.533019,1.939702) -- (-1.533782,1.924383)
    -- (-1.534690,1.909396) -- (-1.535734,1.894720) -- (-1.536906,1.880340) -- (-1.538198,1.866241)
    -- (-1.539603,1.852406) -- (-1.541115,1.838825) -- (-1.542730,1.825483) -- (-1.544441,1.812371)
    -- (-1.546244,1.799478) -- (-1.548135,1.786794) -- (-1.550110,1.774310) -- (-1.552165,1.762018)
    -- (-1.554296,1.749909) -- (-1.556502,1.737977) -- (-1.558775,1.726210) -- (-1.561115,1.714605)
    -- (-1.563520,1.703155) -- (-1.565988,1.691856) -- (-1.568517,1.680702) -- (-1.571105,1.669687)
    -- (-1.573750,1.658807) -- (-1.576447,1.648054) -- (-1.579195,1.637422) -- (-1.581994,1.626912)
    -- (-1.584844,1.616518) -- (-1.587742,1.606237) -- (-1.590688,1.596065) -- (-1.593678,1.585997)
    -- (-1.596710,1.576026) -- (-1.599785,1.566154) -- (-1.602903,1.556378) -- (-1.606063,1.546694)
    -- (-1.609264,1.537099) -- (-1.612502,1.527588) -- (-1.615776,1.518159) -- (-1.619088,1.508811)
    -- (-1.622438,1.499541) -- (-1.625825,1.490348) -- (-1.629248,1.481230) -- (-1.632701,1.472178)
    -- (-1.636188,1.463195) -- (-1.639710,1.454280) -- (-1.643265,1.445430) -- (-1.646853,1.436644)
    -- (-1.650474,1.427919) -- (-1.654121,1.419249) -- (-1.657799,1.410636) -- (-1.661509,1.402080)
    -- (-1.665249,1.393578) -- (-1.669020,1.385129) -- (-1.671401,1.375563) -- (-1.667786,1.361148)
    -- (-1.663439,1.346267) -- (-1.658334,1.330911) -- (-1.652410,1.315045) -- (-1.645568,1.298607)
    -- (-1.637644,1.281485) -- (-1.628357,1.263479) -- (-1.617158,1.244193) -- (-1.602598,1.222555)
    -- (-1.584895,1.198795) -- (-1.570113,1.177517) -- (-1.557524,1.158121) -- (-1.546362,1.139998)
    -- (-1.535143,1.122034) -- (-1.523652,1.104073) -- (-1.511890,1.086117) -- (-1.499855,1.068168)
    -- (-1.487549,1.050231) -- (-1.474971,1.032308) -- (-1.462120,1.014402) -- (-1.448992,0.996511)
    -- (-1.435590,0.978642) -- (-1.421915,0.960798) -- (-1.407960,0.942977) -- (-1.393729,0.925185)
    -- (-1.379222,0.907424) -- (-1.364432,0.889692) -- (-1.349366,0.871998) -- (-1.334011,0.854337)
    -- (-1.318380,0.836717) -- (-1.302456,0.819134) -- (-1.286253,0.801597) -- (-1.269753,0.784099)
    -- (-1.252968,0.766650) -- (-1.235885,0.749246) -- (-1.218507,0.731890) -- (-1.200834,0.714586)
    -- (-1.182851,0.697329) -- (-1.164576,0.680132) -- (-1.146001,0.662992) -- (-1.127121,0.645912)
    -- (-1.107949,0.628900) -- (-1.088483,0.611959) -- (-1.068708,0.595083) -- (-1.048641,0.578285)
    -- (-1.028283,0.561569) -- (-1.007629,0.544934) -- (-0.986668,0.528377) -- (-0.965416,0.511910)
    -- (-0.943873,0.495535) -- (-0.922037,0.479256) -- (-0.899900,0.463070) -- (-0.877457,0.446977)
    -- (-0.854719,0.430986) -- (-0.831686,0.415100) -- (-0.808355,0.399321) -- (-0.784726,0.383651)
    -- (-0.760789,0.368088) -- (-0.736523,0.352625) -- (-0.711948,0.337275) -- (-0.687057,0.322039)
    -- (-0.661851,0.306918) -- (-0.636334,0.291919) -- (-0.610513,0.277047) -- (-0.584332,0.262280)
    -- (-0.557840,0.247644) -- (-0.531051,0.233148) -- (-0.503971,0.218797) -- (-0.476609,0.204597)
    -- (-0.448973,0.190556) -- (-0.420943,0.176624) -- (-0.392623,0.162850) -- (-0.364044,0.149249)
    -- (-0.335219,0.135827) -- (-0.306163,0.122594) -- (-0.276894,0.109558) -- (-0.247102,0.096599)
    -- (-0.217002,0.083808) -- (-0.186760,0.071249) -- (-0.156396,0.058931) -- (-0.125922,0.046859)
    -- (-0.095350,0.035037) -- (-0.064684,0.023468) -- (-0.033929,0.012152) -- (0.000000,0.000000) -- cycle;

  \draw[black,line width=1.15pt]
    (0.000000,0.000000)
    .. controls (0.261799,0.000000) and (0.530585,0.026479) .. (0.796553,0.078196)
    .. controls (1.062522,0.129914) and (1.325603,0.206846) .. (1.575000,0.305449)
    .. controls (1.824397,0.404052) and (2.060020,0.524283) .. (2.271320,0.660786)
    .. controls (2.870000,1.040000) and (3.400000,1.690000) .. (3.140000,2.250000)
    .. controls (2.910000,2.750000) and (2.690000,2.540000) .. (3.020000,3.020000)
    .. controls (2.560000,3.580000) and (1.920000,3.970000) .. (1.320000,3.790000)
    .. controls (1.020000,3.700000) and (0.850000,3.360000) .. (0.500000,3.440000)
    .. controls (0.150000,3.520000) and (-0.020000,3.980000) .. (-0.520000,4.020000)
    .. controls (-1.020000,4.060000) and (-1.760000,3.640000) .. (-2.100000,3.400000)
    .. controls (-2.570000,3.080000) and (-2.860000,2.770000) .. (-2.700000,2.300000)
    .. controls (-2.570000,1.920000) and (-2.180000,1.900000) .. (-2.400000,1.520000)
    .. controls (-2.460000,1.200000) and (-2.240000,0.840000) .. (-1.971320,0.660786)
    .. controls (-1.812380,0.524283) and (-1.629052,0.404052) .. (-1.425000,0.305449)
    .. controls (-1.220948,0.206846) and (-0.996150,0.129914) .. (-0.756361,0.078196)
    .. controls (-0.516572,0.026479) and (-0.261799,0.000000) .. (0.000000,0.000000) -- cycle;

  \draw[black,dash pattern=on 5pt off 4pt,line width=.65pt]
    (x) circle (2.64);
  \node at (-2.10,2.65) {$\Omega$};
  \node at (3.00,3.30) {$\Gamma$};
  \node at (.10,1.95) {$\mathcal N_a^\Omega(x)$};
  \node at (-3.17,-1.20) {$B(x,R)$};

  \filldraw[fill=black!28,draw=black,line width=.85pt]
    (z) circle (.33);
  \node at (1.18,1.28) {$D_z$};

  \draw[black,densely dotted,line width=.7pt] (x)--(z);
  \draw[<->,black,line width=.65pt,
        shorten <=2pt,shorten >=2pt]
    (xi)--(z)
    node[pos=.30,right,fill=white,inner sep=1.5pt] {$t$};
  \node[point,label={below left:$x$}] at (x) {};
  \node[point] at (z) {};
  \node at ($(z)+(-.16,.10)$) {$z$};
  \node[point,label={below right:$\xi$}] at (xi) {};

  \draw[->,black,line width=.65pt]
    (x)--(-105:2.64)
    node[pos=.63,right,fill=white,inner sep=2pt]
      {$R=(a+1)t$};
\end{tikzpicture}
\caption{The non-tangential approach region with vertex $x\in\Gamma$.}
\end{figure}
The doubling property of $\mu$ therefore gives 
\begin{equation}\label{eq:nt-ball-comparison}
  \mu(B(x,R)\cap\Omega)
  \leq\mu(B(z,(2a+1)t))
  \leq C_a\mu(D_z),
\end{equation}
where $C_a$ depends only on $a$ and the doubling constant of $\mu$. 

The harmonic mean-value property on $D_z$, combined with the above estimates,  gives
\begin{align*}
  |U(z)-(\operatorname{Tr}_\Gamma U)(x)|
  &\leq\frac1{|D_z|}
    \int_{D_z}|U(\zeta)-(\operatorname{Tr}_\Gamma U)(x)|\,dA(\zeta)\\
  &\leq\frac{C_\alpha}{\mu(D_z)}
    \int_{D_z}|U(\zeta)-(\operatorname{Tr}_\Gamma U)(x)|\,d\mu(\zeta)\\
  &\leq\frac{C_{a,\alpha}}{\mu(B(x,R)\cap\Omega)}
    \int_{B(x,R)\cap\Omega}|U(\zeta)-(\operatorname{Tr}_\Gamma U)(x)|\,d\mu(\zeta). 
\end{align*}
As $z\to x$, we have $t\leq|z-x|\to0$, so $R\to 0$.
The last expression therefore tends to zero. Since the aperture $a>1$ was arbitrary,
this proves the non-tangential convergence. 

\end{proof}
 
\begin{rem}
Let $\Gamma$ be a rectifiable Jordan curve, and let $\gamma$ be
an arclength parametrization of  $\Gamma$. Then $\gamma$ is
differentiable almost everywhere with $|\gamma'|=1$. In particular, the arclength parametrization $\gamma$ has a non-zero derivative at $\mathcal H^1$-almost every
$x\in\Gamma$. At each such point,  the distance-based non-tangential approach regions
used above are locally comparable to truncated straight cones
with vertex $x$ and axis normal to the tangent line at $x$.
More precisely, each such approach region contains a narrower
truncated cone and is contained in a wider one. The distance-based
and classical conical definitions therefore give equivalent
notions of non-tangential convergence and yield the same
boundary trace $\mathcal H^1$-almost everywhere on $\Gamma$.
\end{rem}
\end{appendices}


\begin{thebibliography}{99}

\bibitem{AxlerBourdonRamey}
S. Axler, P. Bourdon, and W. Ramey,
\emph{Harmonic Function Theory}, 2nd ed.,
Graduate Texts in Mathematics, vol.~137,
Springer, New York, 2001.

\bibitem{Ahl}
Ahlfors, L.V.:
\newblock \emph{Lectures on Quasiconformal Mappings}.
\newblock Van Nostrand Mathematical Studies, vol.~10. D.~Van Nostrand, Princeton, NJ (1966).

\bibitem{Becker1972}
Becker, J.:
\newblock L\"ownersche Differentialgleichung und quasikonform fortsetzbare schlichte Funktionen.
\newblock \emph{J. Reine Angew. Math.} 255, 23--43 (1972).

\bibitem{BIPS}
Binder, I., Ivrii, O., Prause, I., Smirnov, S.:
\newblock Brennan's conjecture and the universal integral means spectrum.
\newblock Preprint,
\href{https://arxiv.org/abs/2610.23853}{arXiv:2610.23853}
(2026).

\bibitem{BL}
Bergh, J., L\"ofstr\"om, J.:
\newblock \emph{Interpolation Spaces: An Introduction}.
\newblock Grundlehren der mathematischen Wissenschaften, vol.~223. Springer-Verlag, Berlin--Heidelberg--New York (1976).

\bibitem{Brez}
Brezis, H.:
\newblock \emph{Functional Analysis, Sobolev Spaces and Partial Differential Equations}.
\newblock Universitext. Springer, New York (2011).

\bibitem{Chua}
Chua, S.-K.:
\newblock Extension theorems on weighted Sobolev spaces.
\newblock \emph{Indiana Univ. Math. J.} 41(4), 1027--1076 (1992).

\bibitem{Chua1993}
Chua, S.-K.:
\newblock Weighted Sobolev inequalities on domains satisfying the chain condition.
\newblock \emph{Proc. Amer. Math. Soc.} 117(2), 449--457 (1993).

\bibitem{Chua-2}
Chua, S.-K.:
\newblock Extension theorems on weighted Sobolev spaces and some applications.
\newblock \emph{Canad. J. Math.} 58(3), 492--528 (2006).

\bibitem{CoF}
Coifman, R.R., Fefferman, C.:
\newblock Weighted norm inequalities for maximal functions and singular integrals.
\newblock \emph{Studia Math.} 51(3), 241--250 (1974).

\bibitem{CwE}
Cwikel, M., Einav, A.:
\newblock Interpolation of weighted Sobolev spaces.
\newblock \emph{J. Funct. Anal.} 277(7), 2381--2441 (2019).


\bibitem{Dou}
Douglas, J.:
\newblock Solution of the problem of Plateau.
\newblock \emph{Trans. Amer. Math. Soc.} 33(1), 263--321 (1931).

\bibitem{DT2002}
Driscoll, T.A., Trefethen, L.N.:
\newblock \emph{Schwarz--Christoffel Mapping}.
\newblock Cambridge Monographs on Applied and Computational Mathematics,
vol.~8. Cambridge University Press, Cambridge (2002).

\bibitem{Duren}
Duren, P.L.:
\newblock \emph{Theory of $H^p$ Spaces}.
\newblock Pure and Applied Mathematics, vol.~38. Academic Press, New York--London (1970).

\bibitem{dss}
Duren, P.L., Shapiro, H.S., Shields, A.L.:
\newblock Singular measures and domains not of Smirnov type.
\newblock \emph{Duke Math. J.} 33(2), 247--254 (1966).

\bibitem{Gar}
Garnett, J.B.:
\newblock \emph{Bounded Analytic Functions}.
\newblock Pure and Applied Mathematics, vol.~96. Academic Press, New York (1981).

\bibitem{GM}
Gehring, F.W., Martio, O.:
\newblock Quasidisks and the Hardy--Littlewood property.
\newblock \emph{Complex Variables Theory Appl.} 2(1), 67--78 (1983).

\bibitem{Gri}
Grisvard, P.:
\newblock \emph{Elliptic Problems in Nonsmooth Domains}.
\newblock Monographs and Studies in Mathematics, vol.~24. Pitman, Boston (1985); reprinted in Classics in Applied Mathematics, vol.~69, SIAM, Philadelphia, PA (2011).

\bibitem{HaP}
Haroske, D.D., Piotrowska, I.:
\newblock Atomic decompositions of function spaces with Muckenhoupt weights, and some relation to fractal analysis.
\newblock \emph{Math. Nachr.} 281(10), 1476--1494 (2008).

\bibitem{Hin}
Hinkkanen, A.:
\newblock Modulus of continuity of harmonic functions.
\newblock \emph{J. Anal. Math.} 51(1), 1--29 (1988).

\bibitem{HMW}
Hunt, R.A., Muckenhoupt, B., Wheeden, R.L.:
\newblock Weighted norm inequalities for the conjugate function and Hilbert transform.
\newblock \emph{Trans. Amer. Math. Soc.} 176, 227--251 (1973).

\bibitem{JK95}
\revision{Jerison, D., Kenig, C.E.:
\newblock The inhomogeneous Dirichlet problem in Lipschitz domains.
\newblock \emph{J. Funct. Anal.} 130(1), 161--219 (1995).}

\bibitem{Jones}
Jones, P.W.:
\newblock Quasiconformal mappings and extendability of functions in Sobolev spaces.
\newblock \emph{Acta Math.} 147, 71--88 (1981).

\bibitem{JoZi}
Jones, P.W., Zinsmeister, M.:
\newblock Sur la transformation conforme des domaines de Lavrentiev.
\newblock \emph{C. R. Acad. Sci. Paris S\'er. I Math.} 295(10), 563--566 (1982).

\bibitem{JW}
Jonsson, A., Wallin, H.:
\newblock \emph{Function Spaces on Subsets of $\mathbb R^n$}.
\newblock Mathematical Reports, vol.~2, pt.~1.
Harwood Academic Publishers, Chur (1984).

\bibitem{kah}
Kahane, J.-P.:
\newblock Trois notes sur les ensembles parfaits lin\'eaires.
\newblock \emph{Enseign. Math. (2)} 15, 185--192 (1969).

\bibitem{Kil94}
Kilpel\"ainen, T.:
\newblock Weighted Sobolev spaces and capacity.
\newblock \emph{Ann. Acad. Sci. Fenn. Ser. A I Math.} 19(1), 95--113 (1994).

\bibitem{lav}
Lavrentiev, M.A.:
\newblock Boundary problems in the theory of univalent functions.
\newblock \emph{Amer. Math. Soc. Transl. (2)} 32, 1--35 (1963);
Russian original: \emph{Mat. Sb.} 1(43), no.~6, 815--846 (1936).

\bibitem{Lehto-Virtanen}
Lehto, O., Virtanen, K.I.:
\newblock \emph{Quasiconformal Mappings in the Plane}, 2nd edn.
\newblock Grundlehren der mathematischen Wissenschaften, vol.~126. Springer-Verlag, Berlin--Heidelberg--New York (1973).

\bibitem{LS}
Liu, T., Shen, Y.:
\newblock The jump problem for the critical Besov space.
\newblock \emph{Math. Z.} 306(4), article~59 (2024).

\bibitem{MR}
Mironescu, P., Russ, E.:
\newblock Traces of weighted Sobolev spaces. Old and new.
\newblock \emph{Nonlinear Anal.} 119, 354--381 (2015).

\bibitem{Muc72}
Muckenhoupt, B.:
\newblock Weighted norm inequalities for the Hardy maximal function.
\newblock \emph{Trans. Amer. Math. Soc.} 165, 207--226 (1972).

\bibitem{Pom}
Pommerenke, C.:
\emph{Boundary Behaviour of Conformal Maps}.
Grundlehren der mathematischen Wissenschaften, vol.~299.
Springer-Verlag, Berlin (1992).

\bibitem{ST87}
Schmeisser, H.-J., Triebel, H.:
\newblock \emph{Topics in Fourier Analysis and Function Spaces}.
\newblock A Wiley-Interscience Publication. John Wiley \& Sons, Chichester (1987).

\bibitem{Ste56}
Stein, E.M.:
\newblock Interpolation of linear operators.
\newblock \emph{Trans. Amer. Math. Soc.} 83(2), 482--492 (1956).

\bibitem{Stein1970}
Stein, E.M.:
\newblock \emph{Singular Integrals and Differentiability Properties of Functions}.
\newblock Princeton Mathematical Series, vol.~30. Princeton University Press, Princeton, NJ (1970).

\bibitem{SW58}
Stein, E.M., Weiss, G.:
\newblock Interpolation of operators with change of measures.
\newblock \emph{Trans. Amer. Math. Soc.} 87(1), 159--172 (1958).

\bibitem{Triebel78}
Triebel, H.:
\newblock \emph{Interpolation Theory, Function Spaces,
Differential Operators}.
\newblock North-Holland Mathematical Library, vol.~18.
North-Holland, Amsterdam--New York--Oxford (1978).

\bibitem{Usp1961}
Uspenskii, S.V.:
\newblock Imbedding theorems for weighted classes.
\newblock \emph{Trudy Mat. Inst. Steklov.} 60, 282--303 (1961) (Russian);
English translation in \emph{Amer. Math. Soc. Transl. (2)} 87, 121--145 (1970).

\bibitem{WZ1}
Wei, H., Zinsmeister, M.:
\newblock Dirichlet spaces over chord-arc domains.
\newblock \emph{Math. Ann.} 391, 1045--1064 (2025).

\bibitem{WZ2}
Wei, H., Zinsmeister, M.:
\newblock $p$-Dirichlet spaces over chord-arc domains.
\newblock \emph{Proc. Amer. Math. Soc.} (to appear).
\newblock \href{https://arxiv.org/abs/2410.02183v3}
{arXiv:2410.02183v3}.

\bibitem{zinRegular}
Zinsmeister, M.:
\newblock Domaines r\'eguliers du plan.
\newblock \emph{Ann. Inst. Fourier (Grenoble)} 35(1), 49--55 (1985).

\bibitem{Zi1}
Zinsmeister, M.:
\newblock Probl\`emes de Dirichlet, Neumann, Calder\'on dans les quasidisques pour les classes holderiennes.
\newblock \emph{Rev. Mat. Iberoam.} 2(3), 319--332 (1986).

\end{thebibliography}
\end{document}